\documentclass[11pt]{amsart}

\theoremstyle{plain}
\newtheorem{theorem}{Theorem}
\numberwithin{theorem}{section}
\newtheorem{lemma}[theorem]{Lemma}
\newtheorem{proposition}[theorem]{Proposition}
\newtheorem{corollary}[theorem]{Corollary}
\newtheorem{conjecture}[theorem]{Conjecture}

\theoremstyle{definition}

\newtheorem{example}[theorem]{Example}

\newtheorem{question}[theorem]{Question}
\newtheorem{remark}[theorem]{Remark}

\newcommand{\C}{{\mathbb C}}
\newcommand{\R}{{\mathbb R}}
\newcommand{\Z}{{\mathbb Z}}
\newcommand{\Q}{{\mathbb Q}}
\renewcommand{\P}{{\mathbb P}}
\newcommand{\s}{{\mathbb S}}

              \newcommand{\M}{{\mathcal M}}

\begin{document}
\title{Small symplectic $4$-manifolds via contact gluing and some applications}
\author{Weimin Chen}
\thanks{}
\date{\today}
\maketitle

\begin{abstract}
We introduce a streamlined procedure for constructing small symplectic $4$-manifolds via contact gluing, based on a technique invented by David Gay around 2000. We give several applications of this procedure, which include results concerning embeddings of singular Lagrangian $\R\P^2$s, or embeddings of lens spaces as a hypersurface of contact type, in small rational surfaces such as 
$\C\P^2\#\overline{\C\P^2}$ and $\s^2\times \s^2$, as well as results on the uniqueness or classification of $\Q$-homology ball symplectic fillings. Further work on the classification of singular Lagrangian $\R\P^2$s is suggested. Moreover, our investigation on the $\s^1$-invariant contact structures suggests an interesting and fairly strong upper bound for the self-intersection of a rational unicuspidal curve with one Puiseux pair $(p,q)$ in any algebraic surface (the bound depends only on the values $p,q$), and for the symplectic version, we prove the existence of an ``optimal" symplectic rational unicuspidal curve in a rational $4$-manifold which realizes the upper bound for any given Puiseux pair $(p,q)$. Our results also suggest a revisit of the ``symplectic divisorial capping" problem first considered by Li and Mak. Further applications of the techniques developed in this paper hinge upon better understandings on the tightness and fillability criterions of $\s^1$-invariant contact structures as well as their (small) symplectic fillings. 
\end{abstract}
\tableofcontents

\section{Introduction and the main results}

A compact symplectic $4$-manifold $(W,\omega)$ is said to have a {\bf convex contact boundary} 
if there is an outward-pointing normal vector field $v$ along its boundary $M=\partial W$ such that
$L_v \omega =\omega$. In this case, there is a canonical contact structure $\xi:=\ker (i_v\omega)|_M$
on $M$, which is uniquely determined up to a contact isotopy, as the normal vector fields $v$, though not unique, form a convex, thus path-connected, space. If in the above definition, the normal vector field $v$ is inward-pointing, then the compact symplectic $4$-manifold $(W,\omega)$ is said to have a
{\bf concave contact boundary}. One continues to have a canonically defined contact structure 
$\xi:=\ker (i_v\omega)|_M$ which is unique up to a contact isotopy, however, a crucial difference in the concave case is that the orientation of $M=\partial W$ as the boundary of $W$ is the opposite of the 
canonical orientation of the contact manifold $(M,\xi)$. Finally, in either the convex or the concave case, it is well-known that the symplectic structure $\omega$ near the boundary of $W$ is completely determined by the 
contactomorphism class of the contact manifold $(M,\xi)$. Consequently, if $(W,\omega)$ is a
compact symplectic $4$-manifold with a convex contact boundary $(M,\xi)$ and 
$(W^\prime,\omega^\prime)$ is a compact symplectic $4$-manifold with a concave contact 
boundary $(M^\prime,\xi^\prime)$, and moreover, there is a contactomorphism $\phi:(M,\xi)
\rightarrow (M^\prime,\xi^\prime)$ (i.e., $\phi_\ast(\xi)=\xi^\prime$), then one can form a closed 
symplectic $4$-manifold $X:=W\cup_\phi W^\prime$ by gluing symplectically $(W,\omega)$
and $(W^\prime,\omega^\prime)$ along their contact boundaries via $\phi:(M,\xi)
\rightarrow (M^\prime,\xi^\prime)$. This procedure is called {\bf contact gluing}, and our goal is to construct symplectic $4$-manifolds with a small topology (e.g. small Betti numbers), preferably even with an exotic smooth structure, via contact gluing (cf. \cite{C}). 

The $4$-manifolds $(W,\omega)$ with convex contact boundary come in natural examples, such
as Stein surfaces, while the $4$-manifolds $(W^\prime,\omega^\prime)$ with concave contact boundary do not, so they need to be constructed. Historically and in the literature, the $4$-manifold 
$(W^\prime,\omega^\prime)$ with concave contact boundary $(M^\prime,\xi^\prime)$ is often called a {\bf symplectic cap} of the contact manifold $(M^\prime,\xi^\prime)$, while $(W,\omega)$ is called a (strong) {\bf symplectic filling} of the convex contact boundary $(M,\xi)$. The cap served as a technical device which allows one to embed a fillable contact $3$-manifold into a closed symplectic $4$-manifold as a hypersurface of contact type, so that the $4$-manifold theory tools such as Seiberg-Witten theory or Gromov-Taubes theory can be used to study the contact $3$-manifolds
(see \cite{PM0, C00} for some earlier applications of this idea in the Stein fillable case). The most general existence theorem for symplectic caps was due to Eliashberg \cite{Eli} and 
Etnyre \cite{Etn} independently, which at the time was the missing ingredient in the solution of Property P Conjecture in knot theory \cite{KM}. We should point out that the symplectic caps constructed in these contexts often have a very large, and uncontrollable, topology thus are unsuitable for the purpose of constructing small symplectic $4$-manifolds with interesting geometric or topological properties. 

Around 2000, Gay \cite{G} invented a technique which has the effect of ``turning a convex
contact boundary into a concave contact boundary". More concretely and phrasing it slightly differently, what Gay introduced in \cite{G} roughly goes as follows. Suppose $(W,\omega)$ is a compact symplectic $4$-manifold with a convex contact boundary $(M,\xi)$. Assume $(B,\pi)$ is a rational open book on $M$ (cf. \cite{BEV}) such that the contact structure $\xi$ is supported by $(B,\pi)$. Then if we attach a symplectic $2$-handle to $(W,\omega)$ along each of the binding components $B_i$ of the rational open book $(B,\pi)$, with a framing $F_i>0$ relative to the page framing determined by 
$(B,\pi)$, we obtain a compact symplectic $4$-manifold $(W^\prime,\omega^\prime)$ with a concave contact boundary. It is clear that in this construction, the ``size" of the topology of 
$(W^\prime,\omega^\prime)$, which has a concave contact boundary, can be easily controlled by limiting the size of the $4$-manifold $(W,\omega)$ and the number of binding components of the rational open book $(B,\pi)$. 

There are however a number of obvious obstacles which prevented Gay's technique from becoming an effective tool in constructing closed symplectic $4$-manifolds with interesting geometric or topological significance. In this paper we introduce some ideas which will help remove these obstacles. See also recent work \cite{EMPR}.

More concretely, we shall restrict to the class of closed, oriented $3$-manifolds which are equipped with a smooth $\s^1$-action, typically a Seifert fibered $3$-manifold. It turns out that on such a $3$-manifold, rational open books which have a periodic monodromy and are compatible with the $\s^1$-action exist in abundance, and can be very easily constructed via a simple procedure. Furthermore, the contact structures supported by such a rational open book belong to a more restricted class, i.e., those which are invariant under the $\s^1$-action on the $3$-manifold. It turns out that this extra symmetry property of the contact structures allows for a very easy identification of the $\s^1$-equivariant contactomorphism class of the $\s^1$-invariant contact structure --- simply look at the location where the $\s^1$-invariant contact structure fails to be transversal to the orbits of the 
$\s^1$-action, and this in turn gives rise to much easier conditions for the tightness of the contact structures. As a result, these advantages, from working with a more restricted class of objects, allow us to establish a streamlined procedure for constructing symplectic caps with very small topology, thus a streamlined procedure for constructing small, closed symplectic $4$-manifolds via contact gluing. 

The technical materials concerning $\s^1$-invariant contact structures, rational open books with
periodic monodromy, as well as Gay's construction in \cite{G} applied to this context, are developed in Sections 2 through 5. We shall refer the readers to these sections for a summary of results or more detailed discussions. 
Furthermore, the table of contents should give an indication as how the paper is structured. In the remaining part of this introduction, we shall describe several applications of this streamlined procedure, stating several theorems, corollaries, and even questions and conjectures, which have interesting geometric or topological implications. 

\vspace{3mm}

Let $M_0=\s^3/Q_8$, where $\s^3$ is given with the non-standard orientation, and $Q_8$ is the subgroup of order $8$ of $SU(2)$ generated by the quaternions $i,j,k$, acting complex linearly 
on $\s^3$. As an oriented $3$-manifold, $M_0$ can be easily identified with the small Seifert 
space $M((2,1),(2,-1), (2,-1))$. It is well-known that $M_0$ is also a non-orientable circle bundle 
over the non-orientable surface $\R\P^2$ with Euler number $-2$. We let $W_0$ denote the associated 
disk bundle, which has a symplectic structure with convex contact boundary 
(e.g., it is known that $W_0$ is a Stein domain, cf. \cite{Go}). 

Let $V_0$ be the compact, oriented smooth $4$-manifold with boundary, which is a handlebody 
consisting of one $0$-handle, two $2$-handles attached along a pair of parallel copies of unknots 
with linking number $2$, with framings $4$ and $0$ respectively (relative to the zero-framing). 
Now we state our first theorem. 

\begin{theorem}
The $4$-manifold $V_0$ admits a symplectic structure with a concave contact boundary 
$(M_0,\xi_{inv})$, where $\xi_{inv}$ is a certain $\s^1$-invariant, tight contact structure on $M_0$. 
Furthermore, the following properties hold true:
\begin{itemize}
\item [{(1)}] The symplectic $4$-manifold $V_0$ contains a pair of embedded symplectic spheres 
$S_1,S_2$ of self-intersection $4$ and $0$, which intersects at a single point with a tangency of order two.\footnote{We may assume that the pair of spheres $S_1,S_2$ satisfy the following technical condition: at the intersection point there exists a local holomorphic chart $(z_1,z_2)$ with respect to which the symplectic form is K\"{a}ler, such that $S_1,S_2$ are given by $z_2=0$ and $z_2=c z_1^2+\cdots$ respectively. Throughout this paper, this condition is always assumed when we speak of such a pair of symplectic spheres.}
\item [{(2)}] For any symplectic structure on $W_0$ with convex contact boundary 
$(M_0,\xi_{fil})$, where $\xi_{fil}$ is a fillable contact structure on $M_0$, there exist exactly two 
contactomorphisms up to a smooth isotopy, $\Phi_1,\Phi_2: (M_0,\xi_{inv})\rightarrow (M_0, \xi_{fil})$, 
such that the closed symplectic $4$-manifolds obtained via contact gluing, i.e., $V_0\cup_{\Phi_1} W_0$ and $V_0\cup_{\Phi_2} W_0$, are diffeomorphic to $\C\P^2\# \overline{\C\P^2}$ and $\s^2\times \s^2$ respectively. 
\end{itemize}
As a corollary, $W_0$ can be realized as a compact domain with convex contact boundary in both 
a symplectic $\C\P^2\# \overline{\C\P^2}$ and a symplectic $\s^2\times \s^2$. 
\end{theorem}

By a {\bf singular Lagrangian $\R\P^2$} in a symplectic $4$-manifold, we mean a topological embedding of $\R\P^2$ which is smooth and Lagrangian except at one point where the embedding is
singular. Furthermore, the following technical assumption will be imposed, i.e., a neighborhood of
the singular point in $\R\P^2$ is symplectically modeled by a cone in $(\R^4,\omega_{std})$ over 
a Legendrian unknot $L_u\subset (\s^3, \xi_{std})$ with $tb=-2$ and $rot=\pm 1$, or equivalently, by
an open Whitney umbrella. (For a comprehensive discussion on singular non-orientable Lagrangian
surfaces, we recommend the paper by Nemirovski and Siegel \cite{NS}.) These technical assumptions
about the singularity will be assumed throughout, however, for notational simplicity, we will simply
call it a singular Lagrangian $\R\P^2$. With this understood, the following fact should be noted:
given any singular Lagrangian $\R\P^2$, one can resolve the singularity so that an embedded 
Lagrangian Klein bottle is produced nearby (cf. \cite{NS}, Proposition 4.6).

With the preceding understood, it was shown (cf. \cite{NS}, Example 5.5) that $W_0$ admits a symplectic structure with convex contact boundary such that a singular Lagrangian $\R\P^2$ is contained in $W_0$. We note the following immediate corollary of Theorem 1.1.

\begin{corollary}
Both $\C\P^2\# \overline{\C\P^2}$ and $\s^2\times \s^2$ admit an embedding of a singular 
Lagrangian $\R\P^2$ (with respect to some symplectic structure on $\C\P^2\# \overline{\C\P^2}$ 
and $\s^2\times \s^2$), such that in its complement there is a pair of embedded symplectic spheres of self-intersection $4$ and $0$ respectively, intersecting at a single point with a tangency of order two. Moreover, in any sufficiently small neighborhood of the singular Lagrangian $\R\P^2$, there is an embedded Lagrangian Klein bottle which has the property that in its complement there is a pair of embedded symplectic spheres of self-intersection $4$ and $0$ which intersect at a single point with a tangency of order two.
\end{corollary}

\begin{remark}
Obstruction of embeddings of $W_0$ in a symplectic $4$-manifold as a compact domain with convex contact boundary, and the somewhat equivalent problem of embeddings of a singular Lagrangian 
$\R\P^2$, represents some global symplectic rigidity phenomenon that is unique in dimension $4$,
and is in contrast to the symplectic flexibility known in the higher dimensions. More concretely, it was shown in \cite{NS} that $W_0$ can be realized as a compact domain in $\C^2$ with a strictly pseudoconvex boundary, which however cannot be made rationally convex, and this result is in 
a direct contrast to the flexibility in higher dimensions as shown in \cite{CE}. Similarly, there is no embedding of a singular Lagrangian $\R\P^2$ in $\C^2$, however, a singular totally real $\R\P^2$ 
in $\C^2$ does exist (cf. \cite{NS}). We should point out that the nature of these non-existence 
results still remain somewhat mysterious, as the only proof, which is an indirect proof, relies on
a constraint for embedded Lagrangian Klein bottles (see \cite{She, Nem}). By a similar argument, there is no corresponding embedding of $W_0$ or a singular Lagrangian $\R\P^2$ in $\C\P^2$ (cf. \cite{C}). It follows easily that $\C\P^2\# \overline{\C\P^2}$ and $\s^2\times \s^2$ are the smallest possible symplectic rational $4$-manifolds which allow for an embedding of $W_0$ or a singular Lagrangian $\R\P^2$, and any such an embedding requires a construction of global nature, i.e., not simply by grafting an embedding in $\C^2$ using Darboux theorem, because no embeddings in 
$\C^2$ exist. With this understood, Theorem 1.1 and
Corollary 1.2 confirm the existence of such embeddings of $W_0$ and a singular Lagrangian $\R\P^2$ in $\C\P^2\# \overline{\C\P^2}$ and $\s^2\times \s^2$. It remains an interesting open problem to find an alternative interpretation for the symplectic rigidity phenomenon in dimension $4$, for example using ideas and techniques from Symplectic Field Theory \cite{HWZ} as we have been advocating in \cite{C}. 
\end{remark}

Corollary 1.2 naturally raises the following question. 

\begin{question}
(1) For any embedding of a singular Lagrangian $\R\P^2$ in $\C\P^2\# \overline{\C\P^2}$ or 
$\s^2\times \s^2$, is there always a pair of embedded symplectic spheres $S_1,S_2$ in its complement which have self-intersection $4$ and $0$ respectively and intersect at a single point with a tangency of order two? Furthermore, is the embedding of a singular Lagrangian $\R\P^2$ (fixing its $\Z_2$-homology) unique up to a 
Hamiltonian/smooth isotopy? 

(2) Does a Lagrangian Klein bottle in $\C\P^2\# \overline{\C\P^2}$ or $\s^2\times \s^2$ always have in its complement a pair of embedded symplectic spheres $S_1,S_2$ of self-intersection $4$ and $0$, which intersect at a single point with a tangency of order two? (Lagrangian Klein bottles 
are known to exist in $\C\P^2\# \overline{\C\P^2}$ and $\s^2\times \s^2$. For descriptions from various perspectives, see \cite{CU, BHL, M, Mik, Nem0, Sym}.)
\end{question}

We remark that the existence of such a pair of symplectic spheres $S_1,S_2$ in the complement of the non-orientable Lagrangian surface may be the first step in an attempt to classify the Lagrangian surface\footnote{The areas of the spheres are constrained by $Area(S_1)>2 Area(S_2)$, which always holds true in the case of $\C\P^2\# \overline{\C\P^2}$. However, for $\s^2\times \s^2$, $Area(S_1)>2 Area(S_2)$ holds true if and only if $Area(S_2)$ is less than twice of the area of the other $\s^2$-factor. We note that the latter condition is ensured by the presence of a Lagrangian Klein bottle in the complement, see \cite{AE}.} (cf. \cite{C3}). For analogous results in the case of Lagrangian spheres, see \cite{Hind, Evans, LW}. In particular, using Gromov's theory of pseudoholomorphic curves, one can show that such a pair of symplectic spheres $S_1,S_2$ in $\C\P^2\# \overline{\C\P^2}$ or $\s^2\times \s^2$ has a unique symplectic isotopy class (see Lemma 7.14). 

The uniqueness of such a pair of symplectic spheres up to isotopy implies the following uniqueness theorem on symplectic fillings as a corollary of Theorem 1.1: recall that $M_0$ has a unique tight contact structure, to be denoted by $\xi_{tight}$,  up to a contactomorphism (see \cite{GLS}). 

\begin{theorem}
Up to a diffeomorphism, $W_0$ is the unique $\Q$-homology ball symplectic filling of $(M_0,\xi_{tight})$.
\end{theorem} 

In the course of proving Theorem 1.1, we also came upon the following result as a byproduct of Theorem 1.1. 

\begin{proposition}
Let $\kappa:\R\P^2\rightarrow \R\P^2$ be the involution induced by the reflection $(x,y,z)\mapsto (x,-y,z)$ on the unit two-sphere $\s^2\subset \R^3$. Then 
\begin{itemize}
\item [{(1)}] $\kappa: \R\P^2\rightarrow \R\P^2$ can be lifted to an involution $\tau: W_0\rightarrow W_0$, such that
the restriction of $\tau$ to the boundary preserves the Seifert fibration on $M_0$, leaving one singular fiber invariant and switching the other two singular fibers. 
\item [{(2)}] $\tau: W_0\rightarrow W_0$ induces a nontrivial action on the set of $spin^c$-structures of $W_0$, i.e., switching the two $spin^c$-structures of $W_0$. 
\end{itemize}
\end{proposition}

We observe the following immediate corollaries of Proposition 1.6, whose proofs are left to the reader. 

\begin{corollary}
(1) $W_0$ admits two distinct Stein structures which are not path-connected through almost complex structures on $W_0$. More concretely, for any Stein structure $J$ on $W_0$, $\tau^\ast J$ and $J$ have distinct $spin^c$-structures, hence are not homotopic as almost complex structures on $W_0$. 

(2) Let $(X,J)$ be any $4$-dimensional almost complex manifold (not necessarily compact closed). Suppose there is a smooth embedding $\iota: \R\P^2\rightarrow X$ such that the Euler number of the normal bundle of $\iota(\R\P^2)$ in $X$ equals $-2$. Then $\iota: \R\P^2\rightarrow X$ and $\iota\circ \kappa: \R\P^2\rightarrow X$ are not smoothly isotopic. 
\end{corollary}

For an example of Corollary 1.7(2), there is a smooth embedding $\iota: \R\P^2\rightarrow \C^2$ whose normal bundle has Euler number $-2$. By Corollary 1.7(2),  $\iota: \R\P^2\rightarrow \C^2$ and $\iota\circ \kappa:  \R\P^2\rightarrow \C^2$ are not smoothly isotopic. We should point out that this statement is consistent with the following fact: the $4$-sphere $\s^4$ admits a decomposition $\s^4=W_0\cup_\psi -W_0$ for some nontrivial self-diffeomorphism $\psi: \partial W_0\rightarrow \partial W_0$ (e.g. see \cite{A}, p. 426). With this understood, if we denote by $i: W_0\rightarrow \s^4$ the embedding induced from the decomposition 
$\s^4=W_0\cup_\psi -W_0$, then $i\circ \tau: W_0\rightarrow \s^4$ cannot be extended to a self-diffeomorphism of $\s^4$, because $\psi\circ \tau\circ \psi^{-1}: \partial W_0\rightarrow \partial W_0$ does not extend to a continuous map from $W_0$ to itself. 

Now we continue with the discussions concerning Question 1.4. The following two theorems give some additional embeddings of $W_0$ as a compact domain with convex contact boundary in 
$\C\P^2\# \overline{\C\P^2}$ and $\s^2\times \s^2$. They give interesting, concrete examples to consider for the more general questions raised in Question 1.4.

\begin{theorem}
There exists a relative handlebody $Z_0$ of two symplectic $2$-handles which have two concave
contact boundaries $(M_0,\xi_{inv})\sqcup (M_0,\xi_{inv})$, such that for any two copies of $W_0$,
each equipped with a symplectic structure $\omega$, $\omega^\prime$ having a convex contact
boundary $(M_0,\xi_{fil})$ and $(M_0,\xi^\prime_{fil})$ respectively, there exist precisely two 
contactomorphisms up to a smooth isotopy, $\Psi_1,\Psi_2: (M_0,\xi_{inv})\sqcup (M_0,\xi_{inv})
\rightarrow (M_0,\xi_{fil})\sqcup (M_0,\xi^\prime_{fil})$, such that the closed symplectic $4$-manifolds
formed by gluing $Z_0$ to two disjoint copies of $W_0$ via $\Psi_1,\Psi_2$ are diffeomorphic to
$\C\P^2\# \overline{\C\P^2}$ and $\s^2\times \s^2$ respectively. 
\end{theorem}

As an immediate corollary, there is a disjoint embedding of two copies of singular Lagrangian 
$\R\P^2$ in $\C\P^2\# \overline{\C\P^2}$ and $\s^2\times \s^2$. 

For the next theorem, let $M_1$ be the small Seifert space $M((13,6), (3,2), (2,-1))$. It is known that $M_1$ is
the link of a weighted homogeneous singularity which admits a rational homology disk smoothing
(cf. \cite{BS}, Figure 1(f) with $q=1$). Let $\xi_{Mil}$ be the Milnor fillable contact structure on 
$M_1$ (cf. \cite{CNP}). Then as a consequence, the contact manifold $(M_1,\xi_{Mil})$ admits a symplectic filling $W$ which is a $\Q$-homology ball, i.e., $W$ is the corresponding Milnor fiber from the smoothing of the weighted homogeneous singularity, hence a Stein domain. According 
to \cite{W}, there is a unique smoothing component, so such a Stein filling $W$ is uniquely determined. 

The lens space $L(25,19)$ as the link of a cyclic quotient singularity also admits a rational homology disk smoothing
(cf. \cite{W0}). Let $W_{Mil}$ be the Milnor fiber of the smoothing, which is a $\Q$-homology ball symplectic filling of $(L(25,19),\xi_{Mil})$, where $\xi_{Mil}$ is the Milnor fillable contact structure on $L(25,19)$. 

\begin{theorem}
(1) There exists a relative handlebody $Z_1$ of two symplectic $2$-handles which have two concave
contact boundaries $(M_1,\xi_{Mil})$ and $(M_0,\xi_{inv})$, such that for any symplectic structure on $W_0$ 
with a convex contact boundary $(M_0,\xi_{fil})$, and for any $\Q$-homology ball symplectic filling $W$ of 
$(M_1,\xi_{Mil})$, the symplectic $4$-manifolds $W\cup Z_1\cup_{\Phi_1} W_0$ and 
$W\cup Z_1\cup_{\Phi_2} W_0$ are diffeomorphic to $\s^2\times \s^2$  and $\C\P^2\# \overline{\C\P^2}$ respectively, where $\Phi_1,\Phi_2$ are the contactomorphisms from Theorem 1.1.

(2) There exists a relative handlebody $Z_2$ of two symplectic $2$-handles which have two concave
contact boundaries $(L(25,19),\xi_{Mil})$ and $(M_0,\xi_{inv})$, such that for any symplectic structure 
on $W_0$ with a convex contact boundary $(M_0,\xi_{fil})$, the symplectic $4$-manifolds $W_{Mil}\cup Z_2\cup_{\Phi_1} W_0$ and $W_{Mil}\cup Z_2\cup_{\Phi_2} W_0$ are diffeomorphic to 
$\C\P^2\# \overline{\C\P^2}$ and $\s^2\times \s^2$ respectively, where $\Phi_1,\Phi_2$ are the 
contactomorphisms from Theorem 1.1.
\end{theorem}

We remark that it is known that $W_{Mil}$ is the unique $\Q$-homology ball symplectic filling of 
$(L(25,19),\xi_{Mil})$ (cf. \cite{ER}). However, it is unknown whether $(M_1,\xi_{Mil})$ has a unique $\Q$-homology ball symplectic filling up to a diffeomorphism. Note that if the singular Lagrangian $\R\P^2$'s in $\C\P^2\# \overline{\C\P^2}$ or $\s^2\times \s^2$ were to have a unique Hamiltonian or smooth isotopy class, it would suggest 
that $(M_1,\xi_{Mil})$ should have a unique $\Q$-homology ball symplectic filling up to a diffeomorphism.


In the next theorem, we reduce the diffeomorphism classification of $\Q$-homology ball symplectic 
fillings of $(M_1,\xi_{Mil})$ to a question concerning certain pairs of pseudoholomorphic rational curves in $\C\P^2$.

\begin{theorem}
There exists a one to one correspondence between $\Q$-homology ball symplectic fillings $W$ of 
$(M_1,\xi_{Mil})$ and pairs of $J$-holomorphic curves $C_1,C_2$ in $\C\P^2$, where $C_1$ is
a degree $5$ unicuspidal rational curve with singularity $x_1$ whose link is a $(2,13)$-torus knot,
and $C_2$ is a degree $3$ rational nodal curve with a nodal singularity $x_2$, such that 
$C_1\cap C_2=\{x_1,x_2\}$, with local intersection number at $x_1,x_2$ being $13$ and $2$
respectively. Moreover, the compatible almost complex structure $J$ is integrable near the singularities 
$x_1$ and $x_2$. With this understood, the relation between $W$ and the configuration 
$C_1\cup C_2\subset \C\P^2$ is as follows: we blow up $\C\P^2$ at $x_2$, 
and denote the proper transforms of $C_1,C_2$ by $\tilde{C}_1,\tilde{C}_2$. Then the complement of 
a regular neighborhood of $\tilde{C}_1\cup \tilde{C}_2$ is diffeomorphic to the $\Q$-homology ball symplectic 
filling $W$ of $(M_1,\xi_{Mil})$.
\end{theorem}

We observe the following immediate corollary of Theorem 1.10.

\begin{corollary}
{\em(1)} There exists a pair of $J$-holomorphic curves $C_1,C_2$ in $\C\P^2$, where $C_1$ is
a degree $5$ unicuspidal rational curve with singularity $x_1$ whose link is a $(2,13)$-torus knot,
and $C_2$ is a degree $3$ rational nodal curve with a nodal singularity $x_2$, such that 
$C_1\cap C_2=\{x_1,x_2\}$, with local intersection number at $x_1,x_2$ being $13$ and $2$
respectively, and the compatible almost complex structure $J$ is integrable near the singularities 
$x_1$ and $x_2$. 

{\em(2)} For any $\Q$-homology ball symplectic filling $W$ of $(M_1,\xi_{Mil})$, the image of the induced homomorphism $\pi_1(M_1)\rightarrow \pi_1(W)$ normally generates $\pi_1(W)$. 
\end{corollary}

We remark that Corollary 1.11(2) is automatically true if $W$ is Stein. On the other hand, in 
Corollary 1.11(1), the individual curves $C_1,C_2$ exist, and even exist algebraically (in fact, $C_1$ is the first member of an infinite family of rational unicuspidal curves with one Puiseux pair, 
see \cite{F} for a
complete list). However, a pair of $J$-holomorphic $C_1,C_2$ which intersect as specified in 
Theorem 1.10 is not known to exist before, and is still unknown if we insist $C_1,C_2$ be algebraic. Theorem 1.10 naturally calls for an investigation of the corresponding symplectic isotopy problem for the configuration
$C_1\cup C_2\subset \C\P^2$; in particular, whether any such a configuration of $J$-holomorphic
curves is isotopic to an algebraic configuration. The underlying contact geometry question is whether
the symplectic filling $W$ is always Stein, and whether $W$ is unique up to a diffeomorphism. We shall also
point out that should $(M_1,\xi_{Mil})$ admit non-diffeomorphic $\Q$-homology ball  symplectic fillings, Theorem 1.10 would imply that there is a Zariski pair for the symplectic configuration $C_1\cup C_2\subset \C\P^2$.

\vspace{2mm}

The next theorem is concerned with embeddings of an infinite family of lens spaces in $\C\P^2\# \overline{\C\P^2}$ 
or $\s^2\times \s^2$ as a hypersurface of contact type. For any $\delta\geq 4$, we let $M_\delta$ denote the lens space $L(\delta^2,-\delta-1)$. As the link of a cyclic quotient singularity, it is known that $M_\delta$ bounds a $\Q$-homology ball, which is a Stein domain from a rational homology disk smoothing of the singularity, see \cite{W0}. From the contact geometry's point of view, there is a canonical tight contact structure $\xi_{can}$ on $M_\delta$, which is characterized by being a universally tight contact structure, and $\xi_{can}$ is unique as 
an un-oriented tight contact structure (cf. \cite{H}). The Milnor fillable contact structure on $M_\delta$ 
is the same as $\xi_{can}$ up to a change of orientation. In conclusion, $(M_\delta,\xi_{can})$ admits a $\Q$-homology ball symplectic filling, denoted by $W_\delta$, which is unique up to a 
diffeomorphism (cf. \cite{ER}).  

For $\delta\geq 4$, let $M_\delta^\prime$ be the small Seifert space 
$M((2,1),(6\delta-17,6),(4\delta-3,\delta-1))$. It is the link of a weighted homogeneous singularity 
which admits a rational homology disk smoothing (cf. \cite{BS}, Figure 1(c) with $r=3$, 
$q=\delta-4$). Furthermore, the singularity has a unique smoothing component 
(cf. \cite{W}). Let $\xi_{Mil}$ be the Milnor fillable contact structure on $M_\delta^\prime$. Then 
$(M_\delta^\prime,\xi_{Mil})$ admits a $\Q$-homology ball symplectic filling. 

\begin{theorem}
For each $\delta\geq 4$, there is a relative handlebody $Z_\delta$ consisting of two symplectic $2$-handles, where $Z_\delta$ is simply connected, with two concave contact boundaries 
$(M_\delta,\xi_{can})$ and $(M_\delta^\prime,\xi_{Mil})$, such that for any $\Q$-homology ball symplectic filling $W_\delta^\prime$ of $(M_\delta^\prime,\xi_{Mil})$, the closed symplectic 
$4$-manifold $X_\delta:=W_\delta\cup Z_\delta\cup W_\delta^\prime$ is diffeomorphic to a Hirzebruch
surface. Furthermore, $X_\delta\cong \s^2\times \s^2$ if and only if $\delta$ is odd. 
\end{theorem}

\begin{corollary}
{\em(1)} The infinite family of lens spaces $L(\delta^2,-\delta-1)$, where $\delta\geq 5$ and is odd, 
can be embedded in $\s^2\times \s^2$ as a hypersurface of contact type. 

{\em(2)} For any $\Q$-homology ball symplectic filling $W_\delta^\prime$ of 
$(M_\delta^\prime,\xi_{Mil})$, the image of the induced homomorphism $\pi_1(M_\delta^\prime)\rightarrow \pi_1(W_\delta^\prime)$ normally generates $\pi_1(W_\delta^\prime)$.
\end{corollary}

Regarding Corollary 1.13(1), for embeddings of lens spaces as a hypersurface of contact type
in $\C\P^2$,  see \cite{V, EvS, EMPR}. On the other hand, it is unknown whether 
$(M_\delta^\prime,\xi_{Mil})$ has a unique $\Q$-homology ball symplectic filling, or whether such a filling $W_\delta^\prime$ must be diffeomorphic to a Stein domain.  

\vspace{2mm}

For the last example  of a small symplectic $4$-manifold via contact gluing, we give 
a construction of $\C\P^2$. To this end, let $C$ denote a degree $3$ cuspidal curve in $\C\P^2$,
and let $C^\prime$ denote a degree $8$ rational unicuspidal curve in $\C\P^2$ with one Puiseux pair 
$(3,22)$. Let $W_C,W_{C^\prime}$ denote the complement of a regular neighborhood of
$C,C^\prime$, which are Stein domains, and each of $W_C,W_{C^\prime}$ is a $\Q$-homology ball, with boundary $M_C,M_{C^\prime}$ where $M_C=M((3,-1),(3,1),(2,-1))$, 
$M_{C^\prime}=M((22,7),(3,2),(2,-1))$ are small Seifert spaces. Moreover, $M_{C^\prime}$ is in fact the link of a weighted homogeneous singularity which admits a rational homology disk smoothing 
(cf. \cite{BS}, Figure 1(f) with $q = 2$). With these understood, we have the following theorem.

\begin{theorem}
There is a relative handlebody $Z$ consisting of one symplectic $2$-handle, which is simply connected, with two concave contact boundaries $(M_C,\xi_{inv})$ and 
$(M_{C^\prime}, \xi_{inv}^\prime)$, where $\xi_{inv}, \xi_{inv}^\prime$ are $\s^1$-invariant, tight contact structures on $M_C$, $M_{C^\prime}$ respectively, symplectically filled by $W_C$ and $W_{C^\prime}$, such that the closed symplectic $4$-manifold 
$X:=W_C\cup Z\cup W_{C^\prime}$ is diffeomorphic to $\C\P^2$. Moreover, $\xi_{inv}^\prime$ is isotopic to the Milnor fillable contact structure on $M_{C^\prime}=M((22,7),(3,2),(2,-1))$. 
\end{theorem}

In other words, the $\Q$-homology ball Stein domains $W_C,W_{C^\prime}$ can be {\it disjointly} embedded in $\C\P^2$ as a compact domain with convex contact boundary, such that the complement of $W_C\cup W_{C^\prime}$ is a simply connected, relative handlebody of one symplectic $2$-handle. In fact, a more general statement holds true, where $W_C,W_{C^\prime}$ may be replaced by any $\Q$-homology ball symplectic fillings of $(M_C,\xi_{inv})$, $(M_{C^\prime}, \xi_{inv}^\prime)$. 

\vspace{2mm}

Characterization of tightness of a general $\s^1$-invariant, non-transverse contact structure (see Section 2) remains elusive. Currently, a sticking point is whether the assumptions in Theorem 2.3 for its constructive proof are also necessary to ensure the tightness of the $\s^1$-invariant contact structure constructed therein, see Question 2.6. In the context of rational unicuspidal curves, the contact geometry question leads to the following question in algebraic geometry, which may have an independent perspective, or even a solution. We formulate it as a conjecture. 

Given any pair of relatively prime integers $(p,q)$, where $1<p<q$, let $(p^\prime,q^\prime)$ be the pair of positive integers uniquely determined by the equation $pq^\prime+q p^\prime=pq+1$. 

\begin{conjecture}
Let $X$ be any algebraic surface, $C\subset X$ be a rational unicuspidal curve with one Puiseux pair 
$(p,q)$. Then the self-intersection $C\cdot C\leq m_{p,q}$, where
$$
m_{p,q}: = pq+\max (\lfloor\frac{p}{p-p^\prime}\rfloor,\lfloor\frac{q}{q-q^\prime}\rfloor).
$$
\end{conjecture}

\begin{remark}
(1) If $C$ is a non-singular rational curve, then we do not expect any upper bound on the self-intersection of $C$. However, if $C$ is singular, there should be a priori upper bound on the self-intersection of $C$ which depends on the singularities of $C$. See Hartshorne \cite{Hartsh} for a relevant discussion. With this understood, the bound $m_{p,q}$ in 
Conjecture 1.15 is for a special class of rational cuspidal curves, and is motivated by a specific question from contact geometry. Paul Hacking informed me that he has a proof of Conjecture 1.15 for the case $(p,q)=(2,3)$, which is purely an algebraic geometry proof. 

(2) Borrowing ideas from Hartshorne \cite{Hartsh}, one can show that in Conjecture 1.15, if 
$C\cdot C\geq pq-p-q$,  the algebraic surface $X$ must be rational. Hence one can assume $X$ is a rational surface in Conjecture 1.15 without loss of generality. In particular, note that if the bound is sharp for the curve $C$, i.e.,  $C\cdot C=m_{p,q}$, $X$ must be rational. 

(3) For $X=\C\P^2$, a complete list of rational unicuspidal curves with one Puiseux pair $(p,q)$ 
was obtained (cf. \cite{F}). A direct verification shows that Conjecture 1.15 is true for these curves, 
where the bound $m_{p,q}$ is even sharp in infinitely many cases, i.e., $(p,q)=(d-1,d)$ or $(d/2,2d-1)$ where $d$ is the degree of the curve. See Example 6.7. In particular, the bound $m_{p,q}$ is optimal for $q=p+1$ and $q=4p-1$. On the other hand, one can produce examples of curves in a rational surface which show that $m_{p,q}$ is also optimal for $p=2,3,4$ and $5$. See Examples 6.8 and 6.9. 

(4) We have checked Conjecture 1.15 against examples of rational cuspidal curves in $\C\P^2$ or in 
$\C\P^1\times \C\P^1$ (low degree curves as well as some infinite families of curves) which are not unicuspidal but contain a cuspidal singularity with one Puiseux pair. (See \cite{Moe} for a nice overview, with many examples and references to the literature.) Rational unicuspidal curves with one Puiseux pair can be produced from such examples, and we are not able to find any counterexample to Conjecture 1.15, but rather, we find many examples for which the bound $m_{p,q}$ in Conjecture 1.15 is sharp. 

(5) Let $C$ be a rational unicuspidal curve with one Puiseux pair. Blow up at the singularity of $C$ 
and let $\tilde{C}$ be the proper transform. Assume $\tilde{C}$ is singular, which is also a rational unicuspidal curve with one Puiseux pair. One can show that it is not possible to improve the upper bound in Conjecture 1.15 (but rather, one often relaxes it) by passing from $C$ to its proper transform $\tilde{C}$. See Lemma 6.10 for more details. 

(6) Golla and Starkston introduced an upper bound for the self-intersection of a symplectic singular rational curve in a symplectic $4$-manifold in terms of its singularities (cf. \cite{GS2}, Proposition 6.2). Specializing to the case of unicuspidal curves with one Puiseux pair $(p,q)$, one can show that the bound $m_{p,q}$ in Conjecture 1.15 is strictly smaller than the bound of Golla and Starkston, except for the cases where $(p,q)=(m-1,m)$ or $(p,q)=(m,km-1)$ for some $m\geq 3$ and $k\geq 2$, in which the two bounds coincide.  See Lemma 6.11 for more details.  
\end{remark}

In the next theorem, we will show that for any Puiseux pair $(p,q)$, there is a symplectic rational unicuspidal curve in a rational $4$-manifold which realizes the bound $m_{p,q}$ as its self-intersection. Along the way, it will become clear how Conjecture 1.15 is related to the aforementioned question in contact geometry. However, we shall cast it in a broader perspective so that some future research projects become natural if seen in this light (e.g. \cite{C4}).

In \cite{LM}, Li and Mak considered the problem of embedding a symplectic divisor (i.e., a union of embedded symplectic surfaces with normal crossing) into a closed symplectic $4$-manifold, motivated by the phenomenon of compactification by normal crossing divisors in complex geometry. They took the first step by showing that if the intersection matrix of the divisor is non-singular and not negative-definite, a regular neighborhood of the symplectic divisor, via symplectic plumbing, admits a symplectic structure with a concave contact boundary, where the contact structure depends only on the plumbing graph up to a contactomorphism. Obviously, the symplectic divisor can be embedded in a closed symplectic $4$-manifold if and only if the concave contact boundary is symplectically fillable,
in which case, the symplectic filling can be capped by a divisorial capping, which is the symplectic analog of compactification by normal crossing divisors in complex geometry. 

We shall put this in a slightly different formulation. Let $M$ be an oriented graph manifold. Consider the set of plumbing graphs $\Gamma$ of symplectic plumbings such that the intersection matrix $Q_\Gamma$ is non-singular and not negative-definite, and the manifold $-M$, i.e., $M$ with the opposite orientation, is the boundary of a plumbing according to $\Gamma$. To each such plumbing graph $\Gamma$, one can associate a contact structure to $M$, which we shall refer to as the  
{\bf Li-Mak contact structure of $M$ associated to $\Gamma$}. We remark that the graph $\Gamma$ needs not to be in a normal form as defined by Neumann \cite{N1}, but often for a given manifold $M$, there is a canonical graph or a natural one arising from a specific context. We should point out that the Li-Mak contact structures are somewhat in analogy with the Milnor fillable contact structures, however, unlike in the Milnor fillable case, the Li-Mak contact structures are poorly understood in general; it could be overtwisted, and we do not have much understanding as for which $3$-manifolds $M$, a Li-Mak contact structure of $M$ is fillable. 

With the preceding understood, let $p,q$ be a pair of relatively prime integers, where $1<p<q$, and $m>0$ a positive integer. We let $M_{p,q,m}=-S^3_m(T_{p,q})$, which is the $3$-manifold obtained by a $m$-surgery on the $(p,q)$-torus knot $T_{p,q}$, but given the opposite orientation. Let $U_{p,q,m}$ be a handlebody consisting of one $0$-handle and one $2$-handle, attached along $T_{p,q}$ with framing $m$. Then $U_{p,q,m}$ is the smooth model for a regular neighborhood of a rational unicuspidal curve with one Puiseux pair $(p,q)$ and self-intersection $m$, and $M_{p,q,m}=-\partial U_{p,q,m}$; in particular, $U_{p,q,m}$ contains an obvious such 
cuspidal curve. To relate it to Li-Mak's work in \cite{LM}, note that by a successive blowing up of $U_{p,q,m}$, denoting the resulting manifold by $\tilde{U}_{p,q,m}$, one can resolve the singularity of the cuspidal curve in $U_{p,q,m}$, giving rise to a smooth normal crossing divisor $D_{p,q,m}$ in $\tilde{U}_{p,q,m}$. We observe
 that $\tilde{U}_{p,q,m}$ is a regular neighborhood of $D_{p,q,m}$, and $M_{p,q,m}=-\partial \tilde{U}_{p,q,m}$. Furthermore, we may assume $D_{p,q,m}$ is minimal, so that it is uniquely determined. As $m>0$, the intersection matrix is not negative-definite. By Li-Mak \cite{LM}, $M_{p,q,m}$ is equipped with a Li-Mak contact structure 
 $\xi_{LM}$, such that if $(M_{p,q,m}, \xi_{LM})$ is symplectically fillable, then $D_{p,q,m}$ is embedded in a closed symplectic $4$-manifold as a symplectic divisor. We remark that one can always symplectically blow down a symplectic $D_{p,q,m}$, resulting a symplectic rational unicuspidal curve\footnote{With this terminology it also goes with the technical assumption that near the singularity the curve is analytic and the symplectic form is 
K\"{a}hler.} with one Puiseux pair $(p,q)$ and self-intersection $m$ (see \cite{C2}, Section 4). To sum up, $M_{p,q,m}$ is equipped with a Li-Mak contact structure $\xi_{LM}$, and a symplectic rational unicuspidal curve with one Puiseux pair $(p,q)$ and self-intersection $m>0$ can be embedded in a closed symplectic $4$-manifold if and only if $(M_{p,q,m}, \xi_{LM})$ is fillable. 

Now we state our theorem. But first, note that $M_{p,q,m}$ admits an $\s^1$-action which defines $M_{p,q,m}$
as a small Seifert space or a lens space if $m\neq pq$, and for $m=pq$, the $\s^1$-action has a fixed component, and $M_{p,q,m}$ is the connected sum of two lens spaces. 

\begin{theorem}
(1) For any $m>0$, $U_{p,q,m}$ has a symplectic structure with concave contact boundary, containing a symplectic rational unicuspidal curve with one Puiseux pair $(p,q)$ and  self-intersection $m$ in the interior of $U_{p,q,m}$,
such that $U_{p,q,m}$ is a regular neighborhood of the cuspidal curve. Moreover, the contact structure on the boundary $M_{p,q,m}$, denoted by $\xi_{inv}$, is $\s^1$-invariant, which is transverse when $m\leq pq$, and is non-transverse when $m>pq$.

(2) A compact $4$-manifold $W$ is a symplectic filling of $(M_{p,q,m}, \xi_{LM})$ if and only if it is a
symplectic filling of $(M_{p,q,m}, \xi_{inv})$ (for a different symplectic structure on $W$). In particular, 
$(M_{p,q,m}, \xi_{LM})$ is fillable if and only if $(M_{p,q,m}, \xi_{inv})$ is fillable. 

(3) For $m=m_{p,q}$, $(M_{p,q,m}, \xi_{inv})$ is fillable. Consequently, a symplectic rational unicuspidal curve with one Puiseux pair $(p,q)$ and  self-intersection $m=m_{p,q}$ can be embedded in a
closed symplectic $4$-manifold $X$, where $X$ is in fact a rational $4$-manifold. 
\end{theorem}

We remark that by blowing up one can always decrease the self-intersection of the curve, so part (3) of the theorem is also true for any self-intersection $m<m_{p,q}$. In particular, $(M_{p,q,m}, \xi_{LM})$ (as well as 
$(M_{p,q,m}, \xi_{inv})$) is fillable for any $0<m\leq m_{p,q}$. On the other hand, if 
$(M_{p,q,m}, \xi_{LM})$ is fillable for some $m>m_{p,q}$, i.e., there is a symplectic rational unicuspidal curve with one Puiseux pair $(p,q)$ and self-intersection $m>m_{p,q}$ which can be embedded in a closed symplectic $4$-manifold, then $(M_{p,q,m}, \xi_{inv})$ is fillable for the same $m>m_{p,q}$, which gives a negative answer to the aforementioned question in contact geometry, i.e., Question 2.6. 

\begin{remark}
As we pointed out earlier, the Li-Mak contact structure $\xi_{LM}$ is not well understood in general regarding tightness and fillability, but it has the advantage of being flexible in the sense that up to a contactomorphism, it does not depend on the geometric data of the symplectic divisor involved in the construction. On the other hand, our construction for a symplectic model of the regular neighborhood of a symplectic rational cuspidal curve, or more generally a union of rational curves, which will result in an $\s^1$-invariant contact structure $\xi_{inv}$ on the concave boundary of the regular neighborhood, is more sensitive of, and even constrained by, the geometric data involved in the construction (see Remark 6.5 and Lemma 8.3). However, the $\s^1$-invariant contact structure $\xi_{inv}$ has the advantage of being given in a more descriptive form, and having more definite results concerning its tightness and fillability. The strategy of Theorem 1.17 is to make connections between the two contact structures and exploit their distinct advantages. The same idea was also used in the proof of Theorem 1.10, and will be further explored in forthcoming projects (see e.g. \cite{C4}).
\end{remark}

Our work naturally suggests the following question.

\begin{question}
Let $M$ be Seifert fibered, and $\xi_{LM}$ be a Li-Mak contact structure of $M$ associated to
a graph $\Gamma$. Under what conditions of $\Gamma$ is there an $\s^1$-invariant contact structure $\xi_{inv}$ on $M$ which is equivalent to $\xi_{LM}$ in the following sense: there are topologically trivial symplectic cobordisms between $(M,\xi_{LM})$ and $(M,\xi_{inv})$ (in both directions)? 
\end{question}

Note that a positive answer to Question 1.19 will bring progress to the problem originally considered by Li and Mak in \cite{LM}, as $(M,\xi_{LM})$ and $(M,\xi_{inv})$ have identical set of symplectic fillings up to diffeomorphisms, so that as far as symplectic filling or capping is concerned, $\xi_{LM}$ may be replaced by the $\s^1$-invariant contact structure $\xi_{inv}$.

\vspace{2mm}

{\bf Update on Conjecture 1.15:} In \cite{Tev} Jenia Tevelev gave a proof,  among other results, of Conjecture 1.15, using the Log Minimal Model Program. In \cite{COS}, we showed that the bound 
$m_{p,q}$ is sharp, extending Theorem 1.17(3) to the algebraic category. We also obtained a precise relation between $m_{p,q}$ and the bound of Golla and Starkston. The sharpness of $m_{p,q}$ was
also independently proved by Tevelev in \cite{Tev}. 

\vspace{2mm}

{\bf Acknowledgements:} This paper is a continuation of the project initiated in \cite{C}. I wish to thank Selman Akbulut for his initial encouragement and continued interest in it. Thanks to both him and Eylem Yildiz for their hospitality during my visit to GGTI where a preliminary version of this work was first reported. I have benefited tremendously from the discussions with them. Many thanks to Paul Hacking for the helpful communications regarding Conjecture 1.15 as well as the relevant work of Hartshorne \cite{Hartsh}, and to Marco Golla for bringing to my attention his work with Starkston \cite{GS2} and other relevant works concerning bound on the self-intersection of rational cuspidal curves. Finally, this work was presented in a seminar during my visit to Center for Geometry and Physics (CGP) in Pohang, South Korea. I am grateful to Yong-Geun Oh for the invitation, as well as for the warm hospitality and the useful discussions with him during my stay at CGP. 

During the preparation of this work, no instruments, software, or any AI related tools were used. The 
work requires no data supporting its findings and analysis.

\section{$\s^1$-invariant contact structures and transverse surgery}

Let $M$ be a compact closed, oriented $3$-manifold, which is equipped with a smooth
$\s^1$-action $\varphi_t: M\rightarrow M$, $t\in\s^1$. We denote by $X$ the vector field 
generating the $\s^1$-action $\varphi_t$. In this section, we shall prove some basic results 
concerning positive, co-orientable contact structures $\xi$ on $M$ which are $\s^1$-invariant, 
i.e., $(\varphi_t)_\ast \xi=\xi$, $\forall t\in\s^1$. 

First of all, by a standard argument we observe that there is always an $\s^1$-invariant
contact form $\alpha$ such that $\xi=\ker \alpha$: simply take $\alpha:=\int_{\s^1} \varphi_s^\ast 
\tilde{\alpha} \; ds$ where $\tilde{\alpha}$ is any contact form such that $\xi=\ker\tilde{\alpha}$. 
Secondly, it is important to distinguish the two different cases of $\s^1$-invariant contact structures, 
i.e., transverse and non-transverse. An $\s^1$-invariant contact structure $\xi$ is called {\bf transverse} if in the complement of the fixed-point set of $\varphi_t$, $\xi$ is transversal to the orbits of the $\s^1$-action. Otherwise, $\xi$ is called {\bf non-transverse}.

When the $\s^1$-action $\varphi_t$ is fixed-point free, $M$ is Seifert fibered. In this case, $M$ admits
a transverse $\s^1$-invariant contact structure $\xi$ if and only if the Euler number of the Seifert fibration is negative, i.e., $e(M)<0$ (cf. \cite{PM}). Furthermore, assuming $e(M)<0$, then $M$ is the link of an isolated, weighted homogeneous singularity (cf. \cite{NR}), and the $\s^1$-invariant, transverse contact structure $\xi$ is simply the so-called Milnor fillable contact structure on $M$, which is unique up to a contact isotopy (cf. \cite{CNP}). 

For the general situation where the $\s^1$-action $\varphi_t$ may have a nonempty fixed-point set, which is a link $L$ in $M$, let's fix an $\s^1$-invariant contact form $\alpha$ such that 
$\xi=\ker\alpha$. In order to put $\alpha$ in a canonical form, we observe that the link $L$ is a transverse link. It follows easily that in some neighborhood $U_i$ of each component $L_i$, there are coordinates $(z_i,r_i,\theta_i)$, where $z_i$ parametrizes $L_i$ and $(r_i,\theta_i)$ are the polar coordinates in the fibers of the normal bundle of $L_i$, such that $\alpha=dz_i+r_i^2d\theta_i$. Moreover, the $\s^1$-action is given by translations in the direction of $\theta_i$ or $-\theta_i$. 

With the preceding understood, we note that the $\s^1$-action defines a Seifert fibration $\pi: M\setminus L\rightarrow S$, where $S$ is a punctured oriented $2$-orbifold whose punctures are in one to one correspondence with the components of $L$. Furthermore, we can choose an $\s^1$-invariant $1$-form $\alpha_0$ on $M\setminus L$ such that 
$\alpha_0(X)=1$, and $\alpha_0=d\theta_i$ or $-d\theta_i$ near $L_i$. With this understood, it follows easily that there is a smooth function $f$ and a $1$-form $\beta$ on $S$, where near the puncture of $S$ corresponding to $L_i$, $f=r^2_i$ or $-r_i^2$ and $\beta=dz_i$, such that 
$$
\alpha=f\alpha_0+\beta.
$$
Next we observe that $d\alpha_0=\pi^\ast \kappa$ for some $2$-form $\kappa$
on $S$. Consequently, 
$$
0<\alpha\wedge d\alpha =f\alpha_0\wedge (f\kappa+d\beta)+(\beta\wedge df)\wedge \alpha_0.
$$
We draw some immediate corollaries from the above inequality: for any $p\in f^{-1}(0)$, 
$\beta\wedge df (p)\neq 0$; in particular, $df(p)\neq 0$. It follows 
easily that $f^{-1}(0)$, if nonempty, is a union of circles in $S$ in the complement of the 
singular points. Finally, the $\s^1$-invariant contact structure $\xi$ is transverse if and only if $f^{-1}(0)$ is empty.

The above analysis has a number of immediate consequences for a non-transverse, $\s^1$-invariant  contact structure $\xi$. First, note that for any $p\in  f^{-1}(0)$, the orbit $K$ over $p$ is a Legendrian vertical circle in $(M,\xi)$, which is easily seen to have a zero twisting number. This implies, by a theorem of Wu (cf. \cite{Wu}, Theorem 1.4), that if $M$ is a small Seifert space with $e_0(M)\leq -2$, the contact structure $\xi$ must be overtwisted. On the other hand, assuming the $\s^1$-action is not fixed-point free, we can connect $p$ by an arc $\gamma$ to a puncture of $S$, such that the tangent vector of $\gamma$ at $p$ lies in $\ker\beta(p)$ (note that $\beta(p)\neq 0$). It is easy to see that the orbits over $\gamma$ form an embedded, overtwisted disk in $(M,\xi)$; in particular, $\xi$ is an overtwisted contact structure. We summarize the above discussions in the following lemma.

\begin{lemma} 
If $M$ is a small Seifert space with $e_0(M)\leq  -2$, or the $\s^1$-action on $M$ is not fixed-point free, then 
any $\s^1$-invariant, non-transverse, contact structure on $M$ must be overtwisted.
\end{lemma}

For the rest of this section, we shall focus on the case where the $\s^1$-action is fixed-point free, which defines a Seifert fibration $\pi: M\rightarrow S$ over a closed, oriented $2$-orbifold. In analogy with the convex surface theory, we introduce the following terminology. Let $\xi$ be an $\s^1$-invariant, non-transverse contact structure on $M$. Denote by $\Gamma:=\sqcup \Gamma_i $ the set of disjoint circles in $S$ over which the fibers of the Seifert fibration $\pi: M\rightarrow S$ are Legendrian. We shall call $\Gamma$ the {\bf dividing set} of $\xi$. The following lemma shows that the dividing set of an $\s^1$-invariant contact structure determines the contact structure up to an $\s^1$-equivariant contactomorphism (we should point out that here, the contact structures are not assumed to be oriented; for oriented contact structures, the determination is up to a change of orientation). 

\begin{lemma}
Let $\xi_1,\xi_2$ be two given $\s^1$-invariant contact structures on $M$, with dividing sets 
$\Gamma_1$, $\Gamma_2$ respectively. There is an $\s^1$-equivariant diffeomorphism 
$\psi: M\rightarrow M$ such that $\psi_\ast(\xi_1)=\xi_2$ if and only if there is an orientation-preserving diffeomorphism of orbifold $\phi: S\rightarrow S$, which preserves the normalized Seifert invariants at the singular points, such that $\phi(\Gamma_1)=\Gamma_2$. 
\end{lemma}

\begin{proof}
Suppose there is an $\s^1$-equivariant diffeomorphism $\psi: M\rightarrow M$ such that 
$\psi_\ast(\xi_1)=\xi_2$. First, since $\psi$ is a contactomorphism, it must be orientation-preserving.
On the other hand, being $\s^1$-equivariant, it induces a diffeomorphism of orbifold 
$\phi: S\rightarrow S$, which must also be orientation-preserving as well as preserving the normalized Seifert invariants at the singular points. Finally, it sends Legendrian fibers to Legendrian fibers, 
hence $\phi(\Gamma_1)=\Gamma_2$. 

For the converse, assume there is an orientation-preserving diffeomorphism of orbifold 
$\phi:S\rightarrow S$ which preserves the normalized Seifert invariants at the singular points, 
such that $\phi(\Gamma_1)=\Gamma_2$. We claim that $\phi$ lifts to an $\s^1$-equivariant diffeomorphism $\psi: M\rightarrow M$. To see this, simply note that the pull-back 
Seifert fibration $\phi^\ast \pi$ is isomorphic to $\pi$ itself, and we can fix an identification between 
$\phi^\ast \pi$ and $\pi$ by such an isomorphism. Now if we replace the contact structure $\xi_2$ by the pull-back via $\psi$, we may reduce the problem at hand to the special case where 
$\Gamma_1=\Gamma_2$, which will be denoted simply by $\Gamma=\sqcup_i \Gamma_i$. 
Furthermore, if $\alpha_1,\alpha_2$ are $\s^1$-invariant contact forms for $\xi_1,\xi_2$, and 
$\alpha_1=f_1\alpha_0+\beta_1$ and $\alpha_2=f_2\alpha_0+\beta_2$, then both $f_1,f_2$, which vanishes on each $\Gamma_i$, will change sign across each $\Gamma_i$. It follows easily that, by reversing the orientation of one of $\xi_1$ and $\xi_2$, we may assume without loss of generality that $f_1$ and 
$f_2$ have the same sign on $S\setminus \Gamma$. Finally, for each component $\Gamma_i$, set 
$T_i:=\pi^{-1}(\Gamma_i)$. 

With the preceding understood, we shall construct the $\s^1$-equivariant contactomorphism 
$\psi$, $\psi_\ast(\xi_1)=\xi_2$, near $T_i$ for each component $\Gamma_i$ first. To this end,
we fix a product structure of the Seifert fibration in a neighborhood $N_i$ of $T_i$, where
$N_i=(-\epsilon,\epsilon)\times \s^1\times \s^1$, with coordinates $(s,\lambda,\theta)$, where
the $\s^1$-action is given by translations in $\theta$ and $T_i$ is given by $s=0$. With this understood, any $\s^1$-invariant contact form on $N_i$ with respect to which the fibers in $T_i$
are Legendrian can be written as
$$
\alpha=f(s,\lambda)d\theta+ a(s,\lambda)ds+b(s,\lambda)d\lambda,
$$
where $f(0,\lambda)=0$ for any $\lambda$. The condition $\alpha\wedge d\alpha>0$ at $s=0$
is equivalent to $-\partial_s f(0,\lambda) b(0,\lambda)>0$. In particular, $b(s,\lambda)\neq 0$
for sufficiently small $|s|>0$. By restricting to a smaller neighborhood, still denoted by $N_i$
for simplicity, we can assume that $\alpha$ takes the following form
$$
\alpha=f(s,\lambda)d\theta+ a(s,\lambda)ds+d\lambda
$$
after replacing $\alpha$ by $b(s,\lambda)^{-1}\alpha$. Note that $\partial_s f(0,\lambda)<0$ for each $\lambda$.

Now let $\alpha_j=f_j(s,\lambda)d\theta+a_j(s,\lambda)+d\lambda$, for $j=1,2$, be a contact form
defining $\xi_j$ on $N_i$. Set $\tilde{\alpha}_t:=(1-t)\alpha_1+t\alpha_2$, $t\in [0,1]$. It is easy to
see that in a small neighborhood of $T_i$, still denoted by $N_i$ for simplicity, each 
$\tilde{\alpha}_t$ is a contact form, with $\tilde{\alpha}_0=\alpha_1$, $\tilde{\alpha}_1=\alpha_2$. 
We note that $\frac{d}{dt}\tilde{\alpha}_t=\alpha_2-\alpha_1=(a_2(0,\lambda)-a_1(0,\lambda))ds$
at $s=0$, i.e., along $T_i$. 

Let $R_{\tilde{\alpha}_t}$ be the Reeb vector field of $\tilde{\alpha}_t$. We define a function $H$
on $N_i$ by
$$
H(s,\lambda,\theta)=\int_{0}^s (a_1(u,\lambda)-a_2(u,\lambda))du.
$$
Then $H=0$ at $s=0$ and $\frac{d}{dt}\tilde{\alpha}_t+dH=0$ at $s=0$. On the other hand, we define
on $N_i$ a function $g_t$ by
$$
g_t=(\frac{d}{dt}\tilde{\alpha}_t+dH)(R_{\tilde{\alpha}_t}), 
$$
and let $Y_t\in \ker\tilde{\alpha}_t$ be the solution of the equation
$$
i_{Y_t} d\tilde{\alpha}_t= g_t\tilde{\alpha}_t - (\frac{d}{dt}\tilde{\alpha}_t+dH). 
$$
Note that $Y_t=0$ along $T_i$ (i.e., at $s=0$). With this understood, we let 
$$X_t=HR_{\tilde{\alpha}_t}+Y_t.$$
We note that $X_t=0$ along $T_i$ as well. Let $\psi_t$, 
$t\in [0,1]$, be the smooth family of open embeddings defined in a neighborhood of $T_i$
which is generated by $X_t$, with $\psi_0=Id$. Then $\psi^\ast_t \tilde{\alpha}_t=f_t\tilde{\alpha}_0$,
$\forall t$. If we let $\psi=\psi_1$, then $\psi_\ast (\xi_1)=\xi_2$ in a neighborhood of $T_i$.
It is clear that $X_t$ is $\s^1$-invariant, hence $\psi$ is $\s^1$-equivariant. Note that
$\psi=Id$ on $T_i$.

We extend $\psi$ from a neighborhood of each $T_i$ to an orientation-preserving, $\s^1$-equivariant diffeomorphism of $M$, still denoted by $\psi: M\rightarrow M$ for simplicity. With this, we can further reduce the problem to the case where we assume $\xi_1=\xi_2$ near each $T_i$, by replacing
$\xi_2$ with the pull-back of $\xi_2$ by $\psi$. With this understood, let $S_j$ be any
connected component of $S\setminus \Gamma$. Then we note that $\xi_1,\xi_2$ are
both transverse to the fibers of the $\s^1$-action over $S_j$, hence there are 
$\s^1$-invariant contact forms $\alpha_1$, $\alpha_2$ on $\pi^{-1}(S_j)$, defining $\xi_1$,
$\xi_2$ respectively, such that $\alpha_1(X)=\alpha_2(X)=1$, where $X$ is the vector field generating the $\s^1$-action. We note that $\alpha_2-\alpha_1$ has a compact support as $\xi_1=\xi_2$ near
each $T_i$. With this understood, by the argument of Lemma 2.6 in \cite{C1}, there is an $\s^1$-equivariant diffeomorphism $\psi: \pi^{-1}(S_j)\rightarrow \pi^{-1}(S_j)$, which is identity outside the
support of $\alpha_2-\alpha_1$, such that $\psi^\ast \alpha_2=\alpha_1$. Putting the pieces together,
we obtain an $\s^1$-equivariant diffeomorphism $\psi: M\rightarrow M$ 
such that $\psi_\ast(\xi_1)=\xi_2$. 

\end{proof}

We end this section with a theorem establishing the existence of a tight $\s^1$-invariant, non-transverse contact structure on a lens space or a small Seifert space, which has a certain 
prescribed dividing set. Furthermore, as a consequence of Lemma 2.2, the existence theorem implies also the tightness of any $\s^1$-invariant contact structures on a lens space or a small Seifert space with such a dividing set. 

Given an $\s^1$-invariant, non-transverse contact structure $\xi$ on a Seifert fibered manifold $M$,
we note that any fiber (regular or singular) which is not a Legendrian vertical circle (i.e., not above
the dividing set $\Gamma$) is a transverse knot. It follows easily that if we perform a transverse 
surgery on such a transverse knot, the $\s^1$-invariance of the contact structure remains unchanged. 
Transverse surgery comes with two flavors: admissible transverse surgery and inadmissible 
transverse surgery (see \cite{BE} or \cite{Con} for a nice introduction on transverse surgery). 
As we will see, an admissible transverse surgery with sufficiently negative surgery coefficient will
preserve the tightness of the contact structure while leaving the dividing set of the contact structure
untouched. On the other hand, an inadmissible transverse surgery may not preserve the tightness
of the $\s^1$-invariant contact structure, although it can leave the dividing set untouched. 
With this understood, the tight, $\s^1$-invariant contact structures in the following theorem are obtained through admissible transverse surgeries with sufficiently negative surgery coefficient (in fact, the contact structures are fillable as well, because the admissible transverse surgeries preserve the fillability). 

\begin{theorem}
{\em(}1{\em)} Let $M$ be a lens space {\em(}including the case of $\s^3${\em)} equipped with a Seifert fibration such that the Euler number $e(M)>0$. Then there exists an $\s^1$-invariant, non-transverse contact structure $\xi$ on $M$ whose dividing set consists of a circle which separates the two singular points in the base orbifold when the Seifert fibration on $M$ has two multiple fibers. Furthermore, $\xi$ is universally tight. As a corollary, any $\s^1$-invariant contact structure on $M$ with such a dividing set is universally tight. 

{\em(}2{\em)} Let $M((\alpha_1,\beta_1), (\alpha_2,\beta_2), (\alpha_3,\beta_3))$ be a Seifert fibered space with at most $3$ singular fibers,  let $p_i$ be the point in the base orbifold which is associated to the fiber of multiplicity 
$\alpha_i>0$, for $i=1,2,3$. (Here $\alpha_i=1$ is allowed.) Assume the Seifert invariants obey the following 
conditions: (i) $\frac{\beta_1}{\alpha_1}+\frac{\beta_2}{\alpha_2}\leq 0$, and (ii)
$\beta_3<0$. Then there exists a tight (in fact fillable) $\s^1$-invariant, non-transverse contact structure $\xi$ on $M$ such that its dividing set consists of a circle separating the points $p_1$ and $p_2$, with $p_3$ lying on either side of the dividing set. As a corollary, any $\s^1$-invariant contact structure on $M$ with such a dividing set 
is tight (in fact fillable). 
\end{theorem}

The following lemma is the key step for a proof of Theorem 2.3.

\begin{lemma}
Let $\pi: \s^3\rightarrow S$ be a Seifert fibration on $\s^3$ which has a positive Euler number. Then
there is a tight, $\s^1$-invariant (with respect to the Seifert fibration $\pi$), non-transverse contact structure $\xi$ on $\s^3$ such that the dividing set of $\xi$ consists of a circle in $S$ which separates the singular points of $S$ when the Seifert fibration $\pi$ has two multiple fibers. 
\end{lemma}

\begin{proof}
We begin by considering the contact form $\alpha=\frac{1}{\sqrt{z^2+r^4}}(dz+r^2 d\theta)$ on $\R^3$, which defines the standard, tight, contact structure on $\R^3$. The $\R$-action $\phi_t: \R^3\rightarrow \R^3$, $t\in\R$, defined by $\phi_t(r,\theta,z)=(e^t r,\theta,e^{2t}z)$, is generated by the vector field
$X=r\partial_r+2z\partial_z$. It is easy to check that $\phi_t^\ast\alpha=\alpha$. The quotient of the free $\Z$-action on $\R^3\setminus\{0\}$ generated by $\phi_{2\pi}$ can be identified with $\s^2\times \s^1$. Note that both $\alpha$, $X$ descend to the quotient $\s^2\times \s^1$, where $X$ generates the $\s^1$-action on it whose orbits are given by $\{p\}\times \s^1$, for $p\in \s^2$, and $\alpha$ is an $\s^1$-invariant contact form on $\s^2\times \s^1$. 

The corresponding contact structure on $\s^2\times \s^1$ is clearly universally tight, hence tight. 
It is convenient to identify the base $\s^2$ with the hypersurface $z^2+r^4=1$. With this understood, 
note that the $\s^1$-invariant contact structure on $\s^2\times \s^1$ is non-transverse, with its dividing set $\Gamma$ given by $z=0, r=1$. Furthermore, there are two special points on the base $\s^2$: 
$p_{\pm}$ which have coordinates $(z,r)=(\pm 1,0)$. We point out that the fiber over $p_{\pm}$,
which is a transverse knot, has a standard neighborhood of arbitrarily large slope for the linear characteristic foliations on the boundary. Moreover, with respect to the framing defined by the product structure of $\s^2\times \s^1$, the characteristic foliation on $\Gamma\times \s^1$ has slope $0$, where $\Gamma$ is the dividing set.

With the preceding understood, we divide the proof into two cases.

\vspace{1.5mm}

{\bf Case 1:} The Seifert fibration on $\s^3$ has at most one singular fiber. Let $\alpha_1\geq 1$ be the multiplicity of the singular fiber, where $\alpha_1=1$ represents the case when there is no singular fiber. With this understood, if we perform a $(-\alpha_1)$-surgery on the fiber over $p_{+}$ in
$\s^2\times\s^1$, we obtain $\s^3$, and moreover, the $\s^1$-action on $\s^2\times \s^1$ induces a
Seifert fibration on $\s^3$, with one singular fiber of multiplicity $\alpha_1$ and Euler number
$\frac{1}{\alpha_1}$, which is exactly the given Seifert fibration on $\s^3$. Now since $p_{+}$ has a standard neighborhood of arbitrarily large slope, the $(-\alpha_1)$-surgery can be done by an admissible transverse surgery, producing an $\s^1$-invariant, non-transverse contact structure $\xi$ on $\s^3$, whose dividing set consists of the same circle $z=0,r=1$. Moreover, by a theorem of Baldwin and Etnyre (cf. \cite{BE}, Theorem 3.2), the admissible transverse $(-\alpha_1)$-surgery is equivalent to a Legendrian surgery on a link, which preserves tightness by a theorem of Wand (cf. \cite{Wa}). This proves that $\xi$ is tight. 

\vspace{1.5mm}

{\bf Case 2:} The Seifert fibration on $\s^3$ has two singular fibers. Let $\alpha_1,\alpha_2$ be the
multiplicities of the singular fibers, where $1<\alpha_1<\alpha_2$. Note that we can always arrange
so that the Seifert invariants take the form $(\alpha_1,\beta_1)$, $(\alpha_2,-\beta_2)$, where
$0<\beta_i<\alpha_i$ for $i=1,2$. Finally, since the Euler number of the Seifert fibration is positive,
we must have $\alpha_2\beta_1-\alpha_1\beta_2=-1$.

With the preceding understood, we observe that there is a natural free $\Z_{\alpha_1}$-action on
$\s^2\times\s^1$ which preserves the product structure and the $\s^1$-invariant contact form $\alpha$ on it, such that if we denote the quotient manifold by $Y$ and the induced contact structure on it 
by $\xi_0$, then $Y$ is the Seifert manifold $M((\alpha_1,\beta_1), (\alpha_1,-\beta_1))$, with the singular fibers given by the quotient of the fibers over $p_{\pm}$ under the $\Z_{\alpha_1}$-action, and $\xi_0$ is $\s^1$-invariant, non-transverse, where the dividing set of $\xi_0$ is given by the quotient of 
$\Gamma$ under the $\Z_{\alpha_1}$-action, hence is also a circle. For simplicity, we denote the corresponding singular points of the base $2$-orbifold of $M((\alpha_1,\beta_1), (\alpha_1,-\beta_1))$ also by $p_{\pm}$. With this understood, we assume without loss of generality that the singular fiber over $p_{+}$, denoted by $f_{+}$, has Seifert invariant $(\alpha_1,-\beta_1)$. 

To proceed further, we recall that according to \cite{N}, the Seifert invariants $(\alpha_1,\beta_1)$, 
$(\alpha_1,-\beta_1)$ correspond to a section $R$ of the Seifert fibration on 
$M((\alpha_1,\beta_1), (\alpha_1,-\beta_1))$ defined over the complement of a regular neighborhood of the singular points $p_{\pm}$ of the base $2$-orbifold. Adapting the notations in \cite{N}, we let $H$ be the class of the regular fiber, $Q$ be the component of $-\partial R$ over the boundary of the disc neighborhood of the singular point $p_{+}$, and $M$ be the meridian of the regular neighborhood of the singular fiber $f_{+}$, we have the relation $M=\alpha_1 Q-\beta_1 H$. 
Now we fix a longitude $L$ of $f_{+}$, and write $L=\alpha_1^\prime Q+\beta_1^\prime H$ for 
some $\alpha_1^\prime$, $\beta_1^\prime$ obeying $\alpha_1\beta_1^\prime+\beta_1\alpha_1^\prime=1$. Then we observe
that a simple closed curve $\gamma=pQ+q H$, with slope $s=p/q$, will have slope $s^\prime$ with
respect to the basis $(M,L)$, where $s^\prime$ and $s$ are related by the equation
$$
s^\prime=\frac{\beta_1^\prime s-\alpha_1^\prime}{\beta_1 s+\alpha_1}.
$$
In particular, the characteristic foliation of $\xi_0$ on the torus $T$ which consists of fibers over the
dividing set of $\xi_0$ has slope $s^\prime=-\alpha_1^\prime/\alpha_1$ with respect to $(M,L)$, as it has slope $s=0$ with respect to the basis $(Q,H)$. It follows easily that the transverse knot $f_{+}$ in $(Y,\xi_0)$ has a standard neighborhood $N$ such that the characteristic foliation on
$\partial N$ has slope $-\alpha_1^\prime/\alpha_1$ with respect to the framing defined by the longitude $L$. 

With the preceding understood, we observe that $\s^3$, together with the given Seifert fibration on it,
is the result of a surgery on $f_{+}$, where if we denote by $\tilde{M}$ the meridian of the solid 
torus in $\s^3$ after the surgery (i.e., the surgery torus), then $\tilde{M}=\alpha_2 Q-\beta_2 H$ in the basis $(Q,H)$. Expressing $\tilde{M}$ in the basis $(M,L)$, we have 
$$
\tilde{M}=(\alpha_2\beta_1^\prime+\beta_2\alpha_1^\prime)M+(\alpha_2\beta_1-\beta_2\alpha_1)L=
(\alpha_2\beta_1^\prime+\beta_2\alpha_1^\prime)M-L
$$
because $\alpha_2\beta_1-\alpha_1\beta_2=-1$. It follows immediately that $\s^3$, together with the given Seifert fibration on it, is the result of a $-(\alpha_2\beta_1^\prime+\beta_2\alpha_1^\prime)$-surgery on $f_{+}$ with respect to the framing defined by the longitude $L$. 

We claim that the $-(\alpha_2\beta_1^\prime+\beta_2\alpha_1^\prime)$-surgery on $f_{+}$ can be
done by an admissible transverse surgery. To see this, we compute the difference of the slopes: 
$$
-\frac{\alpha_1^\prime}{\alpha_1}-(-(\alpha_2\beta_1^\prime+\beta_2\alpha_1^\prime))=\frac{\alpha_2}{\alpha_1}>1>0. 
$$
(In the above calculation, we use the fact that $\alpha_1\beta_1^\prime+\beta_1\alpha_1^\prime=1$
and $\alpha_2\beta_1-\alpha_1\beta_2=-1$.)
In particular, this shows that the $-(\alpha_2\beta_1^\prime+\beta_2\alpha_1^\prime)$-surgery on $f_{+}$ can be done by an admissible transverse surgery, and moreover, if $n\in\Z$ such that 
$n<-\alpha_1^\prime/\alpha_1<n+1$, then  $-(\alpha_2\beta_1^\prime+\beta_2\alpha_1^\prime)<n$,
which implies, by Theorem 3.2 of \cite{BE}, that the admissible transverse 
$-(\alpha_2\beta_1^\prime+\beta_2\alpha_1^\prime)$-surgery on $f_{+}$ is equivalent to a
Legendrian surgery on a link. It follows easily that the surgery on $f_{+}$ produces a tight, 
$\s^1$-invariant, non-transverse contact structure $\xi$ on $\s^3$ with the desired dividing set
(note that the dividing set remains untouched during the transverse surgery). 

\end{proof}

{\bf Proof of Theorem 2.3:}

\vspace{2mm}

The case where $M=\s^3$ is covered in Lemma 2.4. On the other hand, for the case where $M$ is a lens space, the universally tightness of $\xi$ follows from Lemma 2.4 as a corollary, by observing that the lift of $\xi$ to the universal cover of $M$ satisfies the assumptions in Lemma 2.4. To see the existence of such an $\s^1$-invariant, non-transverse contact structure $\xi$ on $M$, we observe that, 
with $(Y,\xi_0)$ being understood as in the proof of Case 2 of Lemma 2.4, $M$ together with the given Seifert fibration on it, is the result of a $(\frac{\alpha_2\beta_1^\prime+\beta_2\alpha_1^\prime}{\alpha_2\beta_1-\alpha_1\beta_2})$-surgery of $Y$ on the multiple fiber $f_{+}$ with respect to the framing defined by the longitude $L$ (in the proof of Lemma 2.4, we are in the special case where 
$\alpha_2\beta_1-\alpha_1\beta_2=-1$). By a similar calculation, i.e.,
$$
-\frac{\alpha_1^\prime}{\alpha_1}-\frac{\alpha_2\beta_1^\prime+\beta_2\alpha_1^\prime}{\alpha_2\beta_1-\alpha_1\beta_2}=-\frac{\alpha_2}{\alpha_1(\alpha_2\beta_1-\alpha_1\beta_2)}>0,
$$
we see that the $(\frac{\alpha_2\beta_1^\prime+\beta_2\alpha_1^\prime}{\alpha_2\beta_1-\alpha_1\beta_2})$-surgery on $f_{+}$ can be done by an admissible transverse surgery which leaves the dividing set untouched. 
As a consequence, it produces an $\s^1$-invariant, non-transverse contact structure on $M$ with the desired dividing set. (Here in the above calculation we have used the fact that $\alpha_2\beta_1-\alpha_1\beta_2<0$ as $e(M)>0$ by assumption.) 

To proceed further, it remains to consider the case where $M$ is a Seifert fibered space with at most $3$ singular fibers. But first of all, we observe that by rearranging the Seifert invariants we may assume without loss of generality that $\alpha_3>1$ and $|\frac{\beta_3}{\alpha_3}|<1$, while preserving the condition (i): $\frac{\beta_1}{\alpha_1}+\frac{\beta_2}{\alpha_2}\leq 0$. 

With the preceding understood, we consider the case where 
$\frac{\beta_1}{\alpha_1}+\frac{\beta_2}{\alpha_2}<0$ first. Under this assumption, we note that the lens space $M((\alpha_1,\beta_1),(\alpha_2,\beta_2))$ admits a universally tight, $\s^1$-invariant and non-transverse contact structure whose dividing set is a circle separating the two points $p_1,p_2$ of multiplicities $\alpha_1,\alpha_2$
respectively (here $\alpha_i=1$, $i=1,2$, is allowed). Furthermore, the Seifert fibered space $M$ is the result 
of a $\frac{\alpha_3}{\beta_3}$-surgery of $M((\alpha_1,\beta_1),(\alpha_2,\beta_2))$ on a regular fiber of the Seifert fibration with respect to the framing determined by the fibration. With this understood, it suffices to show that the $\frac{\alpha_3}{\beta_3}$-surgery can be done by an admissible transverse surgery that is equivalent to a Legendrian surgery on a Legendrian link, which preserves tightness and fillability. 

To this end, we pick a regular fiber $L$ of $M((\alpha_1,\beta_1),(\alpha_2,\beta_2))$ which is a Legendrian vertical circle. Then by a result of Gay (cf. \cite{G}, Proposition 1.8), a transverse, sufficiently close push-off of $L$, to be denoted by $K$, has a standard neighborhood with linear slope $a=-\frac{1}{2}$ on the boundary (with respect to the framing given by the Seifert fibration, which is the contact framing of $L$). In the present situation, we may further assume that $K$ is a regular fiber of the Seifert fibration. Since $-1<a=-\frac{1}{2}<0$, by a theorem of Baldwin and Etnyre (cf. \cite{BE}, Theorem 3.2), any admissible transverse $s$-surgery on $K$, where $-1>s\in\Q$, is equivalent to a Legendrian surgery on a Legendrian link in a neighborhood of $K$. In particular, a transverse 
$\frac{\alpha_3}{\beta_3}$-surgery on $K$, with $-1<\frac{\beta_3}{\alpha_3}<0$, is admissible and is equivalent to a Legendrian surgery on a Legendrian link. This finishes the proof for the case where $\frac{\beta_1}{\alpha_1}+\frac{\beta_2}{\alpha_2}<0$.

Finally, for the case where $\frac{\beta_1}{\alpha_1}+\frac{\beta_2}{\alpha_2}=0$, we must have
$\alpha_2=\alpha_1$ and $\beta_2=-\beta_1$. With this understood, note that the same argument, applied to the Seifert space $M((\alpha_1,\beta_1),(\alpha_1,-\beta_1))$ (which is the $3$-manifold $Y$, with an $\s^1$-invariant, non-transverse, and universally tight contact structure $\xi_0$ in the proof of Lemma 2.4), finishes the proof for this case. The proof of Theorem 2.3 is complete. 

\begin{example}
As Lemma 2.4 is the key step for the proof of Theorem 2.3, a natural question is whether the condition, i.e., the Euler number of the Seifert fibration $\pi: \s^3\rightarrow S$ is positive, is also a necessary 
condition for the tightness of the $\s^1$-invariant, non-transverse contact structure on $\s^3$ 
constructed in Lemma 2.4. In this example, we shall examine a special case where this condition fails, and explain that in this case
the $\s^1$-invariant, non-transverse contact structure on $\s^3$ has to be obtained by an inadmissible 
transverse surgery rather than an admissible transverse surgery we employed in the proof of 
Lemma 2.4. Furthermore, the $\s^1$-invariant contact structure thus obtained must be overtwisted. 

To this end, consider $\s^3$ as the Seifert fibered space $M=M((2,1),(3,-1))$, which has Euler number
$e(M)=-(\frac{1}{2}+\frac{-1}{3})=-\frac{1}{6}<0$. (Note that this is on the contrary of the assumption in Lemma 2.4). We claim that there is an $\s^1$-invariant, non-transverse contact structure on $\s^3$,
whose dividing set is a circle separating the two singular points of multiplicity $2$ and $3$. To see this,
we consider the Seifert fibered space $M=M((2,1), (2,-1))$ with the universally tight, $\s^1$-invariant
contact structure $\xi_0$ as in the Case 2 of the proof of Lemma 2.4. (Here we have 
$(\alpha_1,\beta_1)=(2,1)$, $(\alpha_1^\prime,\beta_1^\prime)=(-1,1)$, and $(\alpha_2,\beta_2)=
(3,1)$.) With this understood, to obtain $\s^3=M((2,1),(3,-1))$, we need to perform a
$(\frac{\alpha_2\beta_1^\prime+\beta_2\alpha_1^\prime}{\alpha_2\beta_1-\alpha_1\beta_2})$-surgery 
on the singular fiber of Seifert invariant $(2,-1)$ in $M=M((2,1), (2,-1))$. But 
$\frac{\alpha_2\beta_1^\prime+\beta_2\alpha_1^\prime}{\alpha_2\beta_1-\alpha_1\beta_2}=2$, 
while the characteristic foliation on the torus over the dividing set has slope $-\frac{\alpha_1^\prime}{\alpha_1}=\frac{1}{2}$. So if we perform an admissible transverse $2$-surgery on the singular fiber
of Seifert invariant $(2,-1)$ in $M=M((2,1), (2,-1))$, we can not leave the dividing set untouched, 
as the slopes $2>\frac{1}{2}$. In order to leave the dividing set untouched, we have to perform an
inadmissible transverse $2$-surgery on the singular fiber. The result is an $\s^1$-invariant, non-transverse contact structure $\xi$ on $\s^3=M((2,1),(3,-1))$, with a dividing set consisting of a circle 
separating the two singular points of multiplicity $3$ and $2$. 

To see that the resulting contact structure $\xi$ must be overtwisted, we appeal to a theorem of 
Ghiggini and Sch\"{o}nenberger (cf. \cite{GSch}, Theorem 4.14) that the small Seifert space 
$M((2,1),(3,-1),(11,-2))$ does not admit any tight contact structures which contain a Legendrian vertical circle with twisting number $0$ (including in particular tight, $\s^1$-invariant non-transverse contact structures).
For if $\xi$ were tight, we could perform an admissible transverse $-\frac{11}{2}$-surgery on a transverse push-off of a Legendrian vertical circle to produce an $\s^1$-invariant, non-transverse tight contact structure on the small Seifert space $M((2,1),(3,-1), (11,-2))$, contradicting the aforementioned theorem of Ghiggini and Sch\"{o}nenberger. Thus $\xi$ must be overtwisted. 

In summary, the above example illustrated the necessity of the assumption in Lemma 2.4 that the
Euler number of the Seifert fibration must be positive. For if otherwise, we are facing a choice:
either perform an admissible transverse surgery but having the dividing set eliminated, or perform 
an inadmissible transverse surgery while keeping the dividing set untouched. The former choice 
results a transverse, tight contact structure, while the latter choice gives rise to a non-transverse 
$\s^1$-invariant contact structure which is not necessarily tight. 
\end{example}

The discussions in Example 2.5 suggest that it is necessary to put in place the assumptions in 
Theorem 2.3, i.e., $e(M)>0$ in part (1) and (i) $\frac{\beta_1}{\alpha_1}+\frac{\beta_2}{\alpha_2}\leq 0$ 
and (ii) $\beta_3<0$ in part (2), in order to ensure that the $\s^1$-invariant, non-transverse 
contact structure constructed therein is tight. Note that in part (2), in the case of a small Seifert space, 
$e_0(M)\geq -1$ is necessary for the existence of a tight, non-transverse $\s^1$-invariant contact structure
(cf. Lemma 2.2). If $e_0(M)\geq 0$, then one can always arrange so that $\beta_i$ are negative for $i=1,2,3$,
which implies that the assumptions in Theorem 2.3 (2) hold true. If $e_0(M)=-1$, then exactly one of 
$\beta_1,\beta_2$ is negative. Suppose the point $p_3$ lies on the same side of the dividing circle with $p_2$. 
We may assume without loss of generality that $\beta_1>0$ and $\beta_2<0$. With this understood, 
Example 2.5 suggests that either $\frac{\beta_1}{\alpha_1}+\frac{\beta_2}{\alpha_2}\leq 0$ and $\beta_3<0$, or $\frac{\beta_1}{\alpha_1}+\frac{\beta_3}{\alpha_3}\leq 0$ and $\beta_2<0$, must be necessarily true in order for the $\s^1$-invariant, non-transverse contact structure to be tight. We are not able to provide a rigorous proof of these beliefs, so we shall formulate the statements in a form as in the following question. 

\begin{question}
\begin{itemize}
\item [{(1)}] In Lemma 2.4, is the Euler number of the Seifert fibration being positive a necessary condition for the 
$\s^1$-invariant, non-transverse contact structure $\xi$ to be tight?
\item [{(2)}] Let $M((\alpha_1,\beta_1), (\alpha_2,\beta_2), (\alpha_3,\beta_3))$ be a small Seifert space with
$e_0(M)=-1$, and let $p_i$ be the singular point associated to the singular fiber of multiplicity $\alpha_i>1$, 
for $i=1,2,3$. Assume the Seifert invariants are arranged to obey the following conditions: $\beta_i<0$, $|\beta_i|<\alpha_i$ for $i=2,3$, and $0<\beta_1<\alpha_1$. Furthermore, assume both 
$\frac{\beta_1}{\alpha_1}+\frac{\beta_2}{\alpha_2}>0$ and $\frac{\beta_1}{\alpha_1}+\frac{\beta_3}{\alpha_3}>0$. 
If $\xi$ is an $\s^1$-invariant, non-transverse contact structure whose dividing set is a circle separating the point 
$p_1$ from the points $p_2,p_3$, must $\xi$ be overtwisted?
\end{itemize}
\end{question}

See Section 6 for some interesting implication of Question 2.6 in the context of rational cuspidal curves.

\section{Rational open books with periodic monodromy}

Let $M$ be a compact closed, oriented $3$-manifold. A rational open book on $M$ consists of an oriented link $B=\sqcup_i B_i$ (called the binding) and a smooth fibration $\pi: M\setminus B\rightarrow \s^1$, which satisfies the following conditions: let $\Sigma:=\pi^{-1}(\lambda)$, 
$\lambda\in\s^1$, denote the pages of the open book, which are canonically oriented by the orientations of $M$ and $\s^1$, then for each $i$, there is a regular neighborhood $N_i$ of $B_i$, such that, setting $(\partial \Sigma)_i:=\Sigma\cap \partial N_i$ and orienting it as a boundary component of $\Sigma\cap (M\setminus \sqcup_i N_i)$, one has (i) each $(\partial \Sigma)_i$ is connected, (ii) $(\partial \Sigma)_i=p_i B_i$ in $H_1(N_i)$ for some integer $p_i>0$. The integer $p_i$ is called the {\bf multiplicity} at $B_i$. Note that when each $p_i=1$, the rational open book $(B,\pi)$ is the usual open book decomposition of $M$. See \cite{BEV} for a more comprehensive discussion. 

We are interested in the case where $\pi: M\setminus \sqcup_i N_i\rightarrow \s^1$ is given by the mapping torus of a periodic diffeomorphism leaving each boundary component invariant; in this case, we say $(B,\pi)$ is of {\bf periodic monodromy}. It is clear that a rational open book $(B,\pi)$ with periodic monodromy gives rise to a natural smooth $\s^1$-action on $M$. The purpose of this section is to relate the data associated to the rational open book $(B,\pi)$ to the corresponding generalized Seifert invariants of $M$ (as discussed in \cite{N}), which enables us to construct rational open books with various properties on a given Seifert manifold. We begin with some technical lemmas. 

\begin{lemma}
Let $S_0$ be a closed, oriented $2$-orbifold with singular points $z_1,z_2,\cdots,z_m$, where we denote the genus of the underlying surface of $S_0$ by $g_0$ and the order of isotropy of $z_i$ by
$n_i$, for $1\leq i\leq m$. Then there exists a closed, oriented surface $\tilde{S}$ together 
with a periodic diffeomorphism $\phi: \tilde{S}\rightarrow \tilde{S}$ of order $n$ such that the quotient 
$\tilde{S}/\langle \phi\rangle=S_0$ {\em(}as orbifolds{\em)} if and only if there are integers $c_1,c_2,\cdots,c_m$ satisfying the following condition:
\begin{itemize}
\item[{(\dag)}] $\forall i, 0<c_i<n_i, gcd(n_i,c_i)=1$, and 
$\frac{c_1}{n_1}+\frac{c_2}{n_2}+\cdots+\frac{c_m}{n_m}\in\Z$.
\end{itemize}
Moreover, $(\tilde{S},\phi)$ uniquely determines the integers $c_1,c_2,\cdots,c_m$, with
$n=lcm(n_1,\cdots,n_m)$ when $g_0=0$. On the other hand, given any integers 
$c_1,c_2,\cdots,c_m$ satisfying $(\dag)$, one can choose $(\tilde{S},\phi)$ such that 
$\phi: \tilde{S}\rightarrow \tilde{S}$ has order $n=lcm(n_1,\cdots,n_m)$.
\end{lemma}

\begin{proof}
The orbifold fundamental group of $S_0$ has the following presentation
$$
\pi_1^{orb}(S_0)=\{a_k,b_k,\gamma_i |\gamma_i^{n_i}=1, \forall i, \gamma_1\gamma_2\cdots
\gamma_m \prod_{k=1}^{g_0} [a_k,b_k]=1\}.
$$
Suppose there exists a closed, oriented surface $\tilde{S}$ together 
with a periodic diffeomorphism $\phi: \tilde{S}\rightarrow \tilde{S}$ such that $\tilde{S}/\langle \phi\rangle=S_0$.
Then $\pi: \tilde{S}\rightarrow S_0$ is an orbifold covering map with $\phi$ being a generator of the deck
transformations. This gives rise to a surjective homomorphism $\rho: \pi_1^{orb}(S_0)\rightarrow \Z_n$
with $\ker\rho=\pi_1(\tilde{S})$. Let $t\in\Z_n$ be the generator which corresponds to $\phi$. 

For each $i$, let $0<m_i<n$ be the number of points in the orbit of $\phi$ that corresponds to 
$z_i\in S_0$. Then $m_in_i=n$, and $\phi^{m_i}$ is a generator of the isotropy subgroup which fixes
each point in the orbit of $\phi$ that corresponds to $z_i\in S_0$. It follows that there is a unique
$c_i$, satisfying $0<c_i<n_i$, such that $\rho(\gamma_i)=t^{c_im_i}$. Since no nontrivial powers of
$\gamma_i$ lies in $\ker\rho=\pi_1(\tilde{S})$ (as $\tilde{S}$ is nonsingular), we must have $gcd(n_i,c_i)=1$.
Finally, note that $\rho(\gamma_1\gamma_2\cdots\gamma_m)=1$, which implies 
$$
t^{(\frac{c_1}{n_1}+\frac{c_2}{n_2}+\cdots+\frac{c_m}{n_m})n}=1.
$$
It follows that $\frac{c_1}{n_1}+\frac{c_2}{n_2}+\cdots+\frac{c_m}{n_m}\in\Z$ as $t$ has order $n$.
It remains to show that if $g_0=0$, we must have $n=lcm(n_1,\cdots,n_m)$. To see this, note that
$g_0=0$ implies that $\pi_1^{orb}(S_0)$ is generated by $\{\gamma_i\}$. Then the surjectivity of
$\rho$ implies that there are integers $d_1,d_2,\cdots,d_m$ such that
$$
t=\rho(\gamma_1^{d_1}\gamma_2^{d_2}\cdots\gamma_m^{d_m}),
$$
which gives rise to $\sum_{i=1}^m d_ic_im_i\equiv 1 \pmod{n}$. It follows that
$gcd(m_i|1\leq i\leq m)=1$, hence $n=lcm(n_1,\cdots,n_m)$, as $m_in_i=n$ for each $i$. 

Conversely, suppose there are integers $c_1,c_2,\cdots,c_m$ satisfying $(\dag)$. Then setting
$n=lcm(n_1,\cdots,n_m)$, and fixing a generator $t\in \Z_n$, we define a homomorphism
$\rho: \pi_1^{orb}(S_0)\rightarrow\Z_n$ by 
$$
\rho(a_k)=\rho(b_k)=1, \forall k, \mbox{ and } \rho(\gamma_i)=t^{c_in/n_i}, \forall i.
$$
The condition $\frac{c_1}{n_1}+\frac{c_2}{n_2}+\cdots+\frac{c_m}{n_m}\in\Z$ implies that
$\rho(\gamma_1\gamma_2\cdots\gamma_m)=1$, so that $\rho$ is well-defined. To see that
$\rho$ is surjective, we note that $gcd(n/n_1,n/n_2,\cdots,n/n_m)=1$, so that there are integers
$k_1,k_2,\cdots,k_m$ such that $k_1n/n_1+k_2n/n_2+\cdots+k_mn/n_m=1$. Finally, for each $i$,
pick $c_i^\prime$ such that $c_ic_i^\prime\equiv 1 \pmod{n_i}$. Then it follows that
$$
t=\rho(\gamma_1^{k_1c_1^\prime}\gamma_2^{k_2c_2^\prime}\cdots\gamma_m^{k_mc_m^\prime}),
$$
hence $\rho$ is surjective. Consequently, there is an orbifold covering $\pi: \tilde{S}\rightarrow S_0$ with
a cyclic deck transformation group generated by a diffeomorphism $\phi: \tilde{S}\rightarrow \tilde{S}$ of order $n$, such that $\ker\rho=\pi_1^{orb}(\tilde{S})$. It remains to show that $\tilde{S}$ has no singular points, which happens
precisely when for each $i$, there are no integers $0<l_i<n_i$ such that $\rho(\gamma_i^{l_i})=1$. Note
that $\rho(\gamma_i^{l_i})=t^{l_ic_in/n_i}$, so $\rho(\gamma_i^{l_i})=1$ implies that $c_il_i/n_i\in\Z$,
which is a contradiction as $gcd(n_i,c_i)=1$ and $0<l_i<n_i$. This finishes the proof.

\end{proof}

We note from the proof that for each $i$, $\phi^{\frac{n}{n_i}c_i}$ fixes each point in the orbit that corresponds to $z_i\in S_0$ and acts as a rotation by an angle $\frac{2\pi}{n_i}$ in a neighborhood of it. 

\vspace{2mm}

Now suppose we are given with a rational open book $(B,\pi)$ on $M$ which has a periodic monodromy of order $n>1$, and let $\phi:\Sigma\rightarrow\Sigma$ be the periodic diffeomorphism which defines the monodromy of $(B,\pi)$. Then as oriented $3$-manifolds, $M\setminus (\cup_i N_i)$ is identified with the mapping torus of $\phi$, i.e., $\Sigma\times [0,1]/ (x,1)\sim (\phi(x),0)$,  under which $\pi$ sends $(x,t)\in \Sigma\times [0,1]$ to $t$. 

To further describe $\phi:\Sigma\rightarrow\Sigma$, for each $i$ we identify the boundary component 
$(\partial \Sigma)_i$ of $\Sigma$ with $\R/\Z$ such that the orientation of $(\partial \Sigma)_i$ is given by the orientation of $\R$. With this understood, since each $(\partial \Sigma)_i$ is invariant under 
$\phi$, there is an integer $k_i$, where $0<k_i<n$ and $gcd(n,k_i)=1$, such that on 
$(\partial \Sigma)_i=\R/\Z$, $\phi$ is given by the translation $x\mapsto x+\frac{k_i}{n}$. We 
close up $\Sigma$ by adding a disc $D_i$ to the boundary component $(\partial \Sigma)_i$, 
obtaining a closed oriented surface $\tilde{S}$. There is a natural extension of $\phi$ to $\tilde{S}$, which is still denoted by $\phi$. 

Let $S_0$ be the quotient orbifold $\tilde{S}/\langle \phi\rangle$. The singular points of $S_0$ are divided into two groups: for each $i$, we let $z_i$ be the singular point contained in $D_i/\langle \phi\rangle$, which is of order of isotropy $n$, and we denote by $\{z_j\}$ the singular points contained in $\Sigma/\langle \phi\rangle$, where the order of isotropy of $z_j$ is denoted by $n_j$. Then 
by Lemma 3.1, there are uniquely determined integers $\{c_j\}$, $\{c_i\}$, such that
$$
0<c_j<n_j, gcd(n_j,c_j)=1, \forall j, \;\;\; 0<c_i<n, gcd(n,c_i)=1, \forall i, 
$$
and $b:=\sum_j \frac{c_j}{n_j}+\sum_i \frac{c_i}{n}\in \Z$. With this understood, we note from the remark following the proof of Lemma 3.1 that for each $i$, $k_ic_i\equiv -1 \pmod{n}$. We denote by $l_i$ the unique integer such that $nl_i-k_ic_i=1$. Finally, let $g_0$ be the genus of $S_0$.

Next, we set $T_i:=(\partial \Sigma)_i \times [0,1]/ (x,1)\sim (\phi(x),0)$, which is oriented as part of the boundary of the mapping torus $\Sigma\times [0,1]/ (x,1)\sim (\phi(x),0)$. We consider two 
curves on $T_i$: $\gamma_i:=(\partial \Sigma)_i \times \{0\}$ and 
$\tau_i:=\{(\frac{k_i}{n}(1-t),t)| 0\leq t\leq 1\}$, where $\gamma_i$ is oriented as $(\partial \Sigma)_i$
and $\tau_i$ is oriented by $t$. Then note that $(\gamma_i,\tau_i)$ is a positively oriented basis of
$H_1(T_i)$ (i.e., the intersection number $\gamma_i\cdot \tau_i=1$). 

On the other hand, for each $i$, we let $\mu_i$ be the meridian on
$\partial N_i=-T_i$, oriented as the boundary of the normal disc of $B_i$ in $N_i$, where the normal disc is so oriented that it intersects positively with $B_i$. With this understood, we note that for each $i$, 
$$
\mu_i=-p_i^\prime\gamma_i +p_i\tau_i
$$
for some integer $p_i^\prime$ (recall $p_i>0$ is the multiplicity at $B_i$). 

On the mapping torus $\Sigma\times [0,1]/ (x,1)\sim (\phi(x),0)$, the translations in the variable $t$
defines a $\s^1$-action, which can be canonically extended to each solid torus $N_i$ such that the binding component $B_i$ is either an orbit or a fixed-point set component. This defines 
a generalized Seifert fibration structure on the $3$-manifold $M$ in the sense of \cite{N}. The next lemma describes this structure on $M$, where we set for each $i$, 
$$
\alpha_i:=k_ip_i+np^\prime_i,\;\; \beta_i:=l_ip_i+c_i p^\prime_i.
$$

\begin{lemma}
The $\s^1$-action on $M$ inherited from the rational open book $(B,\pi)$ identifies $M$ as
$$
M=M(g_0; (1, -b), \{(\alpha_j,\beta_j)\}, \{(\alpha_i,\beta_i)\}), 
$$
where $b=\sum_j \frac{c_j}{n_j}+\sum_i \frac{c_i}{n}$, $(\alpha_j,\beta_j)=(n_j,c_j)$, and
$(\alpha_i,\beta_i)$ is the generalized Seifert invariant at the binding component $B_i$ subject
to the following interpretation: {\em(}i{\em)} if $\alpha_i<0$, then the Seifert invariant at the fiber is 
$(-\alpha_i,-\beta_i)$ and $B_i$ has the opposite orientation of the corresponding fiber, 
and {\em(}ii{\em)} if $\alpha_i=0$, then $\beta_i=1$, and $B_i$ is a component of the fixed-point set of 
the $\s^1$-action. Finally, the Euler number of the Seifert fibration is given by
$e(M)=-\frac{1}{n} \sum_i \frac{p_i}{\alpha_i}$ when $\alpha_i\neq 0$ for each $i$, where $n$ is the order of the monodromy of $(B,\pi)$ and $p_i$ is the multiplicity at $B_i$. 
\end{lemma}

\begin{proof}
We consider the Seifert manifold $Y$ over the $2$-orbifold $S_0$, where 
$$
Y:=\tilde{S}\times \s^1/(x,t)\sim (\phi(x), \exp(-2\pi i/n) t).
$$
Since the Euler number $e(Y)=0$, it follows easily that $Y$ is the Seifert manifold 
$$
Y=M(g_0; (1,-b), \{(n_j,c_j)\}, \{(n,c_i)\}),
$$
where $(n_j,c_j)$, $(n,c_i)$ is the normalized Seifert invariant of the fiber over $z_j,z_i\in S_0$ respectively. Moreover, the choice of the Seifert invariants $(1,-b), \{(n_j,c_j)\}, \{(n,c_i)\}$ corresponds to a section $R$ of the Seifert fibration on $Y$ defined in the complement of a regular neighborhood of the corresponding fibers (cf. \cite{N}). Adapting the notations in \cite{N}, we let $Q_i$ be the component of $-\partial R$ on the boundary of a regular neighborhood of the singular fiber over $z_i\in S_0$, and let $H_i$ be the fiber class on the boundary. With this understood, we observe that the meridian $M_i$ is given by $-\gamma_i$, and $\tau_i$ is a longitude which we will fix to be our choice of $L_i$. Furthermore, we observe that on $T_i$, fixing an $x_0\in (\partial \Sigma)_i=\R/\Z$, a fiber of the Seifert fibration of $Y$ is given by 
$$
\{(x_0+\frac{m}{n}, t)| m=0,1,\cdots, n-1, 0\leq t\leq 1\}/(x,1)\sim (x+\frac{k_i}{n},0).
$$
It follows easily that $H_i=k_i\gamma_i+n\tau_i=-k_i M_i+n L_i$. On the other hand, the Seifert invariant (of the Seifert fibration of $Y$) at the singular fiber over $z_i\in S_0$ is $(n,c_i)$, and with
the relation $nl_i-k_ic_i=1$, it follows easily that $Q_i=l_i M_i-c_i L_i=-l_i\gamma_i-c_i\tau_i$.

To determine the Seifert invariants of the generalized Seifert fibration on $M$, we recall that the 
meridian of the solid torus $N_i$ is $\mu_i=-p^\prime_i\gamma_i+p_i\tau_i$. Moreover, 
we will use the orientation of $T_i$, so that $H_i\cdot Q_i=L_i\cdot M_i=\tau_i\cdot (-\gamma_i)=1$.
It follows that
$$
H_i\cdot\mu_i=k_i\gamma_i\cdot\mu_i+ n\tau_i\cdot\mu_i=k_ip_i+np^\prime_i=\alpha_i,
$$
and
$$
Q_i\cdot \mu_i=-l_i\gamma_i\cdot\mu_i-c_i\tau_i\cdot\mu_i=-l_ip_i-c_i p^\prime_i=-\beta_i.
$$
This proves immediately that when $\alpha_i>0$, the binding component $B_i$ is a regular 
or singular fiber of the Seifert fibration on $M$ with Seifert invariant $(\alpha_i,\beta_i)$, and when
$\alpha_i<0$, $-B_i$ is a fiber with Seifert invariant $(-\alpha_i,-\beta_i)$. Furthermore, if
$\alpha_i=0$, one must have $n=p_i$, $k_i=-p^\prime_i$, so that $\beta_i=l_ip_i+c_ip^\prime_i=
l_in-c_ik_i=1$. In this case, $H_i=\mu_i$, from which it follows easily that $B_i$ is fixed under the
$\s^1$-action on $M$. Finally, assuming $\alpha_i\neq 0$ for each $i$, the Euler number of $M$ is given by 
$$
e(M)= -(-b+\sum_j \frac{c_j}{n_j}+\sum_i \frac{\beta_i}{\alpha_i})=-\sum_i (\frac{\beta_i}{\alpha_i}-
\frac{c_i}{n})=-\frac{1}{n} \sum_i \frac{p_i}{\alpha_i}.
$$
\end{proof}

By reverse engineering, we obtain existence results for rational open books compatible with a
given generalized Seifert fibration. If we require the binding of the open book be connected, there is a
mild condition one needs to impose. 

\begin{theorem}
Let $M=M(g_0; (\alpha,\beta), \{(\alpha_j,\beta_j)\})$ be a generalized Seifert fibration, where for each $j$, $0<\beta_j<\alpha_j$ {\em(}here $\alpha=1$ or $(\alpha,\beta)=(0,1)$ is allowed, but we assume 
$\alpha\geq 0${\em)}, and $e(M)\neq 0$ when $\alpha\neq 0$. Let 
$n=lcm(\alpha_j)$, and $m$ be the integer such that $\sum_j \frac{\beta_j}{\alpha_j}=\frac{m}{n}$.
Then there is a compatible rational open book of periodic monodromy $(B,\pi)$ on $M$ such that 
$B$ or $-B$ is the fiber with Seifert invariant $(\alpha,\beta)$ if and only if 
\begin{itemize}
\item [{(*)}] $gcd(n,m)=1$.
\end{itemize}
Moreover, when such a rational open book $(B,\pi)$ exists, the order of its monodromy must equal 
$n=lcm(\alpha_j)$, and $B$ has the same orientation as the fiber if and only if $e(M)<0$. Finally, we note that the multiplicity $p$ at $B$ is given by $p=|n\alpha \cdot e(M)|$.
\end{theorem}

\begin{proof}
First, assume that there is such a rational open book $(B,\pi)$. Then if we continue to use the notations in Lemma 3.2, we have $(\alpha_j,\beta_j)=(n_j,c_j)$ for each $j$, and with $B=B_i$, we have $\alpha=\alpha_i$ and $\beta=\beta_i-b\alpha_i$. Furthermore, we note that the order of the monodromy of $(B,\pi)$ equals the order of isotropy at $z_i\in S_0$, which is denoted by $n_i$. 
With this understood, note that $b=\sum_j \frac{c_j}{n_j}+\frac{c_i}{n_i}\in \Z$ implies that 
$\frac{m}{n}+\frac{c_i}{n_i}\in\Z$. With $gcd(n_i,c_i)=1$, it follows that $n_i|n$. On the other hand, $n=lcm(\alpha_j)=lcm(n_j)$, which must be a divisor of the order of the monodromy $n_i$. Hence $n=n_i$ must be true. It follows easily that $gcd(n,m)=1$. 

Conversely, assume $gcd(n,m)=1$. To construct the rational open book $(B,\pi)$, we first observe that there is an integer $c$ such that $0<c<n$, $gcd(n,c)=1$, such that 
$$
b:=\sum_j \frac{\beta_j}{\alpha_j}+\frac{c}{n}\in\Z.
$$
Let $S_0$ be an oriented $2$-orbifold whose underlying surface has genus $g_0$, with a singular set
$\{z_j,z\}$, where the order of isotropy at $z_j$ is $\alpha_j$, and the order of isotropy at $z$ is $n$.
By Lemma 3.1, there is a closed, oriented surface $\tilde{S}$ together with a periodic diffeomorphism 
$\phi: \tilde{S}\rightarrow \tilde{S}$ of order $n$, such that $\tilde{S}/\langle\phi\rangle=S_0$ as orbifolds. The pre-image of $z\in S_0$ in $\tilde{S}$ is a fixed point of $\phi$; let $\Sigma$ be the complement of a small, $\phi$-invariant disc containing the fixed-point. 

Let $k,l$ be the integers such that $nl-kc=1$, where $0<k<n$. Set $\tilde{\alpha}=\alpha$, 
$\tilde{\beta}=\beta+b\alpha$. Let $p,p^\prime\in\Z$ be the unique solutions to the equations
$$
\tilde{\alpha}=kp+np^\prime,\;\; \tilde{\beta}=lp+cp^\prime.
$$
Then by Lemma 3.2, $M$ is diffeomorphic to the $3$-manifold obtained by gluing a solid torus $N$ to
the mapping torus of $\phi: \Sigma\rightarrow \Sigma$ with the meridian of $N$ being 
$\mu=-p^\prime\gamma+p\tau$. This proves the existence of the compatible open book $(B,\pi)$.
Finally, we note that if $(\alpha,\beta)=(0,1)$, we have $(\tilde{\alpha},\tilde{\beta})=(0,1)$ as well,
and $p=n>0$, $p^\prime=-k$. If $\alpha>0$, we note that
$$
e(M)=-(\sum_j\frac{\beta_j}{\alpha_j}+\frac{\beta}{\alpha})=-(b+\frac{\beta}{\alpha}-\frac{c}{n})=
-(\frac{\tilde{\beta}}{\tilde{\alpha}}-\frac{c}{n})=-\frac{p}{n\alpha},
$$
which shows that $p>0$ if and only if $e(M)<0$, and $p=|n\alpha \cdot e(M)|$.

\end{proof}

More generally, we have the following theorem by the same argument. 

\begin{theorem}
Let $M=M(g_0; \{(\alpha_i,\beta_i)\},\{(\alpha_j,\beta_j)\})$, where $0<\beta_j<\alpha_j$ for each $j$, and $\alpha_i\geq 0$ for each $i$. Then for any given common multiple $n$ of $\{\alpha_j\}$, any given integers $c_i$ satisfying the conditions $0<c_i<n$ and $gcd(n,c_i)=1$ for each $i$, and 
$\sum_j\frac{\beta_j}{\alpha_j}+\sum_i \frac{c_i}{n}\in\Z$, and for any given integers $b_i$ such that 
$\sum_i b_i=\sum_j\frac{\beta_j}{\alpha_j}+\sum_i \frac{c_i}{n}$, where for each $i$ with 
$\alpha_i\neq 0$, $\frac{\beta_i}{\alpha_i}+b_i-\frac{c_i}{n}\neq 0$, there is a rational open book 
$(B,\pi)$ of periodic monodromy of order $n$ on $M$, such that {\em(}1{\em)} $B=\sqcup_i B_i$, 
where $B_i$ or $-B_i$ is the fiber with Seifert invariant $(\alpha_i,\beta_i)$,  and {\em(}2{\em)} for each $i$ with
$\alpha_i\neq 0$, $B_i$ has the same orientation as the fiber if and only if 
$\frac{\beta_i}{\alpha_i}+b_i-\frac{c_i}{n}>0$.
\end{theorem}

\begin{remark}
We point out that for the rational open book $(B,\pi)$ in Theorem 3.4, the multiplicity $p_i>0$ at the binding component $B_i$ is given by the expression 
$$
p_i=|n\beta_i+n\alpha_i b_i-\alpha_i c_i|,
$$
which includes the case of $\alpha_i=0$, where $p_i=n$. (Note that $\alpha_i\geq 0$ in Theorem 3.4.)
\end{remark}

\section{Relative handlebodies of symplectic $2$-handles on a Seifert manifold}
This section is devoted to a study of relative handlebodies of symplectic $2$-handles built on a
(generalized) Seifert manifold equipped with a certain $\s^1$-invariant contact structure. The construction of such relative handlebodies is based on the technique introduced by Gay in \cite{G}; 
in particular, the two boundary components of the relative handlebody are both concave contact boundaries. It turns out that the other boundary component is also a (generalized) Seifert manifold equipped with a certain $\s^1$-invariant contact structure. Among the results of this section, we give a recipe to determine what are the possible values of the Seifert invariants and the possible $\s^1$-invariant contact structures on the other boundary component which can be realized (see Remark 4.5). We also give a recipe to compute the framings of the symplectic $2$-handles using the relevant Seifert invariants from the two boundary components (cf. Lemma 4.6). As a consequence,
the relative handlebody is completely determined by its two boundary components; in particular, it does not depend on the rational open book used in its construction (see Remark 4.7). 

We begin by describing the $\s^1$-invariant contact structures on the (generalized) Seifert manifolds. 
Let $(B,\pi)$ be a rational open book on $M$, where the pages of $(B,\pi)$ are denoted by $\Sigma$.
Recall that a contact structure $\xi$ on $M$ is said to be supported by $(B,\pi)$ if there exists a contact form $\alpha$ of $\xi$ such that (i) $d\alpha>0$ on $\Sigma$, (ii) $\alpha>0$ on $B$. It is known that such a contact structure $\xi$ always exists and is uniquely determined up to contact isotopy (cf. \cite{BEV}). The following theorem is concerned with the special case where $(B,\pi)$ is of a periodic monodromy.

\begin{theorem}
Let $(B,\pi)$ be a rational open book on $M$ which has periodic monodromy. With $M$ given the natural $\s^1$-action from $(B,\pi)$, there is an $\s^1$-invariant contact structure $\xi$ on $M$ which
is supported by $(B,\pi)$. Furthermore, $\xi$ is a transverse contact structure if and only if all the binding components $B_i$ which are not a fixed component of the $\s^1$-action have the same orientation as the orbit of the $\s^1$-action. Assuming the $\s^1$-action on $M$ is fixed-point free and $\xi$ is non-transverse, then there is a one to one correspondence between the components of the dividing set of $\xi$ and the binding components $B_i$ which have the opposite orientation, such that the component of the dividing set is the boundary of a disk neighborhood of the point over which the corresponding binding component lies. 
\end{theorem}

We remark that a special case of Theorem 4.1 has appeared in \cite{CH}, 
where $(B,\pi)$ is an integral open book and all the binding components $B_i$ have the fiber
orientation. 

On the other hand, we observe the following immediate corollary of Lemma 2.2.

\begin{corollary}
Let $(B,\pi)$, $(\tilde{B},\tilde{\pi})$ be two rational open books with periodic monodromy, which give rise to the same Seifert fibration structure on $M$. If the binding components of $(B,\pi)$, 
$(\tilde{B},\tilde{\pi})$ which have the opposite orientation of the fibers of the Seifert fibration coincide,
then the contact structures supported by $(B,\pi)$, $(\tilde{B},\tilde{\pi})$ are contactomorphic. 
\end{corollary} 

\vspace{2mm}

{\bf Proof of Theorem 4.1:}

\vspace{2mm}

We shall construct a specific $\s^1$-invariant contact form $\alpha$ obeying the conditions
(i) $d\alpha>0$ on each page $\Sigma$, (ii) $\alpha>0$ on the binding $B$. Theorem 4.1 follows immediately 
from such a contact form. On the other hand, Gay's construction in \cite{G} requires the existence of a certain contact form such that, with respect to some local coordinate system near each binding component $B_i$, both the contact form and the fibration $\pi$ from the rational open book are well-behaved in the sense of Definition 4.5 in \cite{G}. The $\s^1$-invariant contact form $\alpha$ we construct in the proof fulfills this requirement (see Remark 4.3 for more details). 

We will continue to use the notations introduced in the proof of Lemma 3.2 regarding the rational open book $(B,\pi)$, except for the notation $\phi$ of the monodromy map being replaced by $h$. Furthermore, for simplicity,
we shall assume that there is only one binding component without loss of generality. Hence in all the notations we shall drop the index $i$.

We begin by identifying $M$ with the mapping torus $\Sigma\times [0,2\pi]/(x,2\pi)\sim (h(x),0)$ with a
solid torus $N$ glued to the boundary. We shall first define $\alpha$ on the mapping torus, and then extend it to the solid torus $N$. 

Let $t^\prime$ denote the coordinate on the interval $[0,2\pi]$, 
and let a neighborhood of $\partial\Sigma$ in 
$\Sigma$ be given with a coordinate system $(s,\phi^\prime)$ where $s\in (1-\epsilon,1]$ and 
$\phi^\prime\in\R/2\pi\Z$. With this understood, we pick a smooth $1$-form $\beta$ on $\Sigma$
such that (1) $h^\ast \beta=\beta$, (2) $d\beta>0$ on $\Sigma$, and (3) $\beta=sd\phi^\prime$ in
a neighborhood of $\partial\Sigma$. Then we set
$$
\alpha:=Cdt^\prime+\beta \mbox{ (for some $0<C\in\R$ to be determined) }
$$
which descends to a $1$-form on the mapping torus $\Sigma\times [0,2\pi]/(x,2\pi)\sim (h(x),0)$, because 
of the condition $h^\ast \beta=\beta$. Furthermore, it is clear that $\alpha$ is a contact form, and
is invariant under the translations in $t^\prime$, thus is $\s^1$-invariant. 

Before we extend $\alpha$ to the solid torus $N$, we shall make some coordinate change near the
boundary of the mapping torus. First, observe that the mapping torus is defined by taking quotient 
of $\Sigma\times \R$ under the $\Z$-action generated by $\bar{h}(x,t^\prime)=(h(x),t^\prime-2\pi)$,
where on the boundary $\partial\Sigma\times \R$, it is given by 
$$
\bar{h}(\phi^\prime,t^\prime)=(\phi^\prime+\frac{2\pi k}{n},t^\prime-2\pi). 
$$
With this understood, we shall make a change of coordinates by 
$\phi=\phi^\prime+\frac{k}{n}t^\prime$, $t=t^\prime$. It is easy to see that in the new coordinates 
$(\phi,t)$, the $\Z$-action on the boundary is generated by 
$$
\tilde{h}(\phi,t)=(\phi,t-2\pi).
$$

The advantage of $(\phi,t)$ over the old coordinates is that the basis $(\gamma,\tau)$ is given by
$(\frac{\partial}{\partial\phi}, \frac{\partial}{\partial t})$. With this understood, we note that the vector field $X$ generating the $\s^1$-action is given near the boundary by
$$
X=n\frac{\partial}{\partial t^\prime}=n\frac{\partial}{\partial t}+k\frac{\partial}{\partial\phi},
$$
and the contact form $\alpha$ is given by the expression 
$$
\alpha=(C-s\frac{k}{n})dt+sd\phi.
$$

Now we pick a longitude-meridian pair on $\partial N=-T$, denoted by $(\lambda,\mu)$. Then 
$$
\gamma=p\lambda+q\mu,\;\, \tau=p^\prime\lambda+q^\prime\mu, 
$$
where $pq^\prime-qp^\prime=1$. Note that $(p,p^\prime)$ is determined by the rational open book 
$(B,\pi)$, as $\mu=-p^\prime\gamma+p\tau$, but $(q,q^\prime)$ depends on the choice of the longitude 
$\lambda$. Furthermore, if we replace $\lambda$ by $\lambda+m\mu$, the pair $(q,q^\prime)$ 
will be transformed to $(q-mp,q^\prime-mp^\prime)$. 

The choice of a longitude $\lambda$ determines an identification of $N$ with the product 
$\s^1\times D^2$. We shall let $z$ be the coordinate on $\s^1=\R/2\pi\Z$, and let $(r,\theta)$ be the polar coordinates on $D^2$, where $r\leq 1$. Note that the basis $(\lambda,\mu)$ is given by 
$(\frac{\partial}{\partial z}, \frac{\partial}{\partial \theta})$. With this understood, the relation between
$(\lambda,\mu)$ and $(\gamma,\tau)$ implies that $(\phi,t)$ and $(z,\theta)$ are related by the
following equations
$$
\phi=q^\prime z-p^\prime \theta,\;\; t=-qz+p\theta.
$$

{\bf Case 1:} $kp+np^\prime \leq 0$ (i.e., $B$ is either a fixed component of the $\s^1$-action or $B$ has the opposite orientation of the $\s^1$-orbits). In this case, we claim that one can choose a longitude $\lambda$ such that the following conditions are satisfied:
$$
q<0 \mbox{ and } kq+nq^\prime>0.
$$
To see this, recall that if we change a given longitude $\lambda$ to $\lambda+m\mu$, the pair
$(q,q^\prime)$ is changed to $(q-mp,q^\prime-mp^\prime)$. With this understood, it is easy to see that $kq+nq^\prime$ is changed to $kq+nq^\prime-m(kp+np^\prime)$. It follows easily that when 
$kp+np^\prime<0$, we can arrange so that $q-mp<0$ and $kq+nq^\prime-m(kp+np^\prime)>0$
by taking $m>0$ sufficiently large. If $kp+np^\prime=0$, we have $p=n$ and $p^\prime=-k$. In this case, 
$$
kq+nq^\prime-m(kp+np^\prime)=-p^\prime q+pq^\prime=1,
$$
which is greater than $0$ for any choice of $m$. Hence the claim.

With the preceding understood, we identify the coordinates $s$ and $r$ by the equation $s=2-r$.
Then the contact form $\alpha$ on $N=\s^1\times D^2$, which is defined near the boundary $r=1$,
is given in the coordinates $(z,r,\theta)$ by the following expression
$$
\alpha=(-qC+(2-r)\frac{kq+nq^\prime}{n})dz+(Cp+(r-2)\frac{kp+np^\prime}{n})d\theta.
$$
We also observe that in the coordinates $(z,r,\theta)$, the vector field $X$ which generates the
$\s^1$-action is given by the following expression
$$
X=(kp+np^\prime)\frac{\partial}{\partial z}+(kq+nq^\prime)\frac{\partial}{\partial \theta}.
$$
Finally, the tangent planes of the pages of $(B,\pi)$ in the solid torus $N$ are
spanned by the pair of vector fields $(p\frac{\partial}{\partial z}+q\frac{\partial}{\partial \theta},
\frac{\partial}{\partial r})$ as a positively oriented basis.  

With the preceding understood, we shall extend $\alpha$ to the rest of $N$ by an expression 
$$
\alpha=f(r)dz+g(r)d\theta
$$
for some appropriate choice of smooth functions $f(r),g(r)$, for $0\leq r\leq 1$, such that near $r=1$,
$$
f(r)=-qC+(2-r)\frac{kq+nq^\prime}{n},\; g(r)=Cp+(r-2)\frac{kp+np^\prime}{n}.
$$
It is easy to check that, for any $\alpha=f(r)dz+g(r)d\theta$, $L_X\alpha=0$ so that $\alpha$ is 
$\s^1$-invariant. 

The constraints on the pair of functions $(f(r),g(r))$ are given below:
\begin{itemize}
\item[{(i)}] $(\frac{g}{f})^\prime>0$, $f(0)>0$, $g(0)=g^\prime(0)=0$, and $g^\prime{^\prime}(0)>0$, which imply that $\alpha$ defines a positive contact structure on $N$ and $\alpha>0$ on $B$.
\item[{(ii)}] $f^\prime p+g^\prime q<0$, which is equivalent to $d\alpha>0$ on the pages. 
\end{itemize}

Note that geometrically, condition (i) means that the curve $(x,y)=(f(r),g(r))$ in the $xy$-plane will
rotates strictly counterclockwise as $r$ increases from $0$ towards $1$, starting at some point on the positive $x$-axis. On the other hand, condition (ii) means that the angle between the tangent vector 
$(f^\prime,g^\prime)$ of the curve and the vector $(-p,-q)$ is less than $\pi/2$. By the assumption
that $p>0$ and $q<0$, this holds true as long as $f^\prime<0$ and $g^\prime\geq 0$. If $g^\prime<0$,
we need to require that the slope of the vector $(f^\prime,g^\prime)$ is bounded by $\frac{p}{-q}$, i.e.,
$$
\frac{g^\prime}{f^\prime}<\frac{p}{-q}. 
$$
For this, note that near $r=1$, $f^\prime=-(kq+nq^\prime)$ and $g^\prime=kp+np^\prime$, and the following inequality holds
$$
0<\frac{kp+np^\prime}{-(kq+nq^\prime)}<\frac{p}{-q}. 
$$
It is easy to see that such a curve $(x,y)=(f(r),g(r))$ exists in the $xy$-plane. This finishes the extension of the contact form $\alpha$ over the solid torus $N$. Finally, we note that the contact structure 
$\xi:=\ker\alpha$ fails to be transverse precisely when there is a $0<r_0<1$ such that
$$
\alpha(X)(r_0)=f(r_0)(kp+np^\prime)+g(r_0)(kq+nq^\prime)=0, 
$$
or equivalently, the slope $\frac{g}{f}$ at $r_0$ equals $-\frac{kp+np^\prime}{kq+nq^\prime}$. For this
we observe that
$$
\frac{kp+np^\prime}{-(kq+nq^\prime)}<\frac{g(1)}{f(1)}<\frac{p}{-q}. 
$$
Since the slope $\frac{g}{f}$ is strictly increasing from $0$ at $r=0$ to $\frac{g(1)}{f(1)}$ at $r=1$,
it is easy to see that if $kp+np^\prime=0$, then $\frac{g}{f}>0$ for all $0<r\leq 1$, and if 
$kp+np^\prime<0$, then there must be one and a unique $0<r_0<1$ such that 
$\frac{g}{f}(r_0)=-\frac{kp+np^\prime}{kq+nq^\prime}$. This shows that the dividing set of $\xi$ is precisely the circle $r=r_0$. 

\vspace{2mm}

{\bf Case 2:} $kp+np^\prime >0$ (i.e., $B$ is an $\s^1$-orbit with the same orientation). By a similar argument, we can show that in this case, one can choose a longitude $\lambda$ such that the following is true
$$
q<0 \mbox{ and } kq+nq^\prime<0.
$$
With this understood, we continue to identify $s$ and $r$ by the equation $s=2-r$, and with that 
the contact form $\alpha$ takes the following expression near $r=1$:
$$
\alpha=(-qC+(2-r)\frac{kq+nq^\prime}{n})dz+(Cp+(r-2)\frac{kp+np^\prime}{n})d\theta.
$$
Now we choose a sufficiently large $C>0$ in the definition of the contact form 
$\alpha=Cdt^\prime+\beta$; in particular, $\frac{Cnp}{kp+np^\prime}>1$. Then 
we extend $\alpha$ to the solid torus $N$ by $\alpha=f(r)dz+g(r)d\theta$, where
$$
f(r)=-qC-\rho(r)\frac{kq+nq^\prime}{n} \mbox{ and } g(r)= Cp+\rho(r)\frac{kp+np^\prime}{n}
$$
for some smooth function $\rho(r)$, for $0\leq r\leq 1$, where 
$$
\rho(r)=-\frac{Cnp}{kp+np^\prime}+r^2 \mbox{ near } r=0, \mbox{ and } 
\rho(r)=-2+r  \mbox{ near } r=1,
$$
and $\rho^\prime>0$ for $0\leq r\leq 1$. Note that the last condition is possible because of the assumption $\frac{Cnp}{kp+np^\prime}>1$. 

Near $r=0$, $f(r)=\frac{Cn}{kp+np^\prime}-r^2 \frac{kq+nq^\prime}{n}$ and 
$g(r)=r^2\frac{kp+np^\prime}{n}$, which obeys $f(0)>0$, $g(0)=g^\prime(0)=0$, and $g^\prime{^\prime}(0)>0$ in condition (i) (see Case 1). To see that $(\frac{g}{f})^\prime>0$, i.e., the slope $\frac{g}{f}$ is strictly increasing, we note that the curve $(x,y)=(f(r),g(r))$ in the $xy$-plane is the straight line 
passing through the point $(-qC,pC)$ and having a slope $\frac{kp+np^\prime}{-(kq+nq^\prime)}>0$. 
On the other hand, in the present case we have the inequality 
$$
\frac{kp+np^\prime}{-(kq+nq^\prime)}>\frac{p}{-q}>0,
$$
from which it follows easily that the slope $\frac{g}{f}$ is strictly increasing, as both
$f$ and $g$ are strictly increasing. Thus condition (i) is verified. As for (ii), we note that the
tangent vector $(f^\prime,g^\prime)$ has a constant slope $\frac{kp+np^\prime}{-(kq+nq^\prime)}$
and pointing upward. It is easy to see that it forms an angle $<\pi/2$ with the vector $(-p,-q)$
because of the inequality $\frac{kp+np^\prime}{-(kq+nq^\prime)}>\frac{p}{-q}$ and the fact that
$p>0$ and $q<0$. Hence (ii) is verified as well. Finally, we note that the slope $\frac{g}{f}$ is always 
less than $-\frac{kp+np^\prime}{kq+nq^\prime}$, which implies that the contact structure 
$\xi:=\ker\alpha$ is transverse in this case. In fact, a simple calculation shows that in Case 2,
one has $\alpha(X)=Cn>0$, a positive constant. 

The proof of Theorem 4.1 is complete. 

\begin{remark}
We note that in our extension of the contact form $\alpha$ to the solid torus $N$, it takes the form 
$\alpha=f(r)dz+g(r)d\theta$, where $f(r)=E+Fr^2$, $g(r)=Gr^2$ for some constants $E,F,G$. Furthermore, 
in both Case 1 and Case 2, $E>0$ and $G>0$, while $F<0$ in Case 1 and $F>0$ in Case 2. 
In the terminology of Gay (cf. \cite{G}, Definition 4.5), such a contact form $\alpha$ is said to be {\bf well-bahaved} with respect to the coordinates $(r,\theta,z)$, as the Reeb vector field $R_\alpha$ takes the form
$R_\alpha=A\frac{\partial}{\partial \theta}+B\frac{\partial}{\partial z}$, where $A=-\frac{F}{EG}$ and 
$B=\frac{1}{E}$, with $B>0$. On the other hand, the closed $1$-form $d\pi$ on $M\setminus  B$, where $\pi: M\setminus B\rightarrow \s^1$ is the fibration from the rational open book, takes the form $d\pi=p d\theta-qd z$ where $p>0$ is the multiplicity. It is easy to see that the closed $1$-form $\alpha^0:=cd\pi=C d\theta+D dz$, for a sufficiently large $c>0$, is also {\bf well-behaved} in the sense of Definition 4.5 in \cite{G}, as $C=cp>0$ and 
$\alpha^0(R_\alpha)=AC+BD>1$. The existence of such a pair of forms $(\alpha,\alpha^0)$ is
required in order to perform Gay's construction in \cite{G}. The quadruple $(A,B,C,D)$ is called the {\bf structural data} for $(\alpha,\alpha^0)$. If furthermore, we also have $A=B$, $D>0$ and $AD>1$, then the pair 
$(\alpha,\alpha^0)$ is called {\bf prepared for surgery} with respect to the coordinates $(r,\theta,z)$. 
\end{remark}

As a consequence, we can apply the construction of Gay in \cite{G} to the present situation. 
More concretely, assume we are given a rational open book $(B,\pi)$ with periodic monodromy on $M$. Let $\alpha$ be the contact form and $\xi=\ker\alpha$ the contact structure constructed in Theorem 4.1, and we equip the product $[0,1]\times M$ with the symplectic form $d(e^t\alpha)$. 
Then according to \cite{G}, if one attaches\footnote{Strictly speaking, one needs to first apply an enlargement procedure so that the contact boundary is under the condition of {\bf ``prepared for surgery"}, see \cite{G}, Corollary 4.8. But we shall suppress this technical detail as it does not effect our discussions here, which are topological in nature.}
to $\{1\}\times M$ a symplectic $2$-handle along each binding component $B_i$ with a framing $F_i>0$ relative to the page framing from the open book $(B,\pi)$ (in general $F_i\in\Q$), 
then the resulting $4$-manifold $Z$ is symplectic with two concave contact boundary components,
one being $(M,\xi)$ (here $M$ is identified with $\{0\}\times M$), where in the other, denoted by 
$(M^\prime,\xi^\prime)$, the contact structure $\xi^\prime$ is supported by a rational open book 
$(B^\prime,\pi^\prime)$ such that $\pi^\prime: M^\prime\setminus B^\prime \rightarrow \s^1$ is diffeomorphic to $\pi: M\setminus B \rightarrow \s^1$ but the pages of $\pi^\prime,\pi$ have the opposite orientations. In particular, $(B^\prime,\pi^\prime)$ is also of periodic monodromy. Moreover, the binding $B^\prime$ is the union of the ascending spheres of the $2$-handles attached to 
$\{1\}\times M$ (cf.  \cite{G}, Addendum 5.1). So $Z$ is a relative handlebody of $2$-handles built on
$M$. (We remark that since the contact boundaries of $Z$ are concave, $\partial Z=-M\sqcup -M^\prime$.)

Our goal in the next theorem is to describe the generalized Seifert fibration structure on $M^\prime$ determined by $(B^\prime,\pi^\prime)$. For simplicity, we restrict ourselves to the case where $M$
is a Seifert manifold. 

\begin{theorem}
Assume $M$ is a Seifert manifold, i.e., the $\s^1$-action induced by the rational open book 
$(B,\pi)$ is fixed-point free, and as a Seifert manifold, $M$ is given by 
$$
M=M(g_0; \{(\alpha_i,\beta_i)\}, \{(\alpha_j,\beta_j)\}),
$$
where $0<\beta_j<\alpha_j$ for each $j$, and $\{(\alpha_i,\beta_i)\}$ are the Seifert invariants at
the binding components $B_i$, subject to the following convention: if $\alpha_i<0$ for some $i$, 
then $B_i$ has the opposite orientation of the fiber.  Let $n$ be the order of the monodromy of 
$(B,\pi)$ and let $p_i$ be the multiplicity at the binding component $B_i$. 
Then $M^\prime$ is given by
$$
M^\prime=M(g_0; \{(\bar{\alpha}_i,\bar{\beta}_i)\}, \{(\alpha_j,-\beta_j)\}),
$$
where $\bar{\alpha}_i=\frac{n}{p_i}-F_i\alpha_i$ and 
$\bar{\beta}_i=F_i\beta_i+\frac{1}{\alpha_i}-\frac{n\beta_i}{p_i\alpha_i}$, subject to the convention that
if $\bar{\alpha}_i<0$, then the corresponding binding component $B_i^\prime$ of the rational open 
book $(B^\prime,\pi^\prime)$ has the opposite orientation of the fiber. Finally, $\bar{\alpha}_i=0$
is also allowed here.  
\end{theorem}

\begin{proof}
The proof is simply an application of Lemma 3.2.

To this end, we begin by recalling the standard description of $(B,\pi)$ used in the proof of Lemma 3.2. First, the periodic monodromy map $\phi: \Sigma\rightarrow \Sigma$ determines a
set of integers $\{(n_j,c_j)\}$, $\{(n,c_i)\}$, where $0<c_j<n_j$, $gcd(n_j,c_j)=1$, 
$0<c_i<n$, $gcd(n,c_i)=1$, and moreover, for each $i$, there are integers $k_i,l_i$, uniquely determined by $c_i$ by the conditions $0<k_i<n$, $nl_i-k_ic_i=1$. Finally, 
$$b:=\sum_j\frac{c_j}{n_j}+\sum_i\frac{c_i}{n}\in\Z.$$

To describe the gluing of the regular neighborhood $N_i$ of $B_i$ to the mapping torus 
$\Sigma\times [0,1]/ (x,1)\sim (\phi(x),0)$, we fix a canonical, positively oriented basis 
$(\gamma_i,\tau_i)$ of $H_1(T_i)$, where $T_i:=(\partial \Sigma)_i \times [0,1]/ (x,1)\sim (\phi(x),0)$,
and $\gamma_i:=(\partial \Sigma)_i \times \{0\}$. With this understood, the meridian $\mu_i$ of
$\partial N_i=-T_i$ is given by $\mu_i=-p_i^\prime\gamma_i +p_i\tau_i$ for a pair of integers 
$(p_i,p_i^\prime)$, where $p_i>0$ is the multiplicity at $B_i$.

With $(p_i,p_i^\prime)$ understood, we note that $\alpha_i=k_ip_i+np_i^\prime$. On the other hand,
we set 
$$
\beta_i^\prime:=l_ip_i+c_ip_i^\prime.
$$
Then by Lemma 3.2, $M=M(g_0; (1,-b), \{(\alpha_j,\beta_j)\}, \{(\alpha_i,\beta_i^\prime)\})$
where $(\alpha_j,\beta_j)=(n_j,c_j)$ for each $j$.  Observe that $\sum_i\frac{\beta_i}{\alpha_i}=\sum_i\frac{\beta_i^\prime}{\alpha_i}-b$.

With the preceding understood, we shall apply Lemma 3.2 to the rational open book 
$(B^\prime,\pi^\prime)$ to obtain the generalized Seifert fibration structure on $M^\prime$. To this end,
we identify $\pi^\prime: M^\prime\setminus B^\prime \rightarrow \s^1$ with $\pi: M\setminus B \rightarrow \s^1$ but with the orientation of the pages $\Sigma$ reversed. We shall describe the rational open book $(B^\prime,\pi^\prime)$ in the same fashion, using the bar-version of the notations. Since the page orientation is reversed, the set of integers $\{(n_j, \bar{c}_j)\}$, $\{(n,\bar{c}_i)\}$ determined by the monodromy map are given by $\bar{c}_j=n_j-c_j$, $\bar{c}_i=n-c_i$. 
Consequently, $\bar{k}_i=n-k_i$, $\bar{l}_i=n+l_i-k_i-c_i$, and 
$$
\bar{b}=\sum_j\frac{\bar{c}_j}{n_j}+\sum_i\frac{\bar{c}_i}{n}=x+y-b,
$$
where $x,y$ are the number of indices of $j$ and $i$ respectively. Finally, we note that the corresponding basis $(\bar{\gamma}_i,\bar{\tau}_i)$ is related to $(\gamma_i,\tau_i)$ by the 
equations $(\bar{\gamma}_i,\bar{\tau}_i)=(-\gamma_i, \tau_i+\gamma_i)$. 

In order to obtain the generalized Seifert invariants of $M^\prime$, it remains to determine the
corresponding pair of integers $(\bar{p}_i,\bar{p}_i^\prime)$. To this end, 
we fix a longitude $\lambda_i$ along $B_i$, and write $\gamma_i=p_i\lambda_i+q_i\mu_i$, 
$\tau_i=p_i^\prime\lambda_i+q^\prime_i\mu_i$ for some $q_i,q_i^\prime$ such that 
$p_iq_i^\prime-p_i^\prime q_i=1$ (note that $\lambda_i\cdot\mu_i=1$). 
Suppose with respect to $(\lambda_i,\mu_i)$, the framing of the $2$-handle attached along $B_i$
is $f_i$. Then the framing $f_i$ and the framing relative to the page framing $F_i$ are related by
the equation $f_i=F_i+\frac{q_i}{p_i}$. On the other hand, as a boundary component of $Z$, 
$-M^\prime$ is obtained by removing the neighborhood $N_i$ of $B_i$ in $\{1\}\times M$ and then
gluing back a $D^2\times \s^1$ along $\partial N_i$ such that 
$$
\mu_i^\prime=-\lambda_i-f_i\mu_i,\;\; \lambda_i^\prime=\mu_i,
$$
where $(\mu_i^\prime,\lambda_i^\prime)$ is the meridian-longitude pair of the solid torus 
$D^2\times \s^1$. 

With the preceding understood, we observe that the reversing of the orientation of the pages $\Sigma$ 
results a change of the orientations of $(\mu_i^\prime,\lambda_i^\prime)$ to the opposite. Consequently, if we let $\bar{\mu}_i=-\mu_i^\prime$ and $\bar{\lambda}_i=-\lambda_i^\prime$, 
then we shall determine the pair $(\bar{p}_i,\bar{p}_i^\prime)$ through the equation
$\bar{\mu}_i=-\bar{p}_i ^\prime\bar{\gamma}_i+\bar{p}_i \bar{\tau}_i$. On the other hand, 
the reversing of the orientation of the pages $\Sigma$ also results a change of the orientation of
$M$ to the opposite, so that the intersection numbers $\gamma_i\cdot\tau_i$ and 
$\lambda_i\cdot \mu_i$ are changed by a sign, i.e., we now have 
$\gamma_i\cdot\tau_i=\lambda_i\cdot \mu_i=-1$. (Note that $\bar{\gamma}_i\cdot \bar{\tau}_i=1$
is unchanged.) Finally, note that $\mu_i^\prime=-\lambda_i-f_i\mu_i$ implies that 
$\bar{\mu}_i=\lambda_i+f_i\mu_i$. With this understood, and appealing to the relation
$(\bar{\gamma}_i,\bar{\tau}_i)=(-\gamma_i, \tau_i+\gamma_i)$, we compute $\bar{p}_i$, 
$\bar{p}_i^\prime$ through intersection numbers:
$$
\bar{p}_i=\bar{\gamma}_i\cdot \bar{\mu}_i=p_if_i-q_i, \;\;\; 
\bar{p}_i^\prime=-\bar{\mu}_i\cdot \bar{\tau}_i=(q_i+q_i^\prime)-f_i(p_i+p_i^\prime).
$$
Note that $\bar{p}_i=p_i F_i$. In particular, $F_i>0$ implies $\bar{p}_i>0$.

Set $\bar{\alpha}_i:=\bar{k}_i\bar{p}_i+n\bar{p}_i^\prime$, 
$\bar{\beta}_i^\prime:=\bar{l}_i\bar{p}_i+\bar{c}_i\bar{p}_i^\prime$. Then with the relations
$$
F_i=f_i-\frac{q_i}{p_i},\;\; p_iq_i^\prime-p_i^\prime q_i=1, \;\; \alpha_i=k_ip_i+np_i^\prime, \;\;\beta_i^\prime=l_ip_i+c_ip_i^\prime, 
$$
a straightforward calculation shows that
$$
\bar{\alpha}_i=\frac{n}{p_i}-F_i\alpha_i, \;\;\;
\bar{\beta}_i^\prime=F_i\beta_i^\prime+\frac{1}{\alpha_i}
-\frac{n\beta_i^\prime}{p_i\alpha_i}+\bar{\alpha}_i.
$$
Now if we write $\beta_i^\prime=\beta_i+b_i\alpha_i$, then
$$
\bar{\beta}_i^\prime=F_i\beta_i+\frac{1}{\alpha_i}
-\frac{n\beta_i}{p_i\alpha_i}+ (1-b_i)\bar{\alpha}_i.
$$
With $\bar{b}=x+y-b$, $\sum_ib_i=b$, and $(n_j,c_j)=(\alpha_j,\beta_j)$, it follows from Lemma 3.2 that
$$
M^\prime=M(g_0; (1,-\bar{b}), \{(n_j,\bar{c}_j)\}, \{(\bar{\alpha}_i,\bar{\beta}_i^\prime)\})=
M(g_0; \{(\bar{\alpha}_i,\bar{\beta}_i)\}, \{(\alpha_j,-\beta_j)\}),
$$
where $\bar{\beta}_i=F_i\beta_i+\frac{1}{\alpha_i}-\frac{n\beta_i}{p_i\alpha_i}$. This finishes the proof of Theorem 4.4.

\end{proof}

\begin{remark}
An important question is: given $M=M(g_0; \{(\alpha_i,\beta_i)\}, \{(\alpha_j,\beta_j)\})$ and $(B,\pi)$, what are the possible values of $\{(\bar{\alpha}_i,\bar{\beta}_i)\}$ that can be realized by choosing
appropriate framings $\{F_i>0\}$? To this end, we observe that $\{(\bar{\alpha}_i,\bar{\beta}_i)\}$ and
the original pairs $\{(\alpha_i,\beta_i)\}$ obey the equations
$$
\bar{\alpha}_i\beta_i+\alpha_i\bar{\beta}_i=1. 
$$
Furthermore, since $\alpha_i\neq 0$, $\bar{\beta}_i$ is uniquely determined by $\bar{\alpha}_i$ and the pair $(\alpha_i,\beta_i)$. On the other hand, observe that 
$F_i=\frac{1}{\alpha_i}(\frac{n}{p_i}-\bar{\alpha}_i)$, where $p_i>0$ is the multiplicity at $B_i$. It follows easily that the condition $F_i>0$ is
equivalent to (i) and (ii) below:
\begin{itemize}
\item [{(i)}] If $\alpha_i>0$, then $\bar{\alpha}_i<\frac{n}{p_i}$.
\item [{(ii)}] If $\alpha_i<0$, then $\bar{\alpha}_i>\frac{n}{p_i}$. 
\end{itemize}

Note that in the context of Theorem 3.3, we have $e(M)=-\frac{p}{n\alpha}$. It is easy to see that
in the special case where the binding $B$ is connected,  
(i) is equivalent to either $\bar{\alpha}\leq 0$ or $e(M)>-\frac{1}{\alpha\bar{\alpha}}$, and (ii) is
equivalent to $e(M)>-\frac{1}{\alpha\bar{\alpha}}$. Similarly in the context of Theorem 3.4 
(with Remark 3.5), (i) and (ii) 
are equivalent to the following conditions: if $\alpha_i>0$, then either $\bar{\alpha}_i\leq 0$ or 
$\frac{\beta_i}{\alpha_i}+b_i-\frac{c_i}{n}=\frac{p_i}{\alpha_i n}<\frac{1}{\alpha_i\bar{\alpha}_i}$, and if $\alpha_i<0$,
$\frac{\beta_i}{\alpha_i}+b_i-\frac{c_i}{n}=\frac{p_i}{\alpha_i n}<\frac{1}{\alpha_i\bar{\alpha}_i}$. 
These are the only constraints on $\bar{\alpha}_i$, derived from $F_i>0$. 
\end{remark}

In the following lemma, we shall express the framing of the $2$-handle attached to the binding component $B_i$ in Theorem 4.4 in a more convenient way, not in terms of the page framing $F_i$.
Instead, we express it in terms of the slope of the orbits of the $\s^1$-action on $M$, the multiplicity
$\alpha_i$ of $B_i$ (the descending sphere) as a fiber in the Seifert fibration, and the multiplicity 
$\bar{\alpha}_i$ of the ascending sphere as a fiber in the Seifert fibration of $M^\prime$. We should point out that the convention for the signs of the multiplicities $\alpha_i$ and $\bar{\alpha}_i$ is the same as specified in Theorem 4.4.

\begin{lemma}
Suppose $(\lambda_i,\mu_i)$ is a longitude-meridian pair for a neighborhood of the binding component $B_i$ with respect to which the orbits of the $\s^1$-action have slope $s_i$. Then
with respect to $(\lambda_i,\mu_i)$, the framing of the $2$-handle attached to $B_i$ is 
$s_i-\frac{\bar{\alpha}_i}{\alpha_i}$. 
\end{lemma}

\begin{proof}
Suppose $\gamma_i=p_i\lambda_i+q_i\mu_i$, $\tau_i=p_i^\prime\lambda_i+q^\prime_i\mu_i$ for
some $q_i,q_i^\prime$ such that $p_iq_i^\prime-p_i^\prime q_i=1$. Then from the proof of 
Lemma 3.2, the orbits of the $\s^1$-action are given by 
$$
H_i=k_i\gamma_i+n\tau_i=(k_ip_i+np_i^\prime)\lambda_i+(k_iq_i+nq_i^\prime)\mu_i,
$$
which implies that the slope 
$$
s_i=\frac{1}{\alpha_i}(k_iq_i+nq_i^\prime)
=\frac{1}{\alpha_ip_i}((\alpha_i-np_i^\prime)q_i+n(1+p_i^\prime q_i))=\frac{q_i}{p_i}
+\frac{n}{\alpha_i p_i}, 
$$
as $\alpha_i=k_ip_i+np_i^\prime$. On the other hand, with respect to $(\lambda_i,\mu_i)$,
the framing of the $2$-handle is given by $f_i=F_i+\frac{q_i}{p_i}$. With 
$F_i=\frac{1}{\alpha_i}(\frac{n}{p_i}-\bar{\alpha}_i)$, we obtain $f_i=s_i-\frac{\bar{\alpha}_i}{\alpha_i}$.

\end{proof}

\begin{remark}
The construction of the symplectic $4$-manifold $Z$ requires a choice of a rational open book 
on $M$, which is not necessarily unique even if we fix the binding components $B_i$. However, 
we shall point out that as long as the other boundary component of $Z$, i.e. $M^\prime$, is 
specified, the diffeomorphism class of $Z$ does not depend on the choice of the rational open book (except its binding being fixed), because by Lemma 4.6, the framings of the $2$-handles, relative to the framing of the regular fibers of the Seifert fibration on $M$, are completely determined by $M^\prime$. With this understood, we shall also point out that the symplectic structure on 
$Z$ (from Gay's construction) does depend on the choice of the rational open book on $M$, as we shall see
in the next section (cf. Lemma 5.1).

\end{remark}

\section{The sign of the canonical class}
Let $Z$ be a symplectic $4$-manifold with concave contact boundaries $(M,\xi)$ and 
$(M^\prime,\xi^\prime)$, which is a relative handlebody built on $M$, constructed using Gay's construction in \cite{G}. Here we do not require $M$ to be a Seifert manifold, nor $(B,\pi)$ to be of periodic monodromy, but we assume both $M$ and $M^\prime$ are $\Q$-homology $3$-spheres.
We denote by $\omega_Z$ the symplectic structure on $Z$, and $c_1(K_Z)$ the corresponding canonical class. Since $M$ and $M^\prime$ are $\Q$-homology $3$-spheres, we note that 
$c_1(K_Z)$ is naturally a class in $H^2(Z,\partial Z,\R)$. On the other hand, the de Rham cohomology class $[\omega_Z]$ lies in $H^2(Z,\R)$. With this understood, one has the pairing $c_1(K_Z)\cdot [\omega_Z]\in H^4(Z,\partial Z,\R)=\R$, where we identify $H^4(Z,\partial Z,\R)$ with $\R$ by integrating over the fundamental class $[Z]\in H_4(Z,\partial Z,\Z)$. We are interested in determining the sign of $c_1(K_Z)\cdot [\omega_Z]$ due to its topological significance.

To this end, let $B_i$ be the binding components of the rational open book $(B,\pi)$ on $M$, and let
$D_i$ be the core disk of the symplectic $2$-handle attached to $\{1\}\times M$ along $B_i$. Note that each $D_i$ is symplectic, thus is naturally oriented and defines a class $[D_i]\in H_2(Z,\partial Z,\Z)$.

\begin{lemma}
Under the Poincar\'{e} durality $H^2(Z,\R)=H_2(Z,\partial Z,\R)$, 
$$
[\omega_Z]=\Omega\cdot \sum_i p_i [D_i]
$$
for some constant $\Omega>0$, where $p_i$ is the multiplicity of $(B,\pi)$ at $B_i$.
\end{lemma}

\begin{proof}
The classes $[D_i]$ form a basis of $H_2(Z,\partial Z,\Z)$, hence 
$[\omega_Z]=\sum_i c_i [D_i]$ for some $c_i\in \R$. It remains to show that for some constant 
$\Omega>0$, $c_i=\Omega p_i$ for each $i$.

To this end, let $D_i^\prime$ be the co-core disk of the symplectic $2$-handle attached to 
$\{1\}\times M$ along $B_i$, which is also symplectic and therefore naturally oriented. The boundary
$\partial D_i^\prime$ is the ascending sphere of the $2$-handle, which is $-B_i^\prime$ where $B_i^\prime$ is the corresponding binding component of the rational open book 
$(B^\prime,\pi^\prime)$ on $M^\prime$ supporting the contact structure $\xi^\prime$. Since $M^\prime$ is a $\Q$-homology $3$-sphere, there exists an integer $m_i>0$ such that $m_i B_i^\prime$ bounds an oriented surface $C_i$ in $M^\prime$. It follows easily that $m_i D_i^\prime\cup C_i$ defines a homology class 
$[m_i D_i^\prime\cup C_i]\in H_2(Z,\Z)$. With this understood, we note that 
$$
[\omega_Z]\cdot [m_i D_i^\prime\cup C_i]=(\sum_j c_j [D_j])\cdot [m_i D_i^\prime\cup C_i]=c_im_i,
$$
as $D_i$ and $D_i^\prime$ intersect transversely and positively at a single point, $D_j$ and $D_i^\prime$ are disjoint if $j\neq i$, and $D_j$ and $C_i$ are disjoint for any $j$. On the other hand, 
$$
[\omega_Z]\cdot [m_i D_i^\prime\cup C_i]=m_i \int_{D_i^\prime}\omega_Z+\int_{C_i} \omega_Z=
m_i(\int_{D_i^\prime}\omega_Z+ \int_{B_i^\prime} \alpha^\prime), 
$$ 
where $\alpha^\prime$ is some contact form on $M^\prime$ defining $\xi^\prime$ such that
$\omega_Z|_{M^\prime}=d\alpha^\prime$, so that $\int_{C_i} \omega_Z=m_i  \int_{B_i^\prime} \alpha^\prime$ by Stokes' theorem. It follows easily that $c_i=\int_{D_i^\prime}\omega_Z+ \int_{B_i^\prime} \alpha^\prime$.

In order to further evaluate $c_i$, we recall some relevant details about the symplectic 
$2$-handles in Gay's construction, see the proof of Proposition 4.6 in \cite{G} for a complete discussion. First, a contact pair $(\alpha^{+},\alpha^{-})$ on $M$ is given, where $\alpha^{+}$ 
is the contact form on $M$ and $\alpha^0:=\alpha^{+}+\alpha^{-}$ is given by $c\cdot d\pi$, with 
$\pi: M\setminus B\rightarrow\s^1$ being the fibration from the rational open book $(B,\pi)$ 
and $c>0$ a sufficiently large constant (see the proof of Theorem 1.2 in \cite{G})). Fixing
attention at a binding component $B_i$, there is a neighborhood of $B_i$ with coordinates 
$(r,\mu,\lambda)$, where $(\mu,\lambda)$ is a meridian-longitude pair, such that the Reeb 
vector field $R_{\alpha^{+}}=A\partial_\mu+ B\partial_\lambda$ and $\alpha^0=Cd\mu+D d\lambda$, for some constants $A,B,C,D$ (called {\bf structural data}) obeying $B,C>0$. It is easy to see that $B,C$ do not depend on the choice of the longitude $\lambda$ (i.e., the framing of $B_i$). Furthermore, one can always arrange such that $A=B$, $D>0$ and $AD>1$ hold true for the structural data, in which case $B_i$ is said to be ``prepared for surgery" (cf. Corollary 4.8 in \cite{G}). 

It should be noted that the notion of contact pair is associated with the so-called {\bf dilation-contraction pair} of vector fields $(V^{+},V^{-})$ of a symplectic structure $\omega$, where $V^{+},V^{-}$ obey
$$
L_{V^{\pm}}\omega =\pm\omega, \;\; \omega(V^{+},V^{-})=0. 
$$
More precisely, given $\omega$ and $(V^{+},V^{-})$, one can canonically produce a contact pair 
$(\alpha^{+},\alpha^{-})$, where $\alpha^{\pm}:= i_{V^{\pm}}\omega$.

With the preceding understood, we now describe the model for a symplectic $2$-handle in Gay's construction. Let $\omega_0:=r_1 dr_1\wedge d\theta_1+ r_2 dr_2\wedge d\theta_2$ be the
standard symplectic form on $\R^4$, in polar coordinates $(r_i,\theta_i)$, $i=1,2$, and let
$f=-r_1^2+r_2^2$. To specify the $2$-handle, one needs to fix a pair of parameters $\epsilon_1,
\epsilon_2$, where $\epsilon_1<0<\epsilon_2$, such that $\epsilon_1$ is determined by 
the structural data by the equation $\epsilon_1=\frac{2}{A}-2D$, and $\epsilon_2$ is chosen to be sufficiently small. With this understood, part of the boundary of the $2$-handle is formed by the level surfaces $f^{-1}(\epsilon_1)$ and $f^{-1}(\epsilon_2)$. In particular, the co-core disc $D_i^\prime$ is given by $r_1=0$, $r_2\leq \sqrt{\epsilon_2}$, and the binding component $B_i^\prime\subset 
f^{-1}(\epsilon_2)$ is given by $r_1=0$, $r_2=\sqrt{\epsilon_2}$, and $B_i^\prime$ is oriented 
by $-d\theta_2$. From this description, it follows easily that in the equation
$c_i=\int_{D_i^\prime}\omega_Z+\int_{B_i^\prime}\alpha^\prime$, one has 
$\int_{D_i^\prime}\omega_Z=\pi\epsilon_2$.

To evaluate $\int_{B_i^\prime}\alpha^\prime$, we note that the contact form 
$\alpha^\prime=-i_{V^{-}}\omega_0$, where $V^{-}$ is the contracting vector field in a 
dilation-contraction pair and is given by the formula
$$
V^{-}= -\frac{1}{2} r_1\partial_{r_1}-(\frac{1}{2} r_2-\frac{C}{r_2})\partial_{r_2}.
$$
(Here in the expression of $V^{-}$, $C$ is the constant from the structural data.) 
It follows easily that $\alpha^\prime=\frac{1}{2} r_1^2 d\theta_1+(\frac{1}{2} r_2^2 -C)d\theta_2$, 
which implies that $\int_{B_i^\prime}\alpha^\prime=-\pi\epsilon_2+2\pi C$. Finally, to relate $C$ with
the multiplicity $p_i$, we observe that, if with respect to the meridian-longitude pair $(\mu,\lambda)$,
the boundary component $(\partial \Sigma)_i$ of the page of the rational open book is given by 
$\gamma_i=p_i\lambda+q_i\mu$ for some $q_i\in\Z$, then $d\pi=p_i d\mu-q_i d\lambda$. It follows easily from this observation that for each $i$, the constant $c_i=2\pi c p_i$,
where $c>0$ is the constant in $\alpha^0=c \cdot d\pi$. This finishes the proof of the lemma.

\end{proof}

As a consequence, we note that $c_1(K_Z)\cdot [\omega_Z]=\Omega \cdot \sum_i c_1(K_Z)[p_i D_i]$.
The next lemma gives a formula for $\sum_i c_1(K_Z)[p_iD_i]$. 

\begin{lemma}
Let $\Sigma$ be the pages of the rational open book $(B,\pi)$ on $M$, and let $p_i>0$ be the multiplicity of $(B,\pi)$ at the binding component $B_i$. Let $F_i>0$ be the framing of the $2$-handle 
attached along $B_i$ relative to the page framing. Then 
$$
\sum_i c_1(K_Z)[p_iD_i]= -\chi(\Sigma)-\sum_i p_i (1+F_i).
$$
\end{lemma}

\begin{proof}
First, there exists an $m>0$ such that $-mK_Z|_M$ is trivial. Moreover, the trivialization of 
$-mK_Z|_M$
is unique as $[M,\s^1]=H^1(M,\Z)=0$. We shall denote by $f_0$ the trivialization of $-mK_Z|_M$,
and compute each $c_1(K_Z)[D_i]$ via $c_1(-mK_Z)(D_i, f_0|_{\partial D_i})$, by the formula
$$
c_1(K_Z)[D_i]=-\frac{1}{m} c_1(-mK_Z)(D_i, f_0|_{\partial D_i}).
$$

Next, there is a canonical identification of $-K_Z|_M$ with the contact structure $\xi$ as oriented 
rank $2$ bundles, where $\xi$ is oriented by $d\alpha$ for a contact form $\alpha$ such that
$d\alpha|_{\Sigma}>0$. This is because near $M$, the symplectic structure $\omega_Z$ on $Z$ 
takes the form of $e^t(dt\wedge \alpha+d\alpha)$, where $\partial_t$ is an inward-pointing normal vector field along $M$. With this understood, there is a $\omega_Z$-compatible almost complex structure $J$ such that $\xi$ is $J$-invariant, and $J(\partial_t)=R_\alpha$, the Reeb vector field associated to $\alpha$. This splits the complex vector bundle $TZ|_M$ as a direct sum of the trivial complex line bundle $Span_\R\{\partial_t,R_\alpha\}$ with the complex line bundle $\xi$, i.e.,
$$
TZ|_M=Span_\R\{\partial_t,R_\alpha\}\oplus \xi, 
$$
and as a consequence, it identifies $-K_Z|_M =\det (TZ|_M)=\xi$. 

We shall be working with $\xi$ instead of $-K_Z|_M$ over $M$. With this understood, it is useful
to observe that if $L$ is a transverse link in $M$ which is oriented by $\alpha$, then $\xi|_L$
is naturally identified with the normal bundle of $L$ in $M$ as oriented rank $2$ bundles. 

With the preceding understood, we observe that $TD_i|_{\partial D_i}=Span_\R\{\partial_t,R_\alpha\}$,
so that the normal bundle $\nu_{D_i}$ of $D_i$ identifies itself with $\xi$ along $\partial D_i$
as oriented rank $2$ bundles. With this understood, let $f_{1,i}$ be the trivialization of $\nu_{D_i}|_{\partial D_i}=\xi|_{\partial D_i}=-K_Z|_{\partial D_i}$ which extends over to a trivialization of 
$\nu_{D_i}$ over $D_i$. Then we have
$$
c_1(-K_Z)(D_i, f_{1,i})=c_1(TZ)(D_i, (\partial_t,f_{1,i}))=e(TD_i)(D_i,\partial_t) + e(\nu_{D_i})(D_i,f_{1,i})=1.
$$
Denoting by $f_{1,i}^m$ the corresponding trivialization of $-mK_Z|_{\partial D_i}$, we obtain 
$$
c_1(-mK_Z)(D_i, f_0|_{\partial D_i})=(f_0-f_{1,i}^m)|_{\partial D_i}+m.
$$

Next, we shall compute the difference of the framings $(f_0-f_{1,i}^m)|_{\partial D_i}\in\Z$. To this end,
we fix a trivialization, denoted by $f_{2,i}$, of the normal bundle of $B_i$ in $M$, which is identified to
$\xi|_{B_i}$, and let $(\lambda_i,\mu_i)$ be the corresponding longitude-meridian pair. Then 
$\gamma_i=(\partial \Sigma)_i=p_i\lambda_i+q_i\mu_i$ for some $q_i\in \Z$, where
$p_i>0$ is the multiplicity of the rational open book $(B,\pi)$ at $B_i$. 
Note that if we denote by $f_i\in\Z$ the framing of the $2$-handle attached along $B_i$ relative to the trivialization $f_{2,i}$ 
of the normal bundle of $B_i$, then the page framing $F_i=f_i-\frac{q_i}{p_i}$. 

To proceed further, we note that $\partial \Sigma=\sqcup_i (\partial \Sigma)_i$ is a positively oriented transverse link, thus $\xi|_{\partial \Sigma}$ is identified with the oriented normal bundle. Let $f_3$ be the trivialization of $\xi|_{\partial \Sigma}$ which corresponds to the zero-framing relative to the Seifert surface $\Sigma$.
Then 
$$
e(\xi)(\Sigma,f_3)=e(T\Sigma)(\Sigma,f_3)=\chi(\Sigma),
$$
because all the singular points of the characteristic foliation of $\xi$ along $\Sigma$ have the $+$ sign, 
as $d\alpha>0$ on $\Sigma$. Denoting by $f_3^m$ the corresponding trivialization of 
$m\xi|_{\partial \Sigma}$, it follows easily that 
$$
(f_0-f_3^m)|_{\partial \Sigma}=e(m\xi)(\Sigma,f_0|_{\partial \Sigma})-e(m\xi)(\Sigma, f_3^m)=-m\chi(\Sigma).
$$

On the other hand, we note that for each $i$, there is an annulus $A_i$ in $M$, with one of the two boundary components being $(\partial \Sigma)_i$ while the other being a $p_i$-fold covering of $B_i$, such that the characteristic foliation of $\xi$ along $A_i$ contains no singular points. 
We parametrize $A_i$ by $[0,1]\times \s^1$, sending $\{0\}\times \s^1$ onto $B_i$ as a 
$p_i$-fold covering, and pull back the bundle $\xi$ to $A_i=[0,1]\times \s^1$. With this understood, note that $(f_0-f_{2,i}^m)|_{\{0\}\times \s^1}=p_i (f_0-f_{2,i}^m)|_{B_i}$ and
$(f_0-f_3^m)|_{(\partial \Sigma)_i}=(f_0-f_3^m)|_{\{1\}\times \s^1}$. Moreover, 
$$
(f_0-f_{2,i}^m)|_{\{0\}\times \s^1}-(f_0-f_3^m)|_{\{1\}\times \s^1}=
m e(\xi)([0,1]\times \s^1,f_{2,i}|_{\{0\}\times \s^1}, f_3|_{\{1\}\times \s^1}).
$$
Since the characteristic foliation of $\xi$ along $A_i$ contains no singular points, it follows easily
that 
$$
e(\xi)([0,1]\times \s^1,f_{2,i}|_{\{0\}\times \s^1}, f_3|_{\{1\}\times \s^1})
=(f_{3,i}-f_{2,i})|_{\{0\}\times \s^1}=q_i,
$$
where $f_{3,i}$ is the lifting of the page framing along $B_i$ under the $p_i$-fold covering map 
$\{0\}\times \s^1\rightarrow B_i$. Consequently,
$$
\sum_i p_i (f_0-f_{2,i}^m)|_{B_i}=-m \chi(\Sigma)+m\sum_i q_i.
$$
Finally, with $\partial D_i=-B_i$, we obtain
$$
\sum_i p_i(f_0-f_{1,i}^m)|_{\partial D_i}
=\sum_i p_i (f_{1,i}^m-f_{2,i}^m)|_{B_i}-\sum_i p_i (f_0-f_{2,i}^m)|_{B_i}=mp_i f_i+m\chi(\Sigma)-
m\sum_i q_i.
$$

Putting things together, we have
$$
\sum_i c_1(K_Z)[p_i D_i]=\sum_i -\frac{p_i}{m} c_1(-mK_Z)(D_i, f_0|_{\partial D_i})
=\sum_i -\frac{1}{m} (p_i(f_0-f_{1,i}^m)|_{\partial D_i}+p_im), 
$$
and with $F_i=f_i-\frac{q_i}{p_i}$, we obtain 
$\sum_i c_1(K_Z)[p_iD_i]= -\chi(\Sigma)-\sum_i p_i (1+F_i)$.

\end{proof}

Now we return to the special case where $M$ is a Seifert manifold and $(B,\pi)$ is of periodic monodromy, as in Theorem 4.4, with $g_0=0$ (as both $M$, $M^\prime$
are $\Q$-homology $3$-spheres). We observe, as a corollary of Lemmas 5.1 and 5.2,  the following formula of $c_1(K_Z)\cdot [\omega_Z]$
which does not involve explicitly the page framings $F_i$. 

\begin{corollary}
For some constant $\Omega>0$, 
$$
c_1(K_Z)\cdot [\omega_Z]=\Omega(n(s+r-2-\sum_{i=1}^s\frac{1}{\alpha_i}-\sum_{j=1}^r\frac{1}{\alpha_j})
-\sum_{i=1}^s p_i(1-\frac{\bar{\alpha}_i}{\alpha_i})),
$$
where $s,r$ are the number of indices of $i,j$ respectively, $n$ is the order of the monodromy of the
rational open book $(B,\pi)$, $p_i$ is the multiplicity at the binding component $B_i$, and $\alpha_i,
\bar{\alpha}_i$ are the multiplicities of the binding components $B_i,B_i^\prime$ as fibers of the
Seifert fibrations on $M,M^\prime$ respectively {\em(}with the sign convention specified in Theorem 4.4{\em)}. 
\end{corollary}

Finally, suppose $W$ and $W^\prime$ are symplectic fillings of $(M,\xi)$ and $(M^\prime,\xi^\prime)$
respectively. Let $X:=W\cup Z\cup W^\prime$ be the closed symplectic $4$-manifold formed by taking the union. Given a symplectic structure on $W$ and $W^\prime$, there is a natural symplectic structure $\omega_X$ on $X$.

\begin{lemma}
If the symplectic structures on $W$ and $W^\prime$ are exact, then 
$$
c_1(K_X)\cdot [\omega_X]=c_1(K_Z)\cdot [\omega_Z]. 
$$
\end{lemma}

\begin{proof}
Since $M,M^\prime$ are $\Q$-homology $3$-spheres, $c_1(K_W)$, $c_1(K_{W^\prime})$ lift
uniquely to a class in $H^2(W, M,\R)$ and $H^2(W^\prime, M^\prime,\R)$ respectively. With this
understood, we have 
$$
c_1(K_X)\cdot [\omega_X]=c_1(K_W)\cdot [\omega_W]+c_1(K_Z)\cdot [\omega_Z]+
c_1(K_{W^\prime})\cdot [\omega_{W^\prime}],
$$
where $\omega_W$, $\omega_{W^\prime}$ are the symplectic structures on $W$ and $W^\prime$
respectively. Note that $c_1(K_W)\cdot [\omega_W]\in H^4(W,M,\R)=\R$ and 
$c_1(K_{W^\prime})\cdot [\omega_{W^\prime}]\in H^4(W^\prime,M^\prime,\R)=\R$. Now if 
$\omega_W$, $\omega_{W^\prime}$ are exact $2$-forms, we have $[\omega_W]=0\in H^2(W,\R)$,
$[\omega_{W^\prime}]=0 \in H^2(W^\prime,\R)$, which implies that 
$c_1(K_X)\cdot [\omega_X]=c_1(K_Z)\cdot [\omega_Z]$.

\end{proof}

We remark that the condition in Lemma 5.4 is fulfilled when $W,W^\prime$ are either a Stein domain or a $\Q$-homology ball. Moreover, by work of Taubes \cite{T}, when $X$ has a small topology,
or $X$ is minimal, one can determine the symplectic Kodaira dimension of $X$ by the sign of
$c_1(K_X)\cdot [\omega_X]$ (cf. \cite{Li}). 

\section{Rational unicuspidal curves with one Puiseux pair}

In this section, we consider symplectic rational unicuspidal curves with one Puiseux pair in a symplectic $4$-manifold. Employing the techniques developed in Sections 2 to 5, we give a proof of Theorem 1.17, in the course of which the contact geometry perspective for Conjecture 1.15 should become clear. We shall also provide the details, with examples and technical lemmas, for the discussions in Remark 1.16. Finally, we point out that as a consequence of Theorem 1.17, the list of rational unicuspidal curves with one Puiseux pair in $\C\P^2$ (cf. \cite{F}) 
gives rise to infinitely many $\Q$-homology $3$-spheres, each equipped with an $\s^1$-invariant contact structure, 
which are filled by a symplectic $\Q$-homology ball (cf. Example 6.7). 

By a symplectic curve in a symplectic $4$-manifold, we always assume that the curve is analytic near the singularities and the symplectic form is K\"{a}hler with respect to the local complex structure near the singularities. Thus we can speak of a symplectic rational unicuspidal curve with one Puiseux pair $(p,q)$, which is simply a singular symplectic sphere with a unique cusp singularity whose link is a $(p,q)$-torus knot $T_{p,q}$, where $1<p<q$ are relative prime. 

We begin with some technical lemmas. Assume $p,q$ such that $1\leq p<q$ and $gcd(p,q)=1$ 
(note that we allow $p=1$ here). Let $C_0$ be the curve in $\C^2$ defined by $z_1=z^p$, 
$z_2=cz^q$. Introducing $f(z_1,z_2)=c^pz_1^q-z_2^p$, it is clear that $C_0=f^{-1}(0)$. 
The curve $C_0$ is invariant under the following $\C^\ast$-action on $\C^2$:
$$
\lambda\cdot (z_1,z_2)=(\lambda^p z_1,\lambda^q z_2), \;\; \lambda\in\C^\ast,
$$
as $f(\lambda\cdot (z_1,z_2))=\lambda^{pq}f(z_1,z_2)$.

The corresponding $\s^1$-action is Hamiltonian with respect to the standard symplectic form
$\omega_0:=\frac{i}{2}(dz_1\wedge d\bar{z}_1+dz_2\wedge d\bar{z}_2)$. If we let $(r_i,\theta_i)$ be the polar coordinates for $z_i$, then the vector field generating 
the $\s^1$-action is given by
$$
X:=p\frac{\partial}{\partial \theta_1}+q\frac{\partial}{\partial \theta_2},
$$
and $i_X\omega_0=-dH$, where the Hamiltonian $H:=\frac{1}{2}(pr_1^2+qr_2^2)$.

\begin{lemma}
There exists a $1$-form $\tau$ on $\C^2$, such that $\omega_0=d\tau$. Moreover, for each
$\lambda>0$, we set $\tau_\lambda:=\frac{1}{\lambda}\tau|_{H^{-1}(\lambda)}$. Then 
\begin{itemize}
\item [{(i)}] $i\tau_\lambda$ is a connection $1$-form for the orbifold principal $\s^1$-bundle 
$H^{-1}(\lambda)\rightarrow H^{-1}(\lambda)/\s^1$, for each $\lambda>0$, and 
\item [{(ii)}] $\tau_\lambda$ is a contact form on $H^{-1}(\lambda)$ for each $\lambda>0$.
\end{itemize}
\end{lemma}

\begin{proof}
Note that $\omega_0=\frac{1}{2} (dr_1^2\wedge d\theta_1+dr^2_2\wedge d\theta_2)$. Introducing
$r^2=pr_1^2+qr_2^2$, we get 
$$
\omega_0=\frac{1}{2} (dr_1^2\wedge (d\theta_1-\frac{p}{q}d\theta_2)+\frac{1}{q} dr^2\wedge d\theta_2).
$$
With this understood, we simply set $\tau=\frac{1}{2}(r_1^2(d\theta_1-\frac{p}{q}d\theta_2)+\frac{1}{q} r^2d\theta_2)$.

Now for each $\lambda>0$, we have 
$\tau_\lambda=\frac{1}{2\lambda} r_1^2(d\theta_1-\frac{p}{q}d\theta_2)+\frac{1}{q}d\theta_2$.
One can check easily that $\tau_\lambda(X)=1$, $L_X\tau_\lambda=0$, and 
$\tau_\lambda\wedge d\tau_\lambda=\frac{1}{\lambda q} r_1 dr_1\wedge d\theta_1\wedge d\theta_2>0$. Hence the lemma.

\end{proof}

For each $\lambda>0$, we define an $\s^1$-equivariant diffeomorphism 
$\psi_\lambda: \C^2\rightarrow \C^2$ by
$$
\psi_\lambda(r_1,\theta_1,r_2,\theta_2)=(\frac{r_1}{\sqrt{\lambda}},\theta_1, \frac{r_2}{\sqrt{\lambda}},
\theta_2).
$$
We observe that $\psi_\lambda$ maps $H^{-1}(\lambda)$ diffeomorphically onto $H^{-1}(1)$, and
moreover, 
$$
\tau_\lambda=\psi^\ast_\lambda \tau_1.
$$
As a consequence, if we identify $H^{-1}([a,b])$ with $[a,b]\times H^{-1}(1)$ via $\psi_\lambda$,
then the symplectic form $\omega_0=d(\lambda \tau_1)$, $\lambda\in [a,b]$. 

For the next lemma, for simplicity we continue to denote by $C_0$ the part of $C_0$ that is contained in $H^{-1}([a,b])$. We shall consider a symplectic structure $\omega$ on $H^{-1}([a,b])$, which has the following form:
identifying $H^{-1}([a,b])$ with $[a,b]\times H^{-1}(1)$ via $\psi_\lambda$, then $\omega=d(\lambda\alpha)$, where 
$\lambda\in [a,b]$ and $\alpha$ is a contact form on $H^{-1}(1)$ and $i\alpha$ is a connection $1$-form for the orbifold principal $\s^1$-bundle $H^{-1}(1)\rightarrow H^{-1}(1)/\s^1$. For example, for the standard symplectic form 
$\omega_0$ on $H^{-1}([a,b])$, $\omega_0=d(\lambda \tau_1)$.

With this understood, it is easy to check that $C_0$ is a symplectic annulus in $H^{-1}([a,b])$ with respect to 
$\omega$, with two boundary components denoted by $F_a\subset H^{-1}(a)$, $F_b\subset H^{-1}(b)$. Note that both $F_a,F_b$ are a regular fiber of the corresponding Seifert fibration on $H^{-1}(a)$, $H^{-1}(b)$ respectively.

\begin{lemma}
Fix $\omega$ as the symplectic structure on $H^{-1}([a,b])$, and 
let $\tilde{F_a}\subset H^{-1}(a)$ be a regular fiber nearby $F_a$. Then there exists a symplectic 
annulus $\tilde{C}_0$ in $H^{-1}([a,b])$ with the following significance:
\begin{itemize}
\item [{(i)}] the boundary component of $\tilde{C}_0$ in $H^{-1}(a)$ is  $\tilde{F_a}$,
\item [{(ii)}] $\tilde{C}_0=C_0$ in  $H^{-1}([\frac{a+b}{2},b])$, and 
\item [{(iii)}] $\tilde{C}_0$ is smoothly isotopic to $C_0$ in $H^{-1}([a,b])$. 
\end{itemize}
\end{lemma} 

\begin{proof}
First, note that the intersection $C_0\cap H^{-1}(\lambda)$ is a regular fiber of the Seifert fibration on
$H^{-1}(\lambda)$, $\forall \lambda\in [a,b]$. With this understood, we identify  $H^{-1}([a,b])$ with 
$[a,b]\times H^{-1}(1)$ via $\psi_\lambda$. Then there is a smooth path $\gamma(\lambda)$,
$\lambda\in [a,b]$, in the $2$-orbifold $H^{-1}(1)/\s^1$ which lies in the complement of the singular points, such that $C_0$ may be regarded as the image of the following embedding
$\phi: [a,b]\times \s^1\rightarrow [a,b]\times \s^1\times U$, sending $(\lambda,\theta)$ to
$(\lambda,\theta,\gamma(\lambda))$. Here $U$ is a regular neighborhood of the path $\gamma$
in $H^{-1}(1)/\s^1$, over which the Seifert fibration is identified with the product $\s^1\times U$,
and without loss of generality, we assume the fiber $\tilde{F_a}$ is over a point $x\in U$. 

With the preceding understood, we choose a smooth path $\tilde{\gamma}(\lambda)$ lying in $U$,
such that $\tilde{\gamma}(a)=x$ and $\tilde{\gamma}(\lambda)=\gamma(\lambda)$ for any
$\lambda\in [\frac{a+b}{2},b]$. Then we simply let $\tilde{C}_0$ be the image of the following
embedding $\tilde{\phi}: [a,b]\times \s^1\rightarrow [a,b]\times \s^1\times U$, sending 
$(\lambda,\theta)$ to $(\lambda,\theta,\tilde{\gamma}(\lambda))$. With 
$\omega=d(\lambda\alpha)$, it is easy to check that $\tilde{C}_0$ is symplectic. Other
properties (i)-(iii) are clear from the construction of $\tilde{C}_0$.

\end{proof}

\begin{remark}
In complete analogy, if $\tilde{F_b}\subset H^{-1}(b)$ is a regular fiber nearby $F_b$, then there exists a 
symplectic annulus $\tilde{C}_0$ in $H^{-1}([a,b])$ with the following significance:
\begin{itemize}
\item [{(i)}] the boundary component of $\tilde{C}_0$ in $H^{-1}(b)$ is  $\tilde{F_b}$,
\item [{(ii)}] $\tilde{C}_0=C_0$ in  $H^{-1}([a, \frac{a+b}{2}])$, and 
\item [{(iii)}] $\tilde{C}_0$ is smoothly isotopic to $C_0$ in $H^{-1}([a,b])$. 
\end{itemize}
\end{remark}

In the next lemma, we will show that given a symplectic rational unicuspidal curve $C$ with one Puiseux pair
$(p,q)$ in a $4$-dimensional symplectic manifold $(X,\omega)$ (not necessarily compact closed), one can always perturb the symplectic form near the singularity and perturb the curve $C$ through a small isotopy, such that
near the singularity, the symplectic structure and the cuspidal curve are given by the standard model $\omega_0$
and $C_0$. 

\begin{lemma}
Let $C$ be a symplectic rational unicuspidal curve with one Puiseux pair $(p,q)$ in a $4$-dimensional symplectic manifold $(X,\omega)$. There exist a holomorphic coordinate system 
$(z_1,z_2)$ centered at the singularity and a sufficiently small  $\epsilon_0>0$, with the following significance:
\begin{itemize}
\item [{(1)}] there is a symplectic structure $\omega^\prime$ on $X$ such that $\omega^\prime=\omega$
outside the ball $|z_1|^2+|z_2|^2<(3\epsilon_0)^2$ and $\omega^\prime=\omega_0:=\frac{i}{2}(dz_1\wedge d\bar{z}_1+dz_2\wedge d\bar{z}_2)$ inside the ball $|z_1|^2+|z_2|^2\leq (2\epsilon_0)^2$, and 
\item [{(2)}] there is a symplectic rational unicuspidal curve $C^\prime$, such that $C^\prime=C$ outside the ball 
$|z_1|^2+|z_2|^2<(2\epsilon_0)^2$, and $C^\prime=C_0$ (i.e., given by $z_1=z^p$, $z_2=c z^q$) in the ball 
$|z_1|^2+|z_2|^2\leq \epsilon_0^2$, and is smoothly embedded and symplectic with respect to 
$\omega^\prime$ on $\epsilon_0^2\leq |z_1|^2+|z_2|^2\leq (3\epsilon_0)^2$, and
\item [{(3)}] $C^\prime$ is isotopic to $C$ through a smooth family of symplectic rational unicuspidal curves
$C^{(s)}$, where $s\in [0,1]$, such that $C^{(s)}$ have the same singularity and $C^{(0)}=C^\prime$, $C^{(1)} =C$.
\end{itemize}
\end{lemma}

\begin{proof}
First of all, we pick a holomorphic coordinate system $(z_1,z_2)$ over which
the K\"{a}hler form $\omega$ equals the standard form $\omega_0$ up to order $1$, i.e., 
$$
\omega=\omega_0+\sum_{i,j} g_{ij} dz_i\wedge d\bar{z}_j, 
$$
where $\omega_0=\frac{i}{2}(dz_1\wedge d\bar{z}_1+dz_2\wedge d\bar{z}_2)$ and $g_{ij}(0)=0$. 
On the other hand, since the singularity has one Puiseux pair $(p,q)$, where $p<q$, we may assume 
$C$ is given in the coordinate system $(z_1,z_2)$ by a Puiseux parametrization (cf. \cite{BK})
$$
z_1=z^p, \; z_2=\sum_{i=1}^\infty c_i z^{q_i}, \;\; z\in D\subset \C,
$$
where $\sum_{i=1}^\infty c_i z^{q_i}$ is a convergent power series, $c_i\neq 0$, $q_i<q_j$ for
$i<j$, and $q_1=q$. 

With the preceding understood, we shall first construct the symplectic form $\omega^\prime$. 
To this end, we first note that there is a $1$-form $\beta$ which vanishes at $(0,0)$ up to order $2$
such that $\omega=\omega_0+d\beta$. To see this, we note that there is a $1$-form $\gamma$
such that $d\gamma=\omega-\omega_0$. Let $\{x_j\}$ be the corresponding real coordinates, and 
write $\gamma=\sum_j\gamma_j dx_j$. Then 
$d\gamma=\sum_{k<j}(\frac{\partial \gamma_j}{\partial x_k}-\frac{\partial \gamma_k}{\partial x_j})
dx_k\wedge dx_j$, and with $g_{ij}(0)=0$, we get
$$
\frac{\partial \gamma_j}{\partial x_k}(0)-\frac{\partial \gamma_k}{\partial x_j}(0)=0. 
$$
Now set $\beta_j:=\gamma_j-(\gamma_j(0)+\sum_k \frac{\partial \gamma_j}{\partial x_k}(0)x_k)$. 
Then $\beta_j(0)=0$, $\frac{\partial \beta_j}{\partial x_k}(0)=0$, $\forall k,j$, and
$d\beta=d\gamma$ where $\beta=\sum_j \beta_jdx_j$.

With this understood, we choose a cut-off function $\rho(t)$ such that $\rho(t)=0$ for 
$t\leq 2\epsilon_0$ and $\rho(t)=1$ for $t\geq 3\epsilon_0$. Note that $|\rho^\prime(t)|\leq C_1\epsilon_0^{-1}$ for some constant $C_1>0$. Setting $r:=\sqrt{|z_1|^2+|z_2|^2}$, we
define
$$
\omega^\prime=\omega_0+d(\rho(r)\beta). 
$$
It remains to show that $\omega^\prime$ is non-degenerate when $\epsilon_0>0$ is chosen
sufficiently small. For this, we have 
$$
d(\rho(r)\beta)=d\rho \wedge \beta+\rho d\beta=\rho^\prime \sum_{k,j} \frac{\partial r}{\partial x_k}\beta_jdx_k\wedge dx_j+\rho \sum_{k,j}\frac{\partial \beta_j}{\partial x_k} dx_k\wedge dx_j.
$$
With $|\frac{\partial r}{\partial x_k}|\leq 1$, $|\rho^\prime(t)|\leq C_1\epsilon_0^{-1}$,
$|\beta_j|\leq C_2\epsilon_0^2$ and $|\frac{\partial \beta_j}{\partial x_k}|\leq C_3\epsilon_0$
on $|z_1|^2+|z_2|^2\leq (3\epsilon_0)^2$, it follows easily that 
$$
d(\rho(r)\beta)=\sum_{k,j} a_{kj} dx_k\wedge dx_j, 
$$
where $|a_{kj}|\leq C_4 \epsilon_0$ for some $C_4>0$. When $\epsilon_0>0$ is chosen
sufficiently small, we have $$\omega^\prime\wedge \omega^\prime>0.$$

Before we construct $C^\prime$, we observe that $\omega^\prime|_C>0$ on 
$2\epsilon_0\leq \sqrt{|z_1|^2+|z_2|^2}\leq 3\epsilon_0$. To see this, recall that $C$ is given by
$z_1=z^p$, $z_2=\sum_{i=1}^\infty c_i z^{q_i}$. With this understood, if we write $z=s+it$, then
on $2\epsilon_0\leq \sqrt{|z_1|^2+|z_2|^2}\leq 3\epsilon_0$,
$$
\omega_0|_C\geq \frac{i}{2} dz_1\wedge d\bar{z}_1|_{C}=p^2 (s^2+t^2)^{p-1} ds\wedge dt\geq 2^{2p-2}p^2
\epsilon_0^{2p-2} ds\wedge dt.
$$
On the other hand, the real coordinates $x_j|C$ are power series in $s,t$ of order $\geq p$, from 
which it follows easily that 
$$
|a_{kj} dx_k\wedge dx_j|_C|\leq C_5 |a_{kj}| (s^2+t^2)^{p-1} ds\wedge dt\leq 3^{2p-2} C_5 |a_{kj}| \epsilon_0^{2p-2} ds\wedge dt.
$$
With $|a_{kj}|\leq C_4 \epsilon_0$, it follows easily that $\omega^\prime|C>0$ for sufficiently small
$\epsilon_0>0$.

Now we construct $C^\prime$. We fix a cut-off function $h(t)$ such that $h(t)=0$ for 
$t\leq \delta$ and $h(t)=1$ for $t\geq 2^{1/2p}\delta$, where $\delta=\epsilon_0^{1/p}$.
We assume $\epsilon_0>0$ is sufficiently small so that if $|z|\leq 2^{1/2p}\delta$, 
$$
|c_1z^{q_1}|<\frac{1}{2}|z|^p \mbox{ and } |\sum_{i=2}^\infty c_i z^{q_i}|<\frac{1}{2}|z|^p.
$$
With this understood, we define $C^\prime$ by the equations
$$
z_1=z^p,\;\; z_2=c_1z^{q_1} +h(|z|) \sum_{i=2}^\infty c_i z^{q_i}.
$$
On $|z_1|^2+|z_2|^2\leq \epsilon_0^2$, $|z^p|=|z_1|_{C^\prime}|\leq \epsilon_0$, so that 
$|z|\leq\delta$, and hence $h(|z|)=0$. It follows immediately that $C^\prime$ is given by 
$z_1=z^p$, $z_2=c z^q$ in the ball $|z_1|^2+|z_2|^2\leq \epsilon_0^2$, where $c=c_1$ 
and $q=q_1$. On the other hand, if $|z|<2^{1/2p}\delta$, it is easy to see that 
$|z_2|_{C^\prime}|< |z|^p=|z_1|_{C^\prime}|$, so that 
$$
|z_1|_{C^\prime}|^2+|z_2|_{C^\prime}|^2<2|z_1|_{C^\prime}|^2=2|z|^{2p}<4\epsilon_0^2.
$$
Thus outside the ball $|z_1|^2+|z_2|^2<(2\epsilon_0)^2$, we must have $h(|z|)=1$, so that
$C^\prime=C$. 

It remains to show that on $\epsilon_0\leq \sqrt{|z_1|^2+|z_2|^2}\leq 2\epsilon_0$, $C^\prime$ is
smoothly embedded and $\omega_0|_{C^\prime}>0$ (as $\omega^\prime=\omega_0$ over here).
To this end, we observe that the $z_1$-component, i.e., $z_1=z^p$, is an immersion for $z\neq 0$, so that 
$C^\prime$ is smoothly embedded if the map $z\mapsto (z_1,z_2)$, where 
$$
z_1=z^p,\;\; z_2=c_1z^{q_1} +h(|z|) \sum_{i=2}^\infty c_i z^{q_i},
$$
is one to one. To see this, we assume $z_1(z)=z_1(z^\prime)$. Then $z^\prime=\lambda z$ for some
$\lambda\in\C$ such that $\lambda^p=1$. With this understood, we note that
$$
z_2(z^\prime)=c_1\lambda^{q_1} z^{q_1}+h(|z|) \sum_{i=2}^\infty c_i \lambda^{q_i} z^{q_i},
$$
so that when $\epsilon_0>0$ is sufficiently small, $|z|\leq (2\epsilon_0)^{1/p}$ will be small enough, and 
$$
z_2(z^\prime)-z_2(z)=c_1(\lambda^{q_1}-1)z^{q_1}
+h(|z|) \sum_{i=2}^\infty c_i (\lambda^{q_i}-1) z^{q_i}\neq 0 \mbox{ if } z\neq 0, \lambda\neq 1.
$$
(Here we use the fact that $p,q$ are relatively prime.) Hence $C^\prime$ is smoothly embedded.

To see that $\omega_0|_{C^\prime}>0$, we note that, writing $z=s+it$, 
$$
\frac{i}{2} dz_1\wedge d\bar{z}_1|_{C^\prime}=p^2 (s^2+t^2)^{p-1} ds\wedge dt\geq C_6\delta^{2p-2}
ds\wedge dt.
$$ 
On the other hand, with $|h^\prime(|z|)|\leq C_7 \delta^{-1}$ for some $C_7>0$, it is easy to see
that 
$$
|\frac{i}{2} dz_2\wedge d\bar{z}_2|_{C^\prime}|\leq C_8 \delta^{2q_1-2} ds\wedge dt.
$$
It follows easily that $\omega_0|_{C^\prime}>0$ when $\epsilon_0=\delta^p$ is small enough,
as $p<q=q_1$.

Finally, to see $C^\prime$ is isotopic to $C$, we choose a smooth family of cut-off functions $h^{(s)}(t)$, where 
$s\in [0,1]$, such that $h^{(s)}(t)=s$ for $t\leq \delta$ and $h^{(s)}(t)=1$ for $t\geq 2^{1/2p}\delta$, 
where $\delta=\epsilon_0^{1/p}$. Define $C^{(s)}$ to be the curve 
$$
z_1=z^p,\;\; z_2=c_1z^{q_1} +h^{(s)}(|z|) \sum_{i=2}^\infty c_i z^{q_i}, \mbox{ where } s\in [0,1]. 
$$
Then each $C^{(s)}$ is a symplectic rational unicuspidal curve, having the same singular point at $(0,0)$, and
$C^{(0)}=C^\prime$ and $C^{(1)} =C$.

\end{proof}

\vspace{3mm}

{\bf Proof of Theorem 1.17:}

\vspace{2mm}

We begin with a proof of part (1) of the theorem. To this end, we fix a $\lambda_0>0$ which is sufficiently small. 
Let $F_{\lambda_0}:=C_0\cap H^{-1}(\lambda_0)$ which is a regular fiber of the Seifert fibration on 
$H^{-1}(\lambda_0)$ associated to the $\s^1$-action. With this understood, we let $B$ be the 
image of $F_{\lambda_0}$ in $H^{-1}(1)$ under the $\s^1$-equivariant diffeomorphism 
$\psi_{\lambda_0}$, which is also a regular fiber of the Seifert fibration on $H^{-1}(1)$. Finally, it is easy to see
that under the Seifert fibration, $H^{-1}(1)=M((1,-1), (p,p^\prime),(q,q^\prime))$ as a Seifert manifold, where $p^\prime,
q^\prime>0$ are integers uniquely determined by $pq^\prime+qp^\prime=pq+1$.

With the preceding understood, it follows easily from Theorem 3.3 that there is a rational open book with a periodic monodromy of order $n=pq>1$, which is compatible with the Seifert fibration on $H^{-1}(1)$ and has $B$ as its binding with fiber orientation. Furthermore, following through the proof it is easy to see that $k=1$ (as $c=n-1$), $p=1$ (the multiplicity at $B$), so with the multiplicity of $B$ being $1$ ($B$ is a regular fiber), we also have
$p^\prime=0$. With this understood, the integer $q$ in the equation $pq^\prime-p^\prime q=1$ can be chosen arbitrarily, with $q^\prime=1$ always. 

Next, we shall follow through the proof of Theorem 4.1 to construct a specific $\s^1$-invariant contact form. Note that we are in Case 2, so we shall choose a $q<0$ such that $kq+nq^\prime<0$. It is easy to see that we can let
$q=-(n+1)$, with which $kq+nq^\prime=-1$.

Next we shall explain the choice of the constant $C>0$ and the $1$-form $\beta$ on $\Sigma$ used in
the construction of the $\s^1$-invariant contact form. To this end, we fix a sufficiently small $\epsilon>0$, 
and take a small disc $D_\epsilon$ centered at the point over which the regular fiber $B$ is lying, such that the integral of $d\tau_1$ over the complement of the disc $D_\epsilon$ equals 
$\frac{2\pi}{n}(1-\epsilon)$ (here $\tau_1$ is the $1$-form $\tau_\lambda$ from Lemma 6.1 
with $\lambda=1$),  i.e.,
$$
\int_{(H^{-1}(1)/\s^1)\setminus D_\epsilon} d\tau_1=\frac{2\pi}{n}(1-\epsilon).
$$
(Note that $\int_{H^{-1}(1)/\s^1} d\tau_1=\frac{2\pi}{pq}$ by Chern-Weil theory as the Euler number of the Seifert fibration on $H^{-1}(1)$ equals $-\frac{1}{pq}$, and $n=pq$.)

With this understood, note that $\Sigma$ is a $n$-fold branch covering over 
$(H^{-1}(1)/\s^1)\setminus D_\epsilon$, and the pull-back of the Seifert fibration over $\Sigma$ is
trivial. As a consequence, we can write the pull-back of $\tau_1$ as $d\theta+\beta_1$ for
some $1$-form $\beta_1$ on $\Sigma$. With this understood, we require the $1$-form $\beta$ to be equal to $\frac{1}{1-\epsilon}\beta_1$ in the complement of a small neighborhood of the boundary 
$\partial\Sigma$ (near $\partial \Sigma$ we require $\beta$ to take the standard form as 
in the proof of Theorem 4.1\footnote{To see this, note that we may assume $d\tau_1=rdr\wedge d\theta$ near $\partial \Sigma$ where $(r,\theta)$ are polar coordinates on a slightly larger disk containing $D_\epsilon$. Then $\beta_1-\frac{r^2}{2} d\theta=\kappa d\theta+ df$ for some 
constant $\kappa$. Simply let $\beta=\frac{1}{1-\epsilon}(\frac{r^2}{2} d\theta+\kappa d\theta+
d(\rho f)$) near $\partial \Sigma$ where $\rho$ is a cut-off function equaling zero near $\partial \Sigma$.}). 
Finally, we take the constant $C=\frac{1}{(1-\epsilon)n}$ so that 
$\frac{Cnp}{kp+np^\prime}=Cn>1$ is satisfied. Denote by $\alpha$ the resulting contact form on $H^{-1}(1)$.

With the preceding understood, we set $\alpha_1:=(1-\epsilon)\alpha$. Then it is easy to see that
$\alpha_1$ is a positive contact form on $H^{-1}(1)$, such that $i\alpha_1$ is a connection $1$-form 
for the Seifert fibration on $H^{-1}(1)$, and moreover, $\alpha_1=\tau_1$ in the complement of a
small neighborhood of the binding $B$. We note that the Reeb vector field $R_{\alpha_1}$ is given 
near the binding $B$, in the coordinates $(r,\theta,z)$ from Theorem 4.1, by
$$
R_{\alpha_1}=-\frac{\partial}{\partial \theta}+\frac{\partial}{\partial z}.
$$
On the other hand, in the coordinates $(r,\theta,z)$, $d\pi=pd\theta-q dz=d\theta+(n+1) dz$, where 
$\pi: H^{-1}(1)\setminus B\rightarrow \s^1$ is the fibration associated to the rational open book on $H^{-1}(1)$.

To proceed further, we note that by Lemma 2.6 in \cite{C1}, there exists an $\s^1$-equivariant diffeomorphism 
$\psi: H^{-1}(1)\rightarrow H^{-1}(1)$ such that $\psi^\ast \tau_1=\alpha_1$. We remark that in the present case, since $\alpha_1=\tau_1$ in the complement of a small neighborhood of $B$, $\psi$ is actually supported in a small neighborhood of $B$. With this understood, we fix a small $\delta>0$. It is easy to see that, via the diffeomorphisms 
$\psi_\lambda$, $\psi: H^{-1}(1)\rightarrow H^{-1}(1)$ gives rise to an $\s^1$-equivariant diffeomorphism 
$\Psi: H^{-1}([\lambda_0-\delta,\lambda_0])\rightarrow H^{-1}([\lambda_0-\delta,\lambda_0])$ such that 
$\Psi^\ast \omega_0=\tilde{\omega}$, i.e., $\Psi$ is a symplectomorphism with respect to the symplectic structures 
$\omega_0=d(\lambda\tau_1)$ and $\tilde{\omega}=d(\lambda\alpha_1)$ on 
$H^{-1}([\lambda_0-\delta,\lambda_0])$. The regular fiber $F_{\lambda_0}$ in $H^{-1}(\lambda_0)$ is sent to a nearby regular fiber, denoted by $F_{\lambda_0}^\prime$, under $\Psi$. With this understood, we observe that 
by the version of Lemma 6.2 stated in Remark 6.3, 
$F_{\lambda_0}^\prime$ bounds a symplectic unicuspidal disc $D_0$ in the $4$-ball $H^{-1}([0,\lambda_0])$, 
which is a small perturbation of the standard model curve $C_0$ and equals $C_0$ near the singularity. 

Next we shall apply Gay's construction in \cite{G} to attach a symplectic $2$-handle to the $4$-ball 
$H^{-1}([0,\lambda_0])$, which is equipped with the standard symplectic structure $\omega_0$, along 
$F_{\lambda_0}^\prime\subset H^{-1}(\lambda_0)$ with framing $m>0$ (relative to the zero-framing, which is the same as the page-framing because there is only one binding component and the multiplicity $p=1$). The resulting manifold topologically can be identified with the handlebody $U_{p,q,m}$, but the construction will give rise to an
$\s^1$-invariant contact structure $\xi_{inv}$ on the concave boundary $M_{p,q,m}$.
Note that there is a natural symplectic rational unicuspidal curve with self-intersection $m$ in the interior of
$U_{p,q,m}$, i.e., $D\cup D_0$, where $D$ denotes the core disc of the symplectic $2$-handle. 

For simplicity, however, we shall work with the symplectic structure $\tilde{\omega}=d(\lambda\alpha_1)$ and the
contact form $\lambda_0\alpha_1$ on the boundary $H^{-1}(\lambda_0)$ via an identification by the $\s^1$-equivariant diffeomorphism $\Psi: H^{-1}([\lambda_0-\delta,\lambda_0])\rightarrow H^{-1}([\lambda_0-\delta,\lambda_0])$. With this understood, recall that near the binding there is a local coordinate system $(r,\theta,z)$ such that the Reeb vector field $R_{\alpha_1}=-\frac{\partial}{\partial \theta}+\frac{\partial}{\partial z}$ and $d\pi=d\theta+(n+1) dz$, where $\pi$ is the fibration associated to the rational open book. It follows 
easily that the Reeb vector field associated to the contact form $\lambda_0\alpha_1$ is 
$$
R_{\lambda_0\alpha_1}=-\frac{1}{\lambda_0}\frac{\partial}{\partial \theta}+\frac{1}{\lambda_0}\frac{\partial}{\partial z}.
$$
With this understood, we change the framing of the coordinate system, still denoted by $(r,\theta,z)$ but now
$\frac{\partial}{\partial z}$ gives the framing of the $2$-handle to be attached to the binding. With respect to the
new framing, $d\pi=d\theta+mdz$ and the Reeb vector field is given by 
$$
R_{\lambda_0\alpha_1}=-\frac{1}{\lambda_0}(m-n)\frac{\partial}{\partial \theta}+\frac{1}{\lambda_0}\frac{\partial}{\partial z}.
$$

Now we choose a constant $C>0$ and let $\alpha^0:=Cd\pi=Cd\theta+Ddz$, where we require $D=Cm$ obeys
the condition $D\cdot \frac{1}{\lambda_0}>1$. Under these assumptions, according to Lemma 4.7 in \cite{G},
one can perform an enlargement procedure, by considering the graph of a small function $h$ which is supported 
in an arbitrarily small neighborhood of the binding, we may assume for a new contact form, denoted by
$\alpha^{+}$, the Reeb vector field $R_{\alpha^{+}}=A\frac{\partial}{\partial \theta}+B\frac{\partial}{\partial z}$, while we still have $\alpha^0=Cd\theta+Ddz$, such that $A=B>0$, $C>0$, $AD>1$, where $A$ is a constant chosen to obey $\frac{1}{D}<A< \frac{1}{\lambda_0}$, which is possible because $D\cdot \frac{1}{\lambda_0}>1$. In Gay's terminology \cite{G}, the pair $(\alpha^{+},\alpha^0)$ is under the condition of ``prepared for surgery", with which
a symplectic $2$-handle can be attached along the binding. See Proposition 4.6 in \cite{G} for more details, in particular, for the explicit description of the model for the symplectic $2$-handle. 

With the preceding understood, there are two parameters $\epsilon_1,\epsilon_2$ in Gay's construction , 
with $\epsilon_1<0<\epsilon_2$, that are specially relevant in our consideration here. Roughly speaking, $\epsilon_1,\epsilon_2$ determine the ``size" and ``shape" of the $2$-handle, and they are constrained to the structural data $(A,B,C,D)$ by the following relations
$$
\epsilon_1=\frac{2}{A}-2D, \mbox{ and } 0<\epsilon_2<2C. 
$$
More specifically, the constant $\epsilon_1$ is related to the symplectic area of the core disk 
of the $2$-handle, which equals $\pi|\epsilon_1|$, and $\epsilon_2>0$ can be taken to be sufficiently small, 
to make the $2$-handle arbitrarily ``skinny". With this understood, we shall also make the following important
observation: note that by Stokes's theorem, the area of the cuspidal disc $D_0$ in the $4$-ball equals 
$\int_B \alpha^{+}=\frac{2\pi}{A}$ where $B$ is the binding component. On the other hand, the area of the 
core disc of the $2$-handle equals $\pi |\epsilon_1|= 2\pi(D-\frac{1}{A})$. It follows that the area of the rational unicuspidal curve formed by the union of $D_0$ and the core disc of the $2$-handle equals $2\pi D$. 

Thus we have finished constructing a symplectic structure on the handlebody $U_{p,q,m}$ which has a concave contact boundary $(M_{p,q,m},\xi_{inv})$. There are three parameters the symplectic structure depends on:
$\lambda_0>0$, which gives the ``size" of the $4$-ball, $\epsilon_2>0$ which determines how ``skinny" the $2$-handle is, and the constant $D$ which gives the area of the rational cuspidal curve inside $U_{p,q,m}$. We note 
that the constant $D$ actually has an intrinsic meaning, which is closely related to the rational open book and 
the framing $F$ as follows. Recall that the $1$-form $\alpha^0=cd\pi=Cd\theta+D dz$ for some constant $c>0$, 
with $C=cp$,  where $p$ is the multiplicity of the binding and $\frac{D}{C}=F>0$ is the framing of the $2$-handle relative to the page framing. With this understood, note that $D=cpF$ where $c>0$ is the constant in 
$\alpha^0=c d\pi$. 

Before we proceed to exhibit $M_{p,q,m}$ as a Seifert manifold and describe the $\s^1$-invariant contact
structure $\xi_{inv}$, we make the following remark for a future reference. 

\begin{remark}
Imagine we construct a symplectic structure on a handlebody where there are two $2$-handles. We need to employ 
a rational open book with two binding components $B_1,B_2$, and there are two rational cuspidal curves $C_1,C_2$ inside the handlebody. If we let $p_i$ be the multiplicity at $B_i$ and $F_i>0$ be the framing (relative to the page framing) of the $2$-handle attached along $B_i$, for $i=1,2$, then we observe that
$$
\frac{Area(C_1)}{Area(C_2)}=\frac{p_1 F_1}{p_2 F_2}\in \Q. 
$$
To put it differently, suppose we want to build a symplectic model for the regular neighborhood of a union of two rational cuspidal curves $C_1,C_2$, then a necessary condition is that the ratio of the areas of $C_1$ and
$C_2$ must be a rational number, and we have to construct a rational open book such that 
$\frac{Area(C_1)}{Area(C_2)}=\frac{p_1 F_1}{p_2 F_2}$ holds true. This is considerably more complicated than
the one cuspidal curve case. 

We shall also compare the discussions here with the formula of $[\omega_Z]$ in Lemma 5.1. To this end, note 
that $\bar{p}_i:=p_i F_i$ is actually the multiplicity at the binding component $B_i^\prime$ of the rational open book $(B^\prime,\pi^\prime)$ on the other boundary component $M^\prime$, see the proof of Theorem 4.4. With this 
said, the areas of the rational cuspidal curves $C_1,C_2$ can be written as $Area(C_1)=c\bar{p}_1$
and $Area(C_2)=c\bar{p}_2$ for some constant $c>0$. 

\end{remark}

Now going back to the proof of Theorem 1.17, we apply Theorem 4.4 to describe $M_{p,q,m}$ as a Seifert manifold, by computing the Seifert invariant $(\bar{\alpha},\bar{\beta})$. By the formula $F=\frac{1}{\alpha}(\frac{n}{p}-\bar{\alpha})$, it is easy to see that in this case, one has $\bar{\alpha}=pq-m$. The formula for $\bar{\beta}$ in 
Theorem 4.4 gives $\bar{\beta}=pq-m+1$. It follows easily that
\begin{itemize}
\item if $m<pq$, then $M_{p,q,m}=M((pq-m, pq-m+1), (p,-p^\prime), (q,-q^\prime))$, and $\xi_{inv}$ is
an $\s^1$-invariant, transverse contact structure, 
\item if $m=pq$, then the $\s^1$-action on $M_{p,q,m}$ has a fixed component, and $M_{p,q,m}$ is the connected sum of two lens spaces, $M_{p,q,m}=L(p,p^\prime)\# L(q,q^\prime)$ (cf. \cite{N}). The contact structure 
$\xi_{inv}$ is transverse to the $\s^1$-orbits in the complement of the fixed component, and 
\item if $m>pq$, then $M_{p,q,m}=M((m-pq, m-pq-1), (p,-p^\prime), (q,-q^\prime))$, and $\xi_{inv}$ is
an $\s^1$-invariant, non-transverse contact structure, with a dividing set consisting of a circle separating the singular point of multiplicity $m-pq$ from the singular points of multiplicities $p$ and $q$. 
\end{itemize}

This completes the proof of part (1) of Theorem 1.17. Next, we consider part (3) of the theorem first. 
Observe that if $m=m_{p,q}$, then $m>pq$, so that $\xi_{inv}$ is non-transverse with a dividing set 
consisting of a circle separating the singular point of multiplicity $m-pq$ from the singular points of 
multiplicities $p$ and $q$. Moreover, one of the inequalities 
$m\leq pq+\frac{p}{p-p^\prime}$ or $m\leq pq+\frac{q}{q-q^\prime}$ must be true. Equivalently, either 
$\frac{pq-m-1}{m-pq}+\frac{-p^\prime}{p}\leq 0$, or one must have
$\frac{pq-m-1}{m-pq}+\frac{-q^\prime}{q}\leq 0$. By Theorem 2.3, the contact structure $\xi_{inv}$ is
tight, which is in fact fillable. Consequently, let $W$ be a symplectic filling of $(M_{p,q,m},\xi_{inv})$, then the union $X:=W\cup U_{p,q,m}$ is a closed symplectic 
$4$-manifold, containing an embedded symplectic rational unicuspidal curve with one Puiseux pair $(p,q)$ and 
with self-intersection $m=m_{p,q}$. To see that $X$ must be a rational $4$-manifold, we note that a successive 
blowing up of $X$ contains the symplectic divisor $D_{p,q,m}$. With $m>pq$, the component of $D_{p,q,m}$
which is the proper transform of the cuspidal curve is an embedded symplectic two-sphere of positive 
self-intersection. This proves that $X$ must be a rational $4$-manifold. 

Note that it is clear from the above discussion in connection to Theorem 2.3, that if there is an integer $m>m_{p,q}$ such that the contact structure $\xi_{inv}$ on $M_{p,q,m}$ is tight, then we have an example which provides a negative answer to Question 2.6.

It remains to prove part (2) of Theorem 1.17. The following is the key lemma.

\begin{lemma}
Let $(X,\omega)$ be a $4$-dimensional symplectic manifold (not necessarily compact closed), and let $C$ be a
symplectic rational unicuspidal curve with one Puiseux pair $(p,q)$ and with self-intersection $m>0$ in the interior of
$X$. Then after perturbing $\omega$ slightly near the singularity to a symplectic form $\omega^\prime$, there is a symplectic rational unicuspidal curve $C^\prime$, which is a small perturbation of $C$ and having the same singularity, such that a small regular neighborhood of $C^\prime$ in the interior of $(X,\omega^\prime)$ has a concave contact boundary $(M_{p,q,m},\xi_{inv})$.
\end{lemma}

\begin{proof}
First, we apply Lemma 6.4 to perturb $\omega$ and $C$ near the singularity so that they are given by the standard models. We denote the new symplectic form by $\omega^\prime$, but still use the notation $C$ for the cuspidal curve for simplicity. With this understood, we fix a $0<\lambda_0<\epsilon_0$, where $\epsilon_0$ is from 
Lemma 6.4, and consider the convex $4$-ball neighborhood $H^{-1}([0,\lambda_0])$ of the singularity of $C$. 

Recall from the construction of the symplectic structure on the handlebody $U_{p,q,m}$, there is a regular fiber
$F^\prime_{\lambda_0}\subset H^{-1}(\lambda_0)$, where we perform an enlargement procedure by considering
the graph of a small function $h$ supported in a small neighborhood of $F^\prime_{\lambda_0}$. If we still denote
its image in the graph of $h$ by $F^\prime_{\lambda_0}$ for simplicity, then $F^\prime_{\lambda_0}$ bounds a
symplectic unicuspidal disc $D_0$ inside the enlarged $4$-ball neighborhood of the singularity, and $D_0$ is a small perturbation of the part of $C$ inside the enlarged $4$-ball neighborhood. On the other hand, by Lemma 6.2,
there is an embedded symplectic disc $D^\prime$, bounded by $F^\prime_{\lambda_0}$, which is a small perturbation of the part of $C$ that lies outside the enlarged $4$-ball neighborhood. With this understood, the union $C^\prime:=D^\prime\cup D_0$ is a symplectic rational unicuspidal curve, having the same
singularity, which is a small perturbation of $C$. 

Now we choose a symplectic structure on the handlebody $U_{p,q,m}$ which we built earlier, such that the parameter $\lambda_0$ is the same above, and the unicuspidal curve inside $U_{p,q,m}$ has the same area of
the unicuspidal curve $C$ in $(X,\omega^\prime)$, which is the same as the area of $C^\prime$ in 
$(X,\omega^\prime)$ as $C$ and $C^\prime$ are obviously homologous. Now the crucial observation is that the core disc $D$ of the symplectic $2$-handle in $U_{p,q,m}$ has the same area as the symplectic disc $D^\prime\subset C^\prime$.

By the standard Moser's argument, we obtain an area-preserving diffeomorphism $\phi: D\rightarrow D^\prime$, 
which then can be extended to a symplectomorphism $\Phi$ between a neighborhood of $D$ and a neighborhood of $D^\prime$ by a relative version of Weinstein's neighborhood theorem. 
Now we choose a smaller parameter $\epsilon_2>0$ 
to make the symplectic $2$-handle in $U_{p,q,m}$ more ``skinny", so that with the new symplectic model on $U_{p,q,m}$, it is symplectically embedded under $\Phi$ into a neighborhood of $C^\prime$ in the interior of 
$(X,\omega^\prime)$. This finishes the proof of the lemma.

\end{proof}

As an immediate corollary, it follows easily that there exists a symplectic cobordism from 
$(M_{p,q,m},\xi_{LM})$ to $(M_{p,q,m},\xi_{inv})$ which is topologically a product. Given any symplectic filling $W$ of $(M_{p,q,m},\xi_{LM})$, we add the topologically trivial symplectic cobordism from 
$(M_{p,q,m},\xi_{LM})$ to $(M_{p,q,m},\xi_{inv})$ to the symplectic filling $W$. The resulting $4$-manifold, which is diffeomorphic to $W$, is a symplectic filling of $(M_{p,q,m},\xi_{inv})$.

Conversely, suppose $W$ is a symplectic filling of $(M_{p,q,m},\xi_{inv})$. The union 
$X:=W\cup U_{p,q,m}$ is a closed symplectic $4$-manifold which contains a symplectic rational unicuspidal curve
$C$ with one Puiseux pair $(p,q)$ and with self-intersection $m>0$. By a successive blowing up we obtain a
symplectic $4$-manifold $\tilde{X}$ which contains the symplectic divisor $D_{p,q,m}$
(see \cite{C5}, Lemma 3.1, for more details), such that the complement of any regular neighborhood of $D_{p,q,m}$ in $\tilde{X}$ is diffeomorphic to $W$. Now by a theorem of Li and Mak \cite{LM}, after changing the symplectic structure on $\tilde{X}$ through deformations if necessary while keeping the divisor $D_{p,q,m}$ symplectic, a sufficiently small regular neighborhood of $D_{p,q,m}$ has a concave contact boundary $(M_{p,q,m},\xi_{LM})$. It follows easily that $W$ is a symplectic filling of $(M_{p,q,m},\xi_{LM})$. This finishes the proof of part (2), and the proof of Theorem 1.17 is complete.

\vspace{2mm}

In the remaining part of this section, we shall examine some natural examples of rational cuspidal curves in
an algebraic surface in connection with Theorem 1.17 or Conjecture 1.15, and prove some technical lemmas 
concerning the bound $m_{p,q}$ as we discussed in Remark 1.16.

\begin{example}
Let $(p,q,d)$ be a triple where there is a rational unicuspidal curve $C$ of degree $d$ in $\C\P^2$,
whose singular point has one Puiseux pair $(p,q)$. The list of such triples are classified in \cite{F}.
Note that by Theorem 1.17, for any such a triple $(p,q,d)$, the $3$-manifold $M_{p,q,m}$, where $m=d^2$, admits
an $\s^1$-invariant contact structure $\xi_{inv}$ such that $(M_{p,q,m},\xi_{inv})$ is symplectic
fillable by a $\Q$-homology $4$-ball, the complement of the corresponding rational cuspidal curve $C$ in $\C\P^2$.

The following is the list of triples $(p,q,d)$ from \cite{F}.


(1) $(p,q)=(d-1,d)$ where $d\geq 3$ is the degree of $C$. In this case, we have $(p^\prime,q^\prime)=
(1,d-1)$. With $m=d^2>pq=(d-1)d$, we have $\bar{\alpha}=-d$. Hence 
$$
M_{p,q,m}=M((d,d-1), (d-1,-1), (d, 1-d)). 
$$
Note that $m_{d-1,d}=(d-1)d+d=d^2=m$ in this case. In other words, the bound $m_{p,q}$ is sharp for 
$(p,q)=(d-1,d)$. 

(2) $(p,q)=(d/2,2d-1)$ where the degree $d$ is even. Introducing $\delta=d/2$, we have 
$(p^\prime,q^\prime)=(\delta-1, 4)$. With $m=4\delta^2>pq=\delta(4\delta-1)$, we have 
$\bar{\alpha}=-\delta$. Hence 
$$
M_{p,q,m}=M((\delta,\delta-1), (\delta, 1-\delta), (4\delta-1, -4)), \mbox{ where } \delta\geq 2. 
$$
Note that $m_{d/2,2d-1}=(d/2)(2d-1)+d/2=d^2=m$, hence the bound $m_{p,q}$ is also sharp for $(p,q)=(d/2,2d-1)$. 

\vspace{2mm}

Let $\{\varphi_j\}_{j\geq 0}$ be the Fibonacci numbers, $\varphi_0=0$, $\varphi_1=1$,
$\varphi_{j+2}=\varphi_{j+1}+\varphi_j$.

\vspace{2mm}

(3) $(p,q)=(\varphi_{j-2}^2,\varphi_j^2)$ and $d=\varphi_{j-1}^2+1=\varphi_{j-2}\varphi_j$, where 
$j\geq 5$ and is odd. In this case, the $\s^1$-action on $M_{p,q,m}$ has a fixed component,  and 
$M_{p,q,m}$ is the connected sum of two lens spaces. 

(4) $(p,q)=(\varphi_{j-2}, \varphi_{j+2})$ and $d=\varphi_j$ where $j\geq 5$ and is odd. In this case,
with the relation $\varphi_j^2+1=\varphi_{j-2}\varphi_{j+2}$, it follows that 
$(\bar{\alpha}, \bar{\beta})=(1,2)$. It is easy to see that $M_{p,q,m}$ is a lens space in this case.

(5) $(p,q)=(\varphi_4,\varphi_8+1)=(3,22)$ and $d=\varphi_6=8$. In this case, 
$(p^\prime,q^\prime)=(1, 15)$, and $d^2=64<pq=66$. It follows easily that 
$$
M_{p,q,m}=M((2,-1), (3,2), (22,7)).
$$
We note that $M_{p,q,m}$ is actually the link of a weighted homogeneous singularity which admits a rational homology disk smoothing (see \cite{BS}, Figure 1(f) with $q=2$, also in \cite{W} as $M(0,q,0)$ with $q=2$), 
and the $\s^1$-invariant contact structure $\xi_{inv}$ on $M_{p,q,m}$ is
simply the Milnor fillable contact structure. 

(6) $(p,q)=(2\varphi_4,2\varphi_8+1)=(6, 43)$ and $d=2\varphi_6=16$. In this case, 
$(p^\prime,q^\prime)=(1, 36)$, and $d^2=256<pq=258$. It follows easily that 
$$
M_{p,q,m}=M((2,-1), (6,5), (43,7)).
$$
Note that in this case, $M_{p,q,m}$ is also the link of a weighted homogeneous singularity which admits a rational homology disk smoothing (see \cite{BS}, Figure 1(j) with $q=4$, also in \cite{W} as $C_3^3(p)$ with $p=4$), and 
the $\s^1$-invariant contact structure $\xi_{inv}$ on $M_{p,q,m}$ is simply the Milnor fillable contact structure. 
\end{example}

\begin{example}
For $d\geq 4$, $a\geq b\geq 1$ with $a+b=d-2$, there is a unique degree $d$ rational cuspidal curve $C_{d,a}$ in
$\C\P^2$, which has $3$ singular points, each with one Puiseux pair, i.e., $(d-2,d-1)$, $(2, 2a+1)$, and 
$(2,2b+1)$ respectively (cf. Theorem 3.5 in \cite{FZ}). On the other hand, a simple calculation shows that 
$m_{k-1,k}=k^2$ for $k\geq 3$ and $m_{2,2k+1}=4k+4$ for $k\geq 2$ (and $m_{2,3}=9$). With this understood, 
if we resolve any two of the three singularities of $C_{d,a}$, the proper transform of $C_{d,a}$ is a rational unicuspidal curve $C$ with one Puiseux pair. It is easy to check that the self-intersection of $C$ obeys the upper bound in Conjecture 1.15. In particular, focusing on the case where the singular point of type 
$(d-2,d-1)$ is resolved, we observe that for any $k,l \geq 2$, there is a rational surface $X_l$ which is a $\C\P^2$ blown-up at $l$ points, and a rational unicuspidal curve $C_{k}\subset X_l$ with one Puiseux pair $(2, 2k+1)$, 
such that $C_{k}\cdot C_{k}=4k+4=m_{2,2k+1}$. It follows easily from this example that the bound $m_{p,q}$ is optimal for $p=2$. 
\end{example}

\begin{example}
In this example we consider some simple rational bicuspidal curves, i.e., curves with exactly two cuspidal singularities. 

Let $1<p<q-1$ and $gcd(p,q)=1$, and let $C$ be the rational cuspidal curve in $\C\P^2$ given by the parametrization $[z^p: z^q:1]$, $z\in\C$. It is of degree $q$ and has two singular points, one at $z=0$ with one Puiseux pair 
$(p,q)$, and the other at $z=\infty$, with one Puiseux pair $(q-p,q)$. With this understood, consider specially 
the case where $p=3$. Then $q=3k+1$ or $q=3k+2$ for some $k>0$. If $q=3k+1$, then a simple calculation 
shows that $m_{p,q}=9k+6$. On the other hand, if we resolve the singular point at $z=\infty$, the proper 
transform of $C$ has self-intersection 
$$
(3k+1)^2-(3k-2)^2-9(k-1)=9k+6=m_{p,q},
$$
which shows the bound is sharp. If $q=3k+2$, then $m_{p,q}=9k+9$. However, the self-intersection of the proper transform of $C$ equals 
 $$
 (3k+2)^2-(3k-1)^2-9(k-1)-4=9k+8<m_{p,q}.
 $$
Conjecture 1.15 holds true, but the bound is not sharp when $q=3k+2$.

To get a curve for which the bound is sharp when $q=3k+2$, we consider the following rational bicuspidal curve
in $\C\P^1\times \C\P^1$ (cf. \cite{Moe}): given $1<p<q$ with $gcd(p,q)=1$, we fix an $\alpha>0$ such that
$q<\alpha p$, and let $C$ be the curve in $\C\P^1\times \C\P^1$ given by parametrization 
$([z^p:1], [z^q+z^{\alpha p}:1])$, $z\in \C$. It is of bi-degree $(p,\alpha p)$, and has two singular points, one 
at $z=0$ with one Puiseux pair $(p,q)$, the other at $z=\infty$ with one Puiseux pair $(p,2\alpha p-q)$.
Specializing at the case where $p=3$, $q=3k+2$, and $\alpha=k+1$, we obtain a rational bicuspidal curve $C$
in $\C\P^1\times \C\P^1$ of bi-degree $(3, 3k+3)$, each singularity is of one Puiseux pair, which is 
$(3,3k+2)$ and $(3, 3k+4)$ respectively. Resolving the singularity with one Puiseux pair $(3, 3k+4)$, we note that
the proper transform of $C$ has self-intersection
$$
2\times 3\times (3k+3)- 9\times (k+1)=9k+9=m_{p,q}.
$$
It follows easily that the bound $m_{p,q}$ is optimal for $p=3$. By a similar argument, one can show that
the bound $m_{p,q}$ is optimal for $p=4,5$ as well. 

No counterexamples to Conjecture 1.15 emerged from the bicuspidal curves considered here. In fact, in all the cases we computed, the self-intersection differs from the bound only by $1$ if the bound is not sharp. 

\end{example}

To proceed further, let $C$ be a rational unicuspidal curve with one Puiseux pair $(p,q)$. We blow up at the 
singular point of $C$ and let $\tilde{C}$ be the proper transform. We assume $\tilde{C}$ continue to be singular. 
Then it is easy to see that $\tilde{C}$ is a rational unicuspidal curve with one Puiseux pair $(\tilde{p},\tilde{q})$
or $(\tilde{q},\tilde{p})$, where $\tilde{p}=p$ and $\tilde{q}=q-p$. Furthermore, we observe that $\tilde{C}\cdot \tilde{C}=C\cdot C-p^2$.

\begin{lemma}
One has $m_{p,q}-p^2\leq m_{\tilde{p},\tilde{q}}$. In particular, one cannot improve the upper bound in 
Conjecture 1.15 by passing from $C$ to its proper transform $\tilde{C}$. 
\end{lemma}

\begin{proof}
It suffices to show that $\max (\frac{p}{p-p^\prime},\frac{q}{q-q^\prime})\leq 
\max (\frac{\tilde{p}}{\tilde{p}-\tilde{p}^\prime},\frac{\tilde{q}}{\tilde{q}-\tilde{q}^\prime})$. To this end, we note that
$pq^\prime+q p^\prime=pq+1$ implies $p(q^\prime-p+p^\prime) +\tilde{q} p^\prime=p\tilde{q}+1$. With
$0<p^\prime<p$, and the assumption that $\tilde{C}$ is singular, i.e., $\tilde{q}>1$, it follows easily that
$\tilde{p}^\prime=p^\prime$ and $\tilde{q}^\prime=q^\prime-p+p^\prime$. With this understood, one has 
$\frac{\tilde{q}}{\tilde{q}-\tilde{q}^\prime}=\frac{q-p}{q-q^\prime-p^\prime}>\frac{q}{q-q^\prime}$, from which the lemma follows.

\end{proof}

Let $x\in C$ be a germ of cuspidal singularity. Let $M(x)$ be the sum of the squares of the entries in the multiplicity sequence associated to the minimal resolution of $x\in C$, and let $\ell(x)>1$ be the last entry in the multiplicity sequence. A bound on the self-intersection of a singular symplectic rational curve was introduced by Golla and 
Starkston which involves a term $M(x)+2\ell(x)+1$ as the contribution from a cuspidal singularity $x$. 
See \cite{GS2}, Section 6,  for more details. We clarify the relation between the bound $m_{p,q}$ and the Golla-Starkston bound in the following lemma. 

\begin{lemma}
Suppose $x\in C$ is of one Puiseux pair $(p,q)$. Then 
$$
m_{p,q}\leq M(x)+2\ell(x)+1, 
$$
with ``$=$" if and only if $(p,q)=(m-1,m)$ or $(p,q)=(m, km-1)$ for some $m\geq 3$ and $k\geq 2$. 
\end{lemma}

\begin{proof}
Let $\tilde{x}\in \tilde{C}$ be the proper transform of $x\in C$ in the blowing-up at $x$, and assume $\tilde{x}\in \tilde{C}$ continues to be singular. Then $\tilde{x}\in \tilde{C}$ is of one Puiseux pair $(\tilde{p},\tilde{q})$
or $(\tilde{q},\tilde{p})$, where $\tilde{p}=p$ and $\tilde{q}=q-p$. With this understood, note that 
$M(x)=M(\tilde{x})+p^2$ and $\ell(x)=\ell(\tilde{x})$, which implies that the bound $M(x)+2\ell(x)+1$ behaves like the 
self-intersection $C\cdot C$, i.e., 
$$
M(\tilde{x})+2\ell(\tilde{x})+1=M(x)+2\ell(x)+1-p^2.
$$
On the other hand, by Lemma 6.10, we have $m_{p,q}-p^2\leq m_{\tilde{p},\tilde{q}}$. This allows us to run an inductive argument, i.e., $m_{\tilde{p},\tilde{q}}\leq M(\tilde{x})+2\ell(\tilde{x})+1$ implies 
$m_{p,q}\leq M(x)+2\ell(x)+1$, and furthermore, the former inequality is strict implies the latter is strict as well. 

With the preceding understood, it suffices to consider the following cases:

Case (i): assume $(p,q)=(m-1,m)$ for some $m\geq 3$. In this case, the multiplicity sequence consists of a single entry $m-1$, so that $M(x)=(m-1)^2$ and $\ell(x)=m-1$. It follows that $M(x)+ 2\ell(x)+1=m^2$. On the other hand,
for $(p,q)=(m-1,m)$, $(p^\prime,q^\prime)=(1,m-1)$. It follows easily that $\max (\frac{p}{p-p^\prime},\frac{q}{q-q^\prime})=\frac{q}{q-q^\prime}=m$, so that $m_{p,q}=m^2$ as well. Hence $m_{p,q}= M(x)+2\ell(x)+1$
in this case. 

Case (ii): assume $(\tilde{p},\tilde{q})=(m-1,m)$ or $(\tilde{p},\tilde{q})=(m,m-1)$ for some $m\geq 3$. In the former
case, $(p,q)=(m-1, 2m-1)$, and in the latter case, $(p,q)=(m,2m-1)$. With this understood, note that from the calculation in Case (i), we have $M(\tilde{x})+ 2\ell(\tilde{x})+1=m^2$. Now suppose $(p,q)=(m-1, 2m-1)$.
It is easy to check that $(p^\prime,q^\prime)=(1,2m-3)$, from which it follows that
$\max (\frac{p}{p-p^\prime},\frac{q}{q-q^\prime})=\frac{q}{q-q^\prime}=m-\frac{1}{2}$. With $M(x)=M(\tilde{x})+(m-1)^2$ and $pq=2m^2-3m+1$, we have 
$$
M(x)+2\ell(x)+1=2m^2-2m+1, \mbox{ and } m_{p,q}=2m^2-2m. 
$$
Consequently, $m_{p,q}< M(x)+2\ell(x)+1$ if $(p,q)=(m-1, 2m-1)$. It remains to consider the case where 
$(p,q)=(m,2m-1)$. It is easy to check that in this case, $(p^\prime,q^\prime)=(m-1, 2)$, and consequently, 
$\max (\frac{p}{p-p^\prime},\frac{q}{q-q^\prime})=\frac{p}{p-p^\prime}=m$. Furthermore, note that in this case,
$M(x)=M(\tilde{x})+m^2$ and $pq=2m^2-m$, which implies that $m_{p,q}=2m^2=M(x)+2\ell(x)+1$. 

Case (iii): assume $(p,q)=(m, km-1)$ for some $m\geq 3$ and $k\geq 2$. In this case, it is easy to check that
$(p^\prime,q^\prime)=(m-1, k)$, from which it follows that $\frac{p}{p-p^\prime}=m$ and 
$\frac{q}{q-q^\prime}=\frac{km-1}{km-1-k}<2$. Consequently, $\max (\frac{p}{p-p^\prime},\frac{q}{q-q^\prime})=m$. With $pq=km^2-m$, it implies that $m_{p,q}=km^2$. On the other hand, inductively starting with $k=2$, one can easily show that $M(x)+2\ell(x)+1=km^2$. Hence $m_{p,q}=M(x)+2\ell(x)+1$ in this case. 

Case (iv): assume $(\tilde{p},\tilde{q})=(km-1,m)$ for some $m\geq 3$ and $k\geq 2$, and 
$(p,q)=(km-1,(k+1)m-1)$. First of all, 
$$
M(x)+2\ell(x)+1=M(\tilde{x})+2\ell(\tilde{x})+1+ (km-1)^2=(k^2+k)m^2-2km+1.
$$
Secondly, from the proof of Lemma 6.10, we have $\frac{p}{p-p^\prime}=\frac{\tilde{p}}{\tilde{p}-\tilde{p}^\prime}=\frac{km-1}{km-1-k}$ and $\frac{q}{q-q^\prime}<\frac{\tilde{q}}{\tilde{q}-\tilde{q}^\prime}=m$, 
so that $\max (\frac{p}{p-p^\prime},\frac{q}{q-q^\prime})<m$. With $pq=(k^2+k)m^2-(2k+1)m+1$, it follows easily
that $m_{p,q}< M(x)+2\ell(x)+1$ in this case. 

Lemma 6.11 follows readily from the above analysis. 

\end{proof}

\section{Embeddings of singular Lagrangian $\R\P^2$s}

This section is devoted to discussions concerning embeddings of the compact domain $W_0$ and singular Lagrangian $\R\P^2$'s, and contains the proof of Theorems 1.1, 1.5, 1.8 and 1.9, as well as the proof of 
Proposition 1.6.

Let $Q_8$ denote the finite group of order $8$ generated by the quaternions $i,j,k$, which acts complex linearly on $\C^2$ by identifying the group of unit quaternions canonically with $SU(2)$.
In particular, the induced action on the unit sphere $\s^3\subset \C^2$, which is free, preserves the Hopf fibration. Let $M_0=\s^3/Q_8$ be the quotient manifold, which inherits a Seifert fibration 
$\pi: M_0\rightarrow S$ where $S$ is an orbifold $2$-sphere with three singular points of multiplicity $2$. We give $M_0$ the non-standard orientation, so that as a Seifert manifold, $M_0=M((2,-1),(2,-1),(2,1))$; note that $e(M_0)=\frac{1}{2}$ and $e_0(M_0)=-1$. The following lemma is straightforward; we leave its proof to the reader. 

\begin{lemma}
{\em (1)} The three singular fibers of $\pi: M_0\rightarrow S$ represent the classes of 
$i,j,k$ in $\pi_1(M_0)$, which is canonically identified with $Q_8$. Moreover, the regular fiber represents the class $-1\in Q_8$.

{\em (2)} $H_1(M_0,\Z)=\Z_2\times \Z_2$, where the three involutions in $H_1(M_0,\Z)$ are given by the classes of $i,j,k$. In particular, the regular fiber, which represents the class of $-1\in Q_8$, 
is null-homologous in $M_0$. 
\end{lemma}

It is also known that $M_0$ is the total space of a non-orientable circle bundle over $\R\P^2$ with
Euler number $-2$, $pr: M_0\rightarrow \R\P^2$, and $W_0$ is the associated disk bundle. We remark 
that there is a canonically embedded $\R\P^2$ in $W_0$, i.e., the zero-section. The following lemma will be used later in this section. 

\begin{lemma}
There is an involution $\tau: M_0\rightarrow M_0$ which preserves both $\s^1$-fibrations 
$\pi: M_0\rightarrow S$ and $pr: M_0\rightarrow \R\P^2$, in particular, extending naturally to
an involution $\tau: W_0\rightarrow W_0$, such that $\tau: M_0\rightarrow M_0$ switches two of the singular fibers of $\pi: M_0\rightarrow S$ and leaves the third singular fiber invariant. Moreover,
the singular fiber which is invariant under $\tau$ bounds a smoothly embedded disk in $W_0$
which intersects the embedded $\R\P^2$ transversely in a single point. 
\end{lemma}

\begin{proof}
We shall first describe $M_0$ as a union of $S(T^\ast (Mb))$ with $D^2\times \s^1$, 
where $S(T^\ast (Mb))$ is the associated non-orientable circle bundle of the cotangent bundle 
$T^\ast (Mb)$ of the M\"{o}bius band $Mb$, which gives naturally a structure 
of $M_0$ as a non-orientable circle bundle over $\R\P^2$, to be denoted by 
$pr: M_0\rightarrow \R\P^2$. This is accomplished by describing the corresponding $\s^1$-action on 
$S(T^\ast (Mb))$ and explaining how to extend it over $D^2\times \s^1$, so that the union of 
$S(T^\ast (Mb))$ with $D^2\times \s^1$ is identified with the Seifert space 
$M_0=M((2,-1),(2,-1),(2,1))$. 

We begin by identify the M\"{o}bius band $Mb$ with the quotient of the annulus $A=\{(x,y,z)|x^2+y^2+z^2=1, |z|\leq 1/2\}$ under the antipodal map $(x,y,z)\mapsto (-x,-y,-z)$. Furthermore, we equip $Mb$ with the $\s^1$-action which is induced by the $\s^1$-action on $A$ given by the rotations about the $z$-axis. The $\s^1$-action naturally lifts to an $\s^1$-action on $S(T^\ast (Mb))$. In order to understand it, we note that the $\s^1$-action on $A$ identifies $S(T^\ast A)$ with the product
$\s^1\times ([0,1]\times \s^1)$, where the orbits of the $\s^1$-action are given by the first $\s^1$-factor,
and the $\s^1$-factor in $[0,1]\times \s^1$ corresponds to the fibers of the projection 
$pr: S(T^\ast (A))\rightarrow A$. With this understood, the antipodal map induces a fiber-preserving, free involution on $\s^1\times ([0,1]\times \s^1)$. The induced involution on the base $[0,1]\times \s^1$, which is orientation-preserving, swaps the two boundary components but reverses the orientation. So it must have two fixed points in its interior by the Lefschetz fixed point theorem. The two fixed points represent the two singular fibers of the Seifert fibration on $S(T^\ast (Mb))$, which can be seen as follows. Notice that the pre-image of
$pr: S(T^\ast (Mb))\rightarrow Mb$ over the core of $Mb$, i.e., the circle given by $z=0$, is a Klein bottle. If we identify the pre-image of $pr: S(T^\ast (A))\rightarrow A$ over the core $z=0$ with 
$\s^1\times \s^1$, with the core $z=0$ given by the first $\s^1$-factor, then the Klein bottle is the quotient of $\s^1\times \s^1$ by the involution $(u,v)\mapsto (-u,\bar{v})$. With this understood, it follows easily that the two singular fibers are the images of $(u,\pm 1)$ under the involution 
$(u,v)\mapsto (-u,\bar{v})$. Finally, we note that the Seifert fibration on $S(T^\ast (Mb))$ comes with a natural trivialization on the boundary, which gives rise to an identification of the boundary with 
$\s^1\times \s^1$, such that the first $\s^1$-factor parametrizes the orbits of the $\s^1$-action and the second 
$\s^1$-factor parametrizes the fibers of $pr: S(T^\ast (Mb))|_{\partial Mb}\rightarrow \partial Mb$. Moreover, with this trivialization over the boundary, the Seifert invariants of the two singular fibers are $(2,-1)$ and $(2,1)$
respectively. Adapting the notations from \cite{N}, we let $H$ be the class of the $\s^1$-orbits and $Q$ be the boundary of the base $2$-orbifold but given with the opposite orientation, then $(H,Q)$ forms a positively oriented basis on the boundary of  $S(T^\ast (Mb))$, i.e., $H\cdot Q=1$. 

With the preceding understood, we shall attach a solid torus $D^2\times \s^1$ to $S(T^\ast (Mb))$
along the boundaries as follows: let $(\bar{m},\bar{l})$ be the standard meridian-longitude pair for 
$D^2\times \s^1$, then the gluing map is uniquely determined by the following relations:
$$
\bar{m}=2Q-H, \;\; \bar{l}=Q.
$$
It is easy to see that the resulting $3$-manifold is $M_0=M((2,-1), (2,-1), (2,1))$. 
Moreover, it also naturally exhibits $M_0$ as a non-orientable circle bundle over $\R\P^2$ with Euler
number $-2$, denoted by $pr: M_0\rightarrow \R\P^2$. 

We proceed further by describing the involution $\tau: M_0\rightarrow M_0$. We begin by considering 
the involution $\tau: Mb\rightarrow Mb$ induced by the involution on $A$ given by $(x,y,z)\mapsto (x,-y,z)$. There is an induced involution on $S(T^\ast (Mb))$, also denoted by $\tau$ for simplicity, which preserves the Seifert fibration on $S(T^\ast (Mb))$ (as $\tau: Mb\rightarrow Mb$ commutes with the $\s^1$-action on $Mb$). To see that $\tau: S(T^\ast (Mb))\rightarrow S(T^\ast (Mb))$ switches the two singular fibers, we simply note that $\tau: Mb\rightarrow Mb$ acts as a reflection on the core of $Mb$,
and the two singular fibers are the images of $(u,\pm 1)$ under the involution 
$(u,v)\mapsto (-u,\bar{v})$, where $\pm 1$ represent the unit cotangent vectors lying in the cotangent bundle of the core. 
To understand $\tau$ on the boundary of $S(T^\ast (Mb))$, we note that $\tau: Mb\rightarrow Mb$ acts as a reflection on $\partial Mb$, and consequently, the action of $\tau$ on the boundary $\s^1\times \s^1$ of $S(T^\ast (Mb))$ is given by the diagonal conjugation, i.e., $(z,w)\mapsto (\bar{z},\bar{w})$,
which extends linearly over $D^2\times \s^1$, as the gluing map given by the relations 
$\bar{m}=2Q-H, \;\; \bar{l}=Q$ is equivariant with respect to the diagonal conjugation on the boundary 
$\s^1\times \s^1$ of $D^2\times \s^1$. Note that the resulting involution $\tau: M_0\rightarrow M_0$
switches the two singular fibers lying in $S(T^\ast (Mb))$ and leaves the core of $D^2\times \s^1$
invariant, which is the third singular fiber of $M_0=M((2,-1), (2,-1), (2,1))$. Finally, note that 
the relation $\bar{l}=Q$ shows that $Q$, which is a fiber of $pr: S(T^\ast (Mb))\rightarrow Mb$,
is a longitude in $D^2\times \s^1$. It follows easily that the core of $D^2\times \s^1$ bounds a 
smoothly embedded disk in $W_0$, intersecting the zero-section (i.e., the embedded $\R\P^2$ in $W_0$) transversely at a single point. This finishes the proof of the lemma. 

\end{proof}

We remark that $\tau: M_0\rightarrow M_0$ is a lifting of an involution $\kappa: \R\P^2\rightarrow \R\P^2$ under
$pr: M_0\rightarrow \R\P^2$. Moreover, the fixed-point set of $\kappa$ consists of an isolated point and a circle,
so it can be identified with the involution induced by the reflection $(x,y,z)\mapsto (x,-y,z)$ on $\s^2\subset \R^3$.

We shall fix the following notations throughout this section: denote by $F_0$ the singular fiber invariant under 
$\tau: M_0\rightarrow M_0$ and by $F_1,F_2$ the remaining two singular fibers swapped under $\tau$. 

\vspace{2mm}

Next we consider tight contact structures on $M_0$. It is an important fact that there is a unique oriented positive tight contact structure on $M_0$ up to a contactomorphism (cf. \cite{GLS}). 
On the other hand, by Theorem 2.3(2) there is a tight $\s^1$-invariant, non-transverse contact 
structure on $M_0$, to be denoted by $\xi_{inv}$, whose dividing set is a circle surrounding one of the singular points in the base orbifold $S$. The uniqueness from \cite{GLS} implies that any tight contact structure on $M_0$ is contactomorphic to $\xi_{inv}$. 

In order to be more precise, we introduce the following notations: let $z_1,z_2,z_3$
be the singular points of $S$, and let $f_1,f_2,f_3$ be the fibers over them. Then we
assume the dividing set of $\xi_{inv}$ is a circle bounding a disk neighborhood of $z_1$. 
Furthermore, we shall fix an orientation of $\xi_{inv}$ such that the fiber $f_1$ is a negative 
transversal to $\xi_{inv}$. As a corollary of Lemma 2.2, we note the following lemma. 

\begin{lemma}
There is an $\s^1$-equivariant diffeomorphism $\psi: M_0\rightarrow M_0$ which switches the fibers $f_2,f_3$ and
leaves the fiber $f_1$ invariant, such that $\psi_\ast(\xi_{inv})=\xi_{inv}$.
\end{lemma}

With the preceding understood, we proceed to construct the symplectic cap $V_0$. We start with 
the Hamiltonian $\s^1$-action on 
$(\C^2,\omega_0)$, defined by $\lambda\cdot (z_1,z_2)=(\lambda z_1,\lambda^2 z_2)$, $\forall \lambda\in\s^1$, with Hamiltonian function $H=\frac{1}{2}(|z_1|^2+2|z_2|^2)$. It defines a Seifert
fibration on $H^{-1}(1)$, which is diffeomorphic to $\s^3$, with one singular fiber of multiplicity $2$
given by $z_1=0$. We let $K_1$, $K_2$ be the two copies of unknot in $H^{-1}(1)$ defined by
$z_2=0$ and $z_2=\epsilon z_1^2$ respectively. We note that $K_1,K_2$ are both regular fibers
of the Seifert fibration on $H^{-1}(1)$, and the linking number $lk(K_1,K_2)=2$ with fiber
orientation. 

\begin{lemma}
There is a relative handlebody $Z$ of two symplectic $2$-handles built on $H^{-1}(1)$, where the
two $2$-handles are attached along $K_1$ and $K_2$ with framings $4$ and $0$ relative to the zero-framing, such that $Z$ is symplectic with two concave contact boundaries $(H^{-1}(1), \xi_{std})$ and
$(M_0,\xi_{inv})$. Here $\xi_{std}$ is the standard tight contact structure on $H^{-1}(1)\cong \s^3$. 
\end{lemma}

\begin{proof}
We begin by exhibit a rational open book on $H^{-1}(1)$ with binding components $K_1,K_2$
using Theorem 3.4. To this end, we regard $H^{-1}(1)$ as $M((1,0),(1,0),(2,1))$, with $(\alpha_i,\beta_i)=(1,0)$ for $i=1,2$ and $(\alpha_j,\beta_j)=(2,1)$. With this understood, let $n=4$,
$c_1=c_2=3$ and $b_1=b_2=1$ in Theorem 3.4. One can check easily that the conditions in Theorem 3.4
are satisfied, and moreover, $\frac{\beta_i}{\alpha_i}+b_i-\frac{c_i}{n}=\frac{1}{4}>0$ for both
$i=1,2$. This proves the existence of the desired rational open book on $H^{-1}(1)$. It is
easy to check that the multiplicities at $K_1,K_2$ are both equal to $1$, i.e., $p_1=p_2=1$, using the formula in
Remark 3.5. Finally, we note that the contact structure supported by the rational open book, being $\s^1$-invariant and transverse,  is the standard tight contact structure $\xi_{std}$ on $H^{-1}(1)\cong \s^3$. 

With the preceding understood, if we attach a symplectic $2$-handle $H_1$ along $K_1$ with 
framing $F_1=6$ and attach a symplectic $2$-handle $H_2$ along $K_2$ with framing 
$F_2=2$, both relative to the page framing, then by Gay \cite{G} the resulting relative handlebody $Z$ is symplectic with two concave contact boundaries, one being $(H^{-1}(1), \xi_{std})$. By Theorem 4.4, it is easy to see that the other component of $-\partial Z$ is $M((-2,1),(2,1),(2,-1))$, which is $M_0$. Moreover, note that one of the Seifert invariants is $(-2,1)$, with $-2<0$, it follows from Theorem 4.1 that the contact structure on $M_0$ is $\xi_{inv}$. 

It remains to show that the $2$-handles $H_1,H_2$ have framings $4$ and $0$ respectively 
relative to the zero-framing. To see this, we observe that respect to the zero-framing of $K_1,K_2$,
the slope of the Seifert fibration on $H^{-1}(1)$ equals $2$, because both $K_1,K_2$ are regular
fibers and $lk(K_1,K_2)=2$. With this understood, it follows easily from Lemma 4.6 that the framing of
$K_1$ equals $2-\frac{-2}{1}=4$, and the framing of $K_2$ equals $2-\frac{2}{1}=0$.
This finishes the proof. 
\end{proof}

We remark that by identifying one of the contact boundary of $Z$ with $(M_0,\xi_{inv})$, we implicitly
identified the ascending sphere of the $2$-handle $H_1$ with the singular fiber $-f_1$ (i.e., $f_1$ 
with the opposite orientation) of $M_0$. To fix the notation and without loss of generality, we shall identify
the ascending sphere of the $2$-handle $H_2$ with the singular fiber $f_2$. 

With the proceeding understood, the symplectic cap $V_0$ is simply the union of $H^{-1}([0,1])$ and $Z$, which is symplectic with concave contact boundary $(M_0,\xi_{inv})$. It remains to show that $V_0$ contains a pair of embedded symplectic spheres $S_1,S_2$ with self-intersection $4$ and $0$, such that $S_1,S_2$
intersect at a single point with a tangency of order $2$. To this end, we need to explain how to choose 
an $\s^1$-invariant contact form on $H^{-1}(1)$ in the construction of $Z$. Recall in Lemma 6.1,
we showed that there is a $1$-form $\tau$ such that the standard symplectic structure $\omega_0$ 
on $H^{-1}([0,1])$ is given by $\omega_0=d\tau$. Moreover, the restriction of $\tau$ on $H^{-1}(1)$, 
denoted by $\tau_1$, is a contact form such that $i\tau_1$ is a connection $1$-form for the Seifert 
fibration on $H^{-1}(1)$. With this understood, we shall construct an $\s^1$-invariant contact form $\alpha_1$ 
on $H^{-1}(1)$ as we did in the proof of Theorem 1.17, such that $i \alpha_1$ is a connection $1$-form, and moreover, $\alpha_1$ and $\tau_1$ agree in the complement of a sufficiently small neighborhood of
$K_1, K_2$. (This can be done by the same argument as in Section 6.) Then as we argued in the proof of 
Theorem 1.17, by Lemma 2.6 in \cite{C1}, there is an $\s^1$-equivariant diffeomorphism 
$\phi: H^{-1}(1)\rightarrow H^{-1}(1)$, supported in the neighborhood of 
$K_1,K_2$, such that $\phi^\ast \tau_1=\alpha_1$. Let $K_1^\prime, K_2^\prime$ be the image
of $K_1,K_2$ under $\phi$. Then by the version of Lemma 6.2 as stated in Remark 6.3, $K_1^\prime, K_2^\prime$
bound embedded symplectic disks $D_1^\prime$, $D_2^\prime$ in $H^{-1}([0,1])$, such that
near $0\in H^{-1}([0,1])$, $D_1^\prime$ is given by $z_2=0$ and $D_2^\prime$ is given by
$z_2=\epsilon z_1^2$. With this understood, we let $V_0=Z\cup_\phi H^{-1}([0,1])$. Denote by
$D_1,D_2$ the core disk of the $2$-handles $H_1,H_2$ respectively. Then $S_1:=D_1\cup_\phi 
D_1^\prime$ is an embedded symplectic sphere of self-intersection $4$ in $V_0$,
$S_2:=D_2\cup_\phi D_2^\prime$ is an embedded symplectic sphere of self-intersection $0$ in $V_0$, and $S_1,S_2$ intersect only at $0\in H^{-1}([0,1])$ with a tangency of order $2$. This completes the construction of the symplectic cap $V_0$.

\vspace{2mm}

Now let $\omega$ be any symplectic structure on $W_0$ which has a convex contact boundary,
and denote by $\xi_{fil}$ the corresponding fillable contact structure on $M_0=\partial W_0$. 
The uniqueness of tight contact structures on $M_0$ (cf. \cite{GLS}) implies that there is a contactomorphism $\Phi: (M_0,\xi_{inv})\rightarrow (M_0,\xi_{fil})$. The corresponding symplectic 
$4$-manifold $X:=V_0\cup_\Phi W_0$ is diffeomorphic to a Hirzebruch surface, because the existence
of the symplectic sphere $S_2\subset V_0$ of self-intersection $0$ (note that this also follows independently from Lemma 5.4). In order to determine the homeomorphism type of $X$, we need to analyze the smooth isotopy class of the contactomorphism $\Phi$. 

To this end, we recall that the mapping class group $\pi_0(Diff(M_0))$, i.e., the group of orientation-preserving self-diffeomorphisms of $M_0$ modulo smooth isotopy, is isomorphic to $S_3$ (cf. \cite{P}), which can be recognized by its action on $H_1(M_0,\Z)$, i.e,, as permutations of the three involutions in $H_1(M_0,\Z)$. In particular, it acts transitively on the three nontrivial elements and each nonzero element of $H_1(M_0,\Z)$ is fixed by a unique order $2$ element of the mapping class group $\pi_0(Diff(M_0))$. On the other hand, $\pi_0(Diff(M_0))$ acts naturally on the set of isomorphism classes of $spin^c$ structures on $M_0$, which admits a natural free and transitive action by 
$H_1(M_0,\Z)=H^2(M_0,\Z)$. It follows easily that the action of $\pi_0(Diff(M_0))$ on the set of isomorphism classes of $spin^c$ structures has two orbits; one $spin^c$ structure is fixed by the entire group $\pi_0(Diff(M_0))$, and the other three $spin^c$ structures are acted upon transitively by 
$\pi_0(Diff(M_0))$. In particular, each of these three $spin^c$ structures is fixed by a unique order $2$ element of $\pi_0(Diff(M_0))$. Finally, there are exactly three tight contact structures on $M_0$ up to a contact isotopy (cf. \cite{GLS}, Theorem 1.1, see also \cite{Ma}, Claim 3.1). Since two tight contact structures on $M_0$ are contact isotopic if and only if the corresponding $spin^c$ structures are isomorphic, it is easy to see that the natural action of $\pi_0(Diff(M_0))$ on the set of isotopy classes of tight contact structures of $M_0$ is transitive, and each isotopy class of tight contact structures is fixed 
by a unique order $2$ element of $\pi_0(Diff(M_0))$.

To proceed further, note that $H_1(M_0,\Z)=\Z_2\times \Z_2$ and $H_1(W_0,\Z)=\Z_2$, and $i_\ast: H_1(M_),\Z)\rightarrow H_1(W_0,\Z)$ is onto. It follows easily that the kernel of
$i_\ast$ consists of a single non-zero element of $H_1(M_0,\Z)$. In particular, there is a 
unique mapping class of order $2$ in $\pi_0(Diff(M_0))$ which preserves the kernel of 
$i_\ast: H_1(M_0,\Z)\rightarrow H_1(W_0,\Z)$. (It is easy to see that this order $2$ element is represented by the involution $\tau$ from Lemma 7.2.) On the other hand, there is the exact sequence
$$
0 \stackrel{j^\ast}{\longrightarrow} H^2(W_0,\Z)\stackrel{i^\ast}{\longrightarrow}
H^2(M_0,\Z)\stackrel{\delta}{\longrightarrow} H^3(W_0, M_0,\Z)\stackrel{j^\ast}{\longrightarrow} 0,
$$
where we use the fact that $H^2(W_0,M_0,\Z)=H_2(W_0,\Z)=H_2(\R\P^2,\Z)=0$ and $H^3(W_0,\Z)=H^3(\R\P^2,\Z)=0$. Note that $H^3(W_0,M_0,\Z)=H_1(W_0,\Z)=\Z_2$. We conclude 
that $H^2(W_0,\Z)=\Z_2$, which is mapped injectively to $H^2(M_0,\Z)$, and its image in 
$H^2(M_0,\Z)$ corresponds to the kernel of $i_\ast: H_1(M_0,\Z)\rightarrow H_1(W_0,\Z)$ under Poincar\'{e} duality.

With the preceding understood, we denote by $\delta_0$ the non-zero element in the kernel of
$i_\ast: H_1(M_0,\Z)\rightarrow H_1(W_0,\Z)$ (note that $\delta_0$ is the homology class of the
singular fiber $F_0$ as it bounds an embedded disk in $W_0$, cf. Lemma 7.2), and by $\delta_1$, 
$\delta_2=\delta_1+\delta_0$, the other two non-zero elements in $H_1(M_0,\Z)$. Furthermore, we denote by $s$ the $spin^c$-structure of $M_0$ which is fixed by the entire mapping class group 
$\pi_0(Diff(M_0))$. Then the other three
$spin^c$-structures of $M_0$ can be conveniently denoted by $\delta_0\cdot s$, $\delta_1\cdot s$, 
and $\delta_2\cdot s$. Note that with these notations, the action of $\pi_0(Diff(M_0))$ on the set of
$spin^c$-structures coincides with its action on $H_1(M_0,\Z)$. 

The set of $spin^c$-structures of $W_0$ consists of two elements, which is acted upon by 
$H^2(W_0,\Z)$ freely and transitively. Consequently, as $i^\ast: H^2(W_0,\Z)\rightarrow H^2(M_0,\Z)$ is injective, there are precisely two $spin^c$-structures on $M_0$ which are induced by a $spin^c$-structure on $W_0$. It is easy to see that there are only two possibilities for the set of the two induced $spin^c$-structures of $M_0$, either $\{s,\delta_0\cdot s\}$ or $\{\delta_1\cdot s,\delta_2\cdot s\}$, because it is invariant under the action of $\delta_0$. We point out that the two possibilities can be distinguished as follows: let $\psi$ be a self-diffeomorphism of $M_0$ which can be extended to a self-diffeomorphism of $W_0$ (e.g. the involution $\tau$ in Lemma 7.2), then $\psi$ fixes each element in $\{s,\delta_0\cdot s\}$ and interchanges the two elements in $\{\delta_1\cdot s,\delta_2\cdot s\}$. 

\begin{lemma}
The set of $spin^c$-structures of $M_0$ induced from $W_0$ is $\{\delta_1\cdot s,\delta_2\cdot s\}$.
\end{lemma}

\begin{proof}
Suppose to the contrary that this is not true. Then the involution $\tau: M_0\rightarrow M_0$
from Lemma 7.2 fixes each of the induced $spin^c$-structures of $M_0$. In particular, it fixes
the $spin^c$-structure associated to the fillable contact structure $\xi_{fil}$, and consequently
$\tau_\ast(\xi_{fil})$ is contact isotopic to $\xi_{fil}$ (cf. \cite{GLS}). Let  
$\Phi: (M_0,\xi_{inv})\rightarrow (M_0,\xi_{fil})$ be a contactomorphism which exists by the uniqueness
of tight contact structures on $M_0$. Then the involution $\Phi^{-1}\circ \tau \circ \Phi$ preserves
$\xi_{inv}$ up to contact isotopy. On the other hand, by Lemma 7.3, there is a diffeomorphism
$\psi: M_0\rightarrow M_0$ such that $\psi_\ast (\xi_{inv})=\xi_{inv}$. It follows that 
$\Phi^{-1}\circ \tau \circ \Phi$ and $\psi$ must be smoothly isotopic. Since $\tau$ switches the
singular fibers $F_1,F_2$ and leaves $F_0$ invariant, and $\psi$ switches the singular fibers
$f_2,f_3$ and leaves $f_1$ invariant, it follows easily that $\Phi$ must be smoothly isotopic 
to a fibration-preserving diffeomorphism $\phi: M_0\rightarrow M_0$ such that $\phi(f_1)=F_0$. 
Finally, note that $X:=V_0\cup_\Phi W_0$ is diffeomorphic to $V_0\cup_\phi W_0$, which we shall identify it to, i.e., $X=V_0\cup_\phi W_0$.

With the preceding understood, we derive a contradiction as follows. Note that by Lemma 7.2, 
$F_0$ bounds an embedded disk $D$ in $W_0$. On the other hand, $f_1$ bounds the co-core disk
$D^\prime$ of the $2$-handle $H_1$ in $V_0$. The union $D^\prime\cup_\phi D$ is an embedded
sphere $S$ in $X$. We claim $S$ has self-intersection $0$. To see this, we go back to the proof of
Lemma 7.2, and recall that $M_0$ is the result of attaching a solid torus $D^2\times \s^1$ to 
$S(T^\ast (Mb))$ along the boundaries with a gluing map determined by the relations
$\bar{m}=2Q-H$ and $\bar{l}=Q$, where $(\bar{m},\bar{l})$ is the standard meridian-longitude 
pair for $D^2\times \s^1$, $Q$ is a fiber of $pr: M_0\rightarrow \R\P^2$ and $H$ is a regular
fiber of the Seifert fibration on $M_0$. It follows easily that the longitude $\bar{l}$ defines the zero-framing of $F_0$ and the slope of the regular fibers of the Seifert fibration on $M_0$ equals 
$-\frac{1}{2}$ with respect to the zero-framing. Now regarding $D^\prime$ as the core disk of a
$2$-handle attached to $M_0$ along $F_0$, its framing relative to the zero-framing equals 
$-\frac{1}{2}-\frac{1}{-2}=0$ by Lemma 4.6. This proves that the sphere $S$ has self-intersection 
$0$.

To proceed further, we note that the sphere $S$ intersects the symplectic sphere $S_1$ transversely
at a single point. This implies that $S_1$ must be a primitive class, and as $S_1\cdot S_1=4$, it implies 
that $X=V_0\cup_\phi W_0$ must be $\s^2\times \s^2$. With this understood, we let $e_1,e_2$
be a standard basis of $H^2(X)$ such that $c_1(K_X)=-2e_1-2e_2$. It follows easily that
the symplectic spheres $S_1,S_2$ must have classes $S_1=2e_1+e_2$ and $S_2=e_2$
without loss of generality. With this understood, $S\cdot S_1=\pm 1$
and $S\cdot S=0$ implies that $S=\pm e_1$. But this is a contradiction to the fact that $S$ and $S_2$
are disjoint. 

\end{proof}

It is not clear how to settle the ambiguity concerning the induced $spin^c$-structures of $M_0$ 
without appealing to the existence of the symplectic cap $V_0$. 

\vspace{2mm}

{\bf Proof of Proposition 1.6:}

\vspace{2mm}

Part (1) follows immediately from Lemma 7.2. Part (2) follows from Lemma 7.5, as $\tau$ acts nontrivially on the set 
$\{\delta_1\cdot s,\delta_2\cdot s\}$, and the set of $spin^c$-structures of $W_0$ is mapped injectively into the
set of $spin^c$-structures of $M_0$ via restriction, because $i^\ast: H^2(W_0,\Z)\rightarrow H^2(M_0,\Z)$ is injective. 

\vspace{2mm}

Now back to the analysis for the contactomorphisms $\Phi: (M_0,\xi_{inv})\rightarrow (M_0,\xi_{fil})$.
Without loss of generality, we may assume the $spin^c$-structure of $\xi_{fil}$ is 
$\delta_1\cdot s$. Furthermore, observe that there is a unique order $2$ mapping class in 
$\pi_0(Diff(M_0))$ which fixes the $spin^c$-structure $\delta_1\cdot s$, and it can be realized 
by a contactomorphism of $\xi_{fil}$. It follows immediately that there are precisely two contactomorphisms (up to a smooth isotopy) from $(M_0,\xi_{inv})$ to $(M_0,\xi_{fil})$, 
to be denoted by $\Phi_1$, $\Phi_2$, and $\Phi_1$ and $\Phi_2$ have the property that $\Phi_2\circ \Phi_1^{-1}$, a self-diffeomorphism of $M_0=\partial W_0$, is not extendable to a self-diffeomorphism of $W_0$. With this understood, for $i=1,2$, we denote by $X_i$ the symplectic $4$-manifold obtained by gluing 
$V_0$ to $W_0$ via $\Phi_i$ along the contact boundaries. As $\Phi_2\circ \Phi_1^{-1}$ is not
extendable to a self-diffeomorphism of $W_0$, $X_1,X_2$ are not necessarily diffeomorphic. 

The following lemma gives a characterization of the smooth isotopy classes of $\Phi_1$, $\Phi_2$, where we recall that $f_1$ is the singular fiber of the Seifert fibration of $M_0=-\partial V_0$ such that the corresponding singular point in the base $2$-orbifold $S$ is encircled by the dividing set of 
$\xi_{inv}$ as the boundary of a small neighborhood of it, and $f_1,f_2$ are the two singular fibers which bound the co-core disks of the two $2$-handles in the relative handlebody $Z$ constructed in
Lemma 7.4. 

\begin{lemma}
Assuming that the $spin^c$-structure of $\xi_{fil}$ is $\delta_1\cdot s$, the two distinct contactomorphisms $\Phi_1,\Phi_2: (M_0,\xi_{inv})\rightarrow (M_0,\xi_{fil})$ are characterized 
by the property that the homology class of both $\Phi_1(f_1)$, $\Phi_2(f_1)$ is 
$\delta_1\in H_1(M_0,\Z)$. Moreover, precisely one of $\Phi_1(f_2)$, $\Phi_2(f_2)$ carries the homology 
class $\delta_0$ which is trivial in $W_0$. 
\end{lemma}

\begin{proof} 
Let $\Phi:(M_0,\xi_{inv})\rightarrow (M_0,\xi_{fil})$ be any contactomorphism, and let $\psi$ be the self-diffeomorphism of $M_0$ from Lemma 7.3. Then $\Phi\circ\psi: (M_0,\xi_{inv}) \rightarrow (M_0,\xi_{fil})$ must be the other contactomorphism isotopically distinct from $\Phi$. 
Since $\psi$ is characterized by the fact that its
mapping class is the unique order $2$ element fixing the homology class of the singular fiber $f_1$,
and $\Phi_2\circ \Phi_1^{-1}$ is characterized by the property that its mapping class is the unique 
order $2$ element fixing the $spin^c$-structure $\delta_1\cdot s$, it follows easily that 
$\Phi_1,\Phi_2$ are characterized by the property that both $\Phi_1(f_1)$, $\Phi_2(f_1)$ carry the homology class
$\delta_1\in H_1(M_0,\Z)$. 

It remains to show that precisely one of $\Phi_1(f_2)$, $\Phi_2(f_2)$ has the homology class 
$\delta_0$ which is trivial in $W_0$. To this end, we note that $\psi: (M_0,\xi_{inv})\rightarrow (M_0,\xi_{inv})$ switches the two singular fibers $f_2,f_3$, and on the other hand, both $\Phi_1(f_1)$, $\Phi_2(f_1)$ have the homology class $\delta_1\in H_1(M_0,\Z)$. It follows easily that precisely one of $\Phi_1(f_2)$, $\Phi_2(f_2)$ 
carries the homology class $\delta_0\in H_1(M_0,\Z)$.

\end{proof}

We shall fix our notations such that  $\Phi_1(f_2)$ has the homology class $\delta_0$.

\begin{lemma}
{\em(}1{\em)} There is an embedded sphere $S$ of self-intersection $-1$ in $X_1$, which intersects 
transversely with each of the following in a single point: the embedded $\R\P^2$ in $W_0$ and the symplectic sphere $S_2$ in $V_0$. In particular, $X_1=\C\P^2\# \overline{\C\P^2}$. 

{\em(}2{\em)} There is an embedded sphere $\tilde{S}$ of self-intersection $0$ in $X_2$, which intersects 
transversely with each of the following in a single point: the embedded $\R\P^2$ in $W_0$, the symplectic spheres $S_1$ and $S_2$ in $V_0$. In particular, $X_2=\s^2\times \s^2$. 
\end{lemma}

\begin{proof}
First of all, each of $\Phi_1,\Phi_2$ is smoothly isotopic to a Seifert fibration preserving diffeomorphism, to be denoted by $\phi_1,\phi_2$ respectively. Moreover, note that $\delta_0$
is the homology class of the singular fiber $F_0$. Without loss of generality, assume $F_1,F_2$
have the homology classes $\delta_1,\delta_2$ respectively. With this understood, it follows easily
from Lemma 7.6 that $\phi_1(f_1)=\phi_2(f_1)=F_1$, $\phi_1(f_2)=F_0$, and $\phi_2(f_2)=F_2$.
Finally, we shall identify $X_i=V_0\cup_{\phi_i} W_0$, for $i=1,2$.

With the preceding understood, we consider $X_1$ first. In this case, as $\phi_1(f_2)=F_0$,
the co-core disk of the $2$-handle $H_2$ in $Z$ and the embedded disk in $W_0$ bounded by
$F_0$ form an embedded sphere $S$ in $X_1$, which is easily seen to have the intersection property we claimed. Finally, we calculate $S\cdot S=-1$ using Lemma 4.6 as we employed in the proof of Lemma 7.5. 

The argument for $X_2$ is similar but slightly more complex. We begin by fixing a choice of Seifert invariants at the singular fibers
$F_0,F_1,F_2$ to be $(2,-1)$, $(2,1)$, and $(2,-1)$ respectively. With such a choice fixed, it determines a specific trivialization of the Seifert fibration in the complement of the singular fibers. 
More precisely, if we let $S_0$ be the oriented surface with boundary which is obtained by removing 
a disk neighborhood of each singular point of the base $2$-orbifold $S$, then the said trivialization is given by a section $R$ over $S_0$ (cf. \cite{N}). With this understood, for $i=0,1,2$, let $Q_i$ be the corresponding boundary component of $-R$, let $H_i$ be the fiber over $Q_i$, and let $\bar{m}_i$ be the meridian of a regular neighborhood of the singular fiber $F_i$. Moreover, in order for the orientation of $F_1$ to match with that of $\phi_2(B_1^\prime)$ as $B_1^\prime=-f_1$ (here 
$B_1^\prime$ is a binding component of the rational open book $(B^\prime,\pi^\prime)$ 
on $M_0=-\partial V_0$ resulted from the construction of $Z$ in Lemma 7.4), we shall reverse the orientation of $F_1$. As a result, the Seifert invariant at $F_1$ is changed to $(-2,-1)$. With
this understood, we have the following relations
$$
\bar{m}_0=2Q_0-H_0, \;\;  \bar{m}_1=-2Q_1-H_1, \;\; \bar{m}_2=2Q_2-H_2.
$$
From these relations, it follows easily that each $Q_i$ is a longitude at $F_i$, which will be denoted 
by $\bar{l}_i$ in what follows. With this understood, we observe that $Q_0$ is a fiber of $pr: M_0
\rightarrow \R\P^2$, so that it bounds an embedded disk $\tilde{D}$ in $W_0$ which intersects transversely with the embedded $\R\P^2$ in $W_0$ at a single point. On the other hand, for $i=1,2$, $Q_i$ and $F_i$ bound an embedded annulus $A_i$ in the regular neighborhood of $F_i$. If we denote by $D_i$, for $i=1,2$, the co-core disk of the $2$-handle $H_i$ in $Z$, then it is easy to see that the union of $\tilde{D}$, $R$, $A_1$, $A_2$, $D_1$, and $D_2$ forms an embedded sphere 
$\tilde{S}$ in $X_2$ with the desired intersection property 
(as $\phi_2(f_1)=F_1$, $\phi_2(f_2)=F_2$). 

It remains to determine the self-intersection of $\tilde{S}$. To this end, we simply observe that
the regular fibers of the Seifert fibration have a slope $\frac{1}{2}$ and $-\frac{1}{2}$ with respect to 
$(\bar{m}_1,\bar{l}_1)$ and $(\bar{m}_2,\bar{l}_2)$ respectively. On the other hand, regarding
$D_1,D_2$ as the core disk of a $2$-handle attached to $M_0$, Lemma 4.6 implies that the framing of
the $2$-handle equals $\frac{1}{2}-\frac{1}{-2}=1$ and $-\frac{1}{2}-\frac{1}{2}=-1$ respectively
(relative to $(\bar{m}_1,\bar{l}_1)$ and $(\bar{m}_2,\bar{l}_2)$ respectively). It follows easily that 
the self-intersection of $\tilde{S}$ equals $1+(-1)=0$. 

\end{proof}

The proof of Theorem 1.1 is complete. We shall next give a proof of Theorem 1.8.

\begin{lemma}
There is a relative handlebody $Z_0$ built on $M_0$ of two symplectic $2$-handles, which is symplectic with two concave contact boundaries which are both $(M_0,\xi_{inv})$. Moreover,
$c_1(K_{Z_0})\cdot [\omega_{Z_0}]<0$, where $\omega_{Z_0}$ is the symplectic structure on $Z_0$. 

\end{lemma}

\begin{proof}
We shall first exhibit a rational open book $(B,\pi)$ on $M_0$ of periodic monodromy of order $2$, with two binding components $B_1,B_2$ where $B_1$ is a singular fiber with opposite orientation and $B_2$ is a regular fiber but with the same orientation. More concretely, we write
$M_0=M((\alpha_i,\beta_i), (\alpha_j,\beta_j)| i =1,2,j=3,4)$, where $(\alpha_1,\beta_1)=(-2,1)$, 
$(\alpha_2,\beta_2)=(1,-1)$, $(\alpha_j,\beta_j)=(2,1)$ for $j=3,4$. In applying Theorem 3.4, we let
$n=2$, $c_1=c_2=1$, $b_1=0$, $b_2=2$, which obeys $\sum_{i=1}^2 b_i=\sum_{i=1}^2 \frac{c_i}{n}+
\sum_{j=3}^4 \frac{\beta_j}{\alpha_j}$. Furthermore,
$$
\frac{\beta_1}{\alpha_1}+b_1-\frac{c_1}{n}=\frac{1}{-2}+0-\frac{1}{2}=-1<0, \;\;
\frac{\beta_2}{\alpha_2}+b_2-\frac{c_2}{n}=\frac{-1}{1}+2-\frac{1}{2}=\frac{1}{2}>0.
$$
It follows from Theorem 3.4 that the desired rational open book $(B,\pi)$ exists. Finally, by the formula in Remark 3.5, the corresponding multiplicities are $p_1=4$, $p_2=1$. By Theorem 4.1, the $\s^1$-invariant contact structure supported by $(B,\pi)$ can be clearly identified with $\xi_{inv}$, under 
which $B_1=-f_1$. 

Now we attach two symplectic $2$-handles $H_1$, $H_2$ along $B_1,B_2$ to obtain $Z_0$ using Theorem 4.4. To this end, note that $(\bar{\alpha}_1,\bar{\beta}_1)=(1,0)$, 
$(\bar{\alpha}_2,\bar{\beta}_2)=(-2,-1)$ obey $\bar{\alpha}_i\beta_i+\bar{\beta}_i\alpha_i=1$ for 
$i=1,2$. Furthermore, for the corresponding framings, we have 
$F_1=\frac{1}{\alpha_1}(\frac{n}{p_1}-\bar{\alpha}_1)=\frac{1}{-2}(\frac{2}{4}-1)=\frac{1}{4}>0$,
$F_2=\frac{1}{\alpha_2}(\frac{n}{p_2}-\bar{\alpha}_2)=\frac{1}{1}(\frac{2}{1}-(-2))=4>0$ (see Remark 4.5). It follows from Theorem 4.4 that if we attach $H_1$, $H_2$ with framings $F_1=\frac{1}{4}$, $F_2=4$ relative to the page framing, we obtain a relative handlebody $Z_0$ whose other concave contact boundary $(M^\prime,\xi^\prime)$ 
has 
$$
M^\prime=M((\bar{\alpha}_1,\bar{\beta}_1), (\bar{\alpha}_2,\bar{\beta}_2), (\alpha_3,-\beta_3),
(\alpha_4,-\beta_4))=M((1,0),(-2,-1), (2,-1), (2,-1))
$$
which is easily seen to be $M_0$, and $\xi^\prime$ can be identified with $\xi_{inv}$, under which
$f_1$ is identified with the ascending sphere of the $2$-handle $H_2$. We point out that the ascending sphere of the $2$-handle $H_1$ is a regular fiber of $M_0$. Finally, 
$c_1(K_{Z_0})\cdot [\omega_{Z_0}]<0$ follows from Corollary 5.3 by an easy calculation. Hence the lemma. 

\end{proof}

Now suppose we are given two copies of $W_0$, denoted by $W_0$ and $W_0^\prime$, each equipped with a symplectic structure $\omega,\omega^\prime$ with convex contact boundary $(M_0,\xi_{fil})$, 
$(M_0^\prime,\xi_{fil}^\prime)$ respectively (note that $M_0^\prime=M_0$). For simplicity, we assume the $spin^c$-structure associated to $\xi_{fil}$ and $\xi_{fil}^\prime$ is the same $\delta_1\cdot s$ without loss of generality. With this understood, we shall now analyze the smooth isotopy class of the possible contactomorphism
$$
\Psi: -\partial Z_0=(M_0,\xi_{inv})\sqcup (M_0,\xi_{inv})\rightarrow \partial W_0\sqcup \partial W_0^\prime =(M_0,\xi_{fil})\sqcup (M_0^\prime,\xi_{fil}^\prime).
$$
Recall that the singular fibers of $-\partial Z_0$ are labelled by $f_1,f_2,f_3$ on both components,
with the convention that the dividing set of $\xi_{inv}$ is the boundary of a disk neighborhood of the singular point labelling $f_1$. We shall introduce a regular fiber labelled by $f_0$, which is the descending sphere of the
$2$-handle $H_2$ on one component, and the ascending sphere of the $2$-handle $H_1$ on the other component. Finally, we point out that there are embedded cylinders $C_2,C_3$ in $Z_0$ connecting the two copies of $f_2,f_3$ in $-\partial Z_0$. 

On the other hand, the singular fibers of $\partial W_0\sqcup \partial W_0^\prime
=(M_0,\xi_{fil})\sqcup (M_0^\prime,\xi_{fil}^\prime)$ are labelled by $F_0$, $F_1$, $F_2$ 
for $\partial W_0=M_0$, and by the primed-version $F_0^\prime$, $F_1^\prime$, $F_2^\prime$ 
for $\partial W_0^\prime=M_0^\prime$.
With this understood, the smooth isotopy classes of the contactomorphisms $\Psi_1,\Psi_2$ are specified by the following conditions: (1) $\Psi_1,\Psi_2$ will be identical on one of the two components of 
$-\partial Z_0$ and we require $\Psi_1(f_1)=\Psi_2(f_1)=F_1$, $\Psi_1(f_2)=\Psi_2(f_2)=F_0$, and
$\Psi_1(f_3)=\Psi_2(f_3)=F_2$. (2) On the other component of $-\partial Z_0$, $\Psi_1,\Psi_2$ are not identical. We still have $\Psi_1(f_1)=\Psi_2(f_1)=F_1^\prime$, but for the other singular fibers $f_2,f_3$, the images are switched: $\Psi_1(f_2)=\Psi_2(f_3)=F_0^\prime$ 
and $\Psi_1(f_3)=\Psi_2(f_2)=F_2^\prime$.

For $i=1,2$, let $X_i$ be the closed symplectic $4$-manifold obtained by gluing $Z_0$ with two
copies of $W_0$ via the contactomorphism $\Psi_i$. Then $c_1(K_{Z_0})\cdot [\omega_{Z_0}]<0$
implies that both $X_1,X_2$ are symplectic Hirzebruch surfaces (cf. Lemma 5.4). Theorem 1.8 follows immediately from the following lemma.

\begin{lemma}
(1) $X_1$ contains an embedded sphere $S$ of self-intersection $-1$, which intersects transversely at one point with each of the two disjoint copies of $\R\P^2$ contained in $W_0, W_0^\prime$ respectively. In particular, 
$X_1\cong \C\P^2\# \overline{\C\P^2}$.

(2) $X_2$ contains an embedded sphere $S$ of self-intersection $2$,  which intersects transversely at one point with each of the two disjoint copies of $\R\P^2$ contained in $W_0, W_0^\prime$ respectively. In particular, 
$X_2\cong \s^2\times \s^2$.
\end{lemma}

\begin{proof}
(1) Since $\Psi_1(f_2)=F_0$ on the first component of $-\partial Z_0$ and $\Psi_1(f_2)=F_0^\prime$
on the other component, and $F_0,F_0^\prime$ each bounds an embedded disk $D, D^\prime$ in $W_0, W_0^\prime$ respectively, there is an embedded sphere $S:=D\cup_{\Psi_1} C_2\cup_{\Psi_1} D^\prime$ in $X_1$ with the said intersection property, where $C_2$ is the cylinder in $Z_0$ connecting the two copies of $f_2$ on the boundary $-\partial Z_0$. To see that $S$ has self-intersection $-1$, simply note that the slope of the regular fibers of the Seifert fibration on $M_0$ and $M_0^\prime$ is $-\frac{1}{2}$ relative to the zero-framing of $F_0,F_0^\prime$. 

(2) To see the embedded sphere $S$ of self-intersection $2$ in $X_2$, we let $F_3=
\Psi_2(f_0)$ which is a regular fiber in $M_0=\partial W_0$. With this understood, we assign 
Seifert invariants $(\alpha_i,\beta_i)$ to the fibers $F_0,F_1,F_2,F_3$ in $M_0=\partial W_0$ as follows: $(\alpha_0,\beta_0)=(2,-1)$, $(\alpha_1,\beta_1)=(-2,-1)$, $(\alpha_2,\beta_2)=(2,1)$,
and $(\alpha_3,\beta_3)=(1,-1)$. According to \cite{N}, this corresponds to a trivialization of
the Seifert fibration on $M_0=\partial W_0$ in the complement of a regular neighborhood of the
fibers $F_0,F_1,F_2,F_3$. Adapting the notations from \cite{N}, let $R$ be the section associated 
to the trivialization, let $Q_i=-\partial R$, for $i=0,1,2,3$, be the component of $-\partial R$ lying
in the neighborhood of fiber $F_i$, and let $H_i$ be the fiber over $Q_i$. Furthermore, let 
$\bar{m}_i$ be the meridian of the regular neighborhood of $F_i$. Then we have the following 
relations:
$$
\bar{m}_0=2Q_0-H_0, \; \bar{m}_1=-2Q_1-H_1, \; \bar{m}_2=2Q_2+H_2, \; \bar{m}_3=Q_3-H_3.
$$
It follows easily from these relations that, if we set $\bar{l}_0=Q_0$, $\bar{l}_1=Q_1$, 
$\bar{l}_2=-Q_2$, and $\bar{l}_3=Q_3$, then $(\bar{m}_i,\bar{l}_i)$, for $i=0,1,2,3$, is a
meridian-longitude pair for the fiber $F_i$. With this understood, recall that $Q_0$ is a fiber of
the non-orientable circle bundle $pr: M_0\rightarrow \R\P^2$ (cf. Lemma 7.2), so $Q_0$ bounds
an embedded disk $\tilde{D}$ in $W_0$, intersecting the embedded $\R\P^2$ (the zero-section in
$W_0$) transversely at a single point. On the other hand, for $i=1,2,3$, $Q_i$ or $-Q_i$ (as a longitude) and the central fiber $F_i$ bounds an embedded annulus $A_i$ in the regular 
neighborhood of $F_i$. Now if we let $D_1,D_2$ be the core disk of the $2$-handle $H_1,H_2$
in $Z_0$, then $\partial D_1=-f_1$, $\partial D_2=f_0$, so that $\Psi_2(\partial D_1)=F_1$,
$\Psi_2(\partial D_2)=F_3$. 

On the other hand, $\Psi_2(f_3)=F_0^\prime\in M_0^\prime=\partial W_0^\prime$, which bounds an
embedded disk $\tilde{D}^\prime$ in $W_0^\prime$, intersecting the corresponding embedded 
$\R\P^2$ (the zero-section in $W_0^\prime$) transversely at a single point. Since 
$\Psi_2(f_3)=F_2\in M_0=\partial W_0$, it follows easily that the union of $\tilde{D}$, $R$, 
$A_1\cup D_1$, $A_3\cup D_2$, and $-(\tilde{D}^\prime\cup C_3\cup A_2)$, forms an embedded 
sphere $S$ in $X_2$. Here $C_3$ is the cylinder in $Z_0$ which connects the two copies of $f_3$
on the boundary $-\partial Z_0$. 

To determine the self-intersection of $S$, we note that the slope of
the regular fibers of the Seifert fibration on $M_0=\partial W_0$ is $\frac{1}{2}$ with respect to
$(\bar{m}_1,\bar{l}_1)$, $\frac{1}{2}$ with respect to $(\bar{m}_2,\bar{l}_2)$, and $-1$ with respect to
$(\bar{m}_3,\bar{l}_3)$. Using Lemma 4.6, we find that the framing of the $2$-handles $H_1,H_2$
is $\frac{1}{2}-\frac{1}{-2}=1$ and $-1-\frac{-2}{1}=1$ respectively (relative to $(\bar{m}_1,\bar{l}_1)$
and $(\bar{m}_3,\bar{l}_3)$). It follows that the total contribution to the self-intersection of $S$ from
the framings of $H_1,H_2$ equals $1+1=2$. One also needs to account the contribution from the
framing of $-(\tilde{D}^\prime\cup C_3\cup A_2)$, which is $\frac{1}{2}+(-\frac{1}{2})=0$
relative to $(\bar{m}_2,\bar{l}_2)$. Putting all contributions into consideration, we conclude that
the self-intersection of the embedded sphere $S$ in $X_2$ equals $2$. It follows, in particular,
that $X_2\cong \s^2\times \s^2$. 

\end{proof}

\begin{remark}
The $\Z_2$-homology class of the embedded $\R\P^2$s in $X_1,X_2$ is represented by an
embedded Lagrangian Klein bottle nearby, whose $\Z_2$-homology class is constrained,
see Shevchishin \cite{She}. Roughly speaking, the $\Z_2$-homology class of an embedded Lagrangian Klein bottle in a symplectic $\C\P^2\#\overline{\C\P^2}$ must be $H-E \pmod{2}$,
and in a symplectic $\s^2\times \s^2$, it must be the $\Z_2$-homology class of one of the
$\s^2$-factors. It is easy to see that the intersection of the embedded sphere $S$ with the embedded 
$\R\P^2$s in Lemma 7.9 is consistent with the constraints above concerning the $\Z_2$-homology classes of embedded Lagrangian Klein bottles. 

\end{remark}

Next, we give a proof for Theorem 1.9. We begin with the construction of the relative handlebodies $Z_1$, $Z_2$.

\begin{lemma}
There exists a relative handlebody $Z_1$ of two symplectic $2$-handles which have two concave
contact boundaries $(M_1,\xi_{Mil})$ and $(M_0,\xi_{inv})$. Moreover,  one has $c_1(K_{Z_1})\cdot [\omega_{Z_1}]<0$, where $\omega_{Z_1}$ is the symplectic structure on $Z_1$. 
\end{lemma}

\begin{proof}
We first construct a rational open book $(B,\pi)$ of periodic monodromy of order $8$, with two binding components on $M_1=M((13,6), (3,2), (2,-1))$, both oriented as fibers. In particular, the Milnor fillable contact structure 
$\xi_{Mil}$ is supported by $(B,\pi)$. 

To this end, we regard $M_1$ as $M(\{(\alpha_i,\beta_i)|i=1,2,3\})$, where 
$(\alpha_1,\beta_1)=(13,6)$, $(\alpha_2,\beta_2)=(3,-1)$ and $(\alpha_3,\beta_3)=(2,1)$, and apply
Theorem 3.4. Let $n=8$, $c_1=7$, $c_2=5$, and $b_1=b_2=1$. Then the condition $b_1+b_2=\frac{c_1}{n}+\frac{c_2}{n}+\frac{\beta_3}{\alpha_3}$ in Theorem 3.4 is satisfied. Moreover,
$$
\frac{\beta_1}{\alpha_1}+b_1-\frac{c_1}{n}=\frac{6}{13}+1-\frac{7}{8}=\frac{61}{104}>0, \;\;
\frac{\beta_2}{\alpha_2}+b_2-\frac{c_2}{n}=\frac{-1}{3}+1-\frac{5}{8}=\frac{1}{24}>0.
$$
It follows immediately that the desired $(B,\pi)$ exists, with the binding components $B_1,B_2$ being the singular fibers of multiplicity $13$ and $3$ respectively. Finally, note that $p_1=61$ and $p_2=1$
for the multiplicities at $B_1,B_2$ (cf. Remark 3.5). 

To proceed further,  we note that $(\bar{\alpha}_1, \bar{\beta}_1)=(-2,1)$, 
$(\bar{\alpha}_2, \bar{\beta}_2)=(2,1)$ satisfy the relations $\alpha_i\bar{\beta}_i+\bar{\alpha}_i\beta_i=1$ for $i=1,2$. The corresponding framings $F_1=\frac{1}{\alpha_1}(\frac{n}{p_1}-\bar{\alpha}_1)=\frac{1}{13}(\frac{8}{61}-(-2))=\frac{10}{61}>0$ and $F_2=\frac{1}{\alpha_2}(\frac{n}{p_2}-\bar{\alpha}_2)=\frac{1}{3}(\frac{8}{1}-2)=2>0$. By Theorem 4.4, if we attach two symplectic $2$-handles to $M_1$ along $B_1, B_2$, to be denoted by $H_1,H_2$, with framings $F_1,F_2$ relative to the page framing, we obtain a relative handlebody $Z_1$ with two concave contact boundary components, one being $(M_1,\xi_{Mil})$, where for the other component 
$(M^\prime,\xi^\prime)$, one has 
$$
M^\prime=M((\bar{\alpha}_1,\bar{\beta}_1), (\bar{\alpha}_2,\bar{\beta}_2), (\alpha_3,-\beta_3))=
M((-2,1), (2,1), (2,-1))=M_0.
$$
Moreover, since $\bar{\alpha}_1<0$ and $\bar{\alpha}_2>0$, the contact structure $\xi^\prime$ is 
the $\s^1$-invariant, non-transverse tight contact structure $\xi_{inv}$ by Theorem 4.1.

It remains to verify $c_1(K_{Z_1})\cdot [\omega_{Z_1}]<0$, which follows easily from Corollary 5.3.

\end{proof}

\begin{lemma}
There exists a relative handlebody $Z_2$ of two symplectic $2$-handles which have two concave
contact boundaries $(L(25,19),\xi_{Mil})$ and $(M_0,\xi_{inv})$. Moreover,  one has $c_1(K_{Z_2})\cdot [\omega_{Z_2}]<0$, where $\omega_{Z_2}$ is the symplectic structure on $Z_2$. 
\end{lemma}

\begin{proof}
We begin by constructing a rational open book $(B,\pi)$ of periodic monodromy of order $4$ and two binding components on $L(25,19)$ which supports the Milnor fillable contact structure $\xi_{Mil}$. To this end, we fix a 
Seifert fibration on the lens space, $L(25,19)=M((13,6), (2,1))$. In order to apply Theorem 3.4, we let 
$(\alpha_1,\beta_1)=(13,6)$, $(\alpha_2,\beta_2)=(1,0)$, $(\alpha_3,\beta_3)=(2,1)$. With this understood, 
we let $n=4$, $c_1=c_2=3$, $b_1=b_2=1$, which obey the condition $b_1+b_2=\frac{c_1}{n}+\frac{c_2}{n}+\frac{\beta_3}{\alpha_3}$ in Theorem 3.4. Moreover,
$$
\frac{\beta_1}{\alpha_1}+b_1-\frac{c_1}{n}=\frac{6}{13}+1-\frac{3}{4}=\frac{37}{52}>0, \;\;
\frac{\beta_2}{\alpha_2}+b_2-\frac{c_2}{n}=\frac{0}{1}+1-\frac{3}{4}=\frac{1}{4}>0.
$$
It follows that there is a compatible rational open book of periodic monodromy of order $4$ with two binding components $B_1,B_2$, where $B_1$ is the singular fiber of multiplicity $13$ and $B_2$ is
a regular fiber, both with the fiber orientation. The $\s^1$-invariant contact structure supported by
$(B,\pi)$ is transverse (as both $B_1,B_2$ have the fiber orientation), which must be the Milnor fillable contact structure $\xi_{Mil}$. Finally, note that the multiplicities $p_1=37$, $p_2=1$ by the formula in Remark 3.5.

Next, we observe that $(\bar{\alpha}_1, \bar{\beta}_1)=(-2,1)$, 
$(\bar{\alpha}_2, \bar{\beta}_2)=(2,1)$ satisfy the relations $\alpha_i\bar{\beta}_i+\bar{\alpha}_i\beta_i=1$ for $i=1,2$. The corresponding framings $F_1=\frac{1}{\alpha_1}(\frac{n}{p_1}-\bar{\alpha}_1)=\frac{1}{13}(\frac{4}{37}-(-2))=\frac{6}{37}>0$ and $F_2=\frac{1}{\alpha_2}(\frac{n}{p_2}-\bar{\alpha}_2)=\frac{1}{1}(\frac{4}{1}-2)=2>0$. Thus by Theorem 4.4, if we attach two symplectic $2$-handles to $L(25,19)$ along $B_1, B_2$, to be denoted by $H_1,H_2$, with framings $F_1,F_2$ relative to the page framing, we obtain a relative handlebody $Z_2$ with two concave contact boundary components, with one being $(L(25,19),\xi_{Mil})$, where for the other component 
$(M^\prime,\xi^\prime)$, one has 
$$
M^\prime=M((\bar{\alpha}_1,\bar{\beta}_1), (\bar{\alpha}_2,\bar{\beta}_2), (\alpha_3,-\beta_3))=
M((-2,1), (2,1), (2,-1))=M_0.
$$
Moreover, since $\bar{\alpha}_1<0$ and $\bar{\alpha}_2>0$, the contact structure $\xi^\prime$ is 
the $\s^1$-invariant, non-transverse tight contact structure $\xi_{inv}$.
Finally, $c_1(K_{Z_2})\cdot [\omega_{Z_2}]<0$ follows easily from Corollary 5.3. 

\end{proof}

Theorem 1.9 follows immediately from the following lemma. 

\begin{lemma}
(1) For any symplectic $\Q$-homology $4$-ball filling $W$ of $(M_1,\xi_{Mil})$, the $4$-manifolds $W\cup Z_1\cup_{\Phi_1} W_0$ and $W\cup Z_1\cup_{\Phi_2} W_0$ are diffeomorphic to $\s^2\times \s^2$  
and $\C\P^2\# \overline{\C\P^2}$ respectively, where $\Phi_1,\Phi_2$ are the contactomorphisms 
from Theorem 1.1.

(2) The $4$-manifolds $W_{Mil}\cup Z_2\cup_{\Phi_1} W_0$ and $W_{Mil}\cup Z_2\cup_{\Phi_2} W_0$ are diffeomorphic to $\C\P^2\# \overline{\C\P^2}$ and $\s^2\times \s^2$ respectively, where $\Phi_1,\Phi_2$ are the contactomorphisms from Theorem 1.1.
\end{lemma}

\begin{proof}
By Lemma 5.4, the $4$-manifolds are all diffeomorphic to a Hirzebruch surface. So we need to determine for each of them whether it is even or odd, using the same argument as in the proof of Lemma 7.7. The point is that inside 
$W_0\cup_{\Phi_i} Z_j$, for $i=1,2$ and $j=1,2$, there is a natural embedded sphere. It suffices to determine the self-intersection number of the sphere in each case. And to this end, the Seifert invariants $(\alpha_1,\beta_1)$
and $(\alpha_2,\beta_2)$ in the construction of the relative handlebodies $Z_1,Z_2$ are crucial information. By the fact that $(\alpha_1,\beta_1)=(13,6)$ for both $Z_1,Z_2$, but $(\alpha_2,\beta_2)=(3,-1)$ for $Z_1$ and
$(\alpha_2,\beta_2)=(1,0)$ for $Z_2$, it follows easily that $W\cup Z_1\cup_{\Phi_1} W_0$ and $W_{Mil}\cup Z_2\cup_{\Phi_1} W_0$ contains an embedded sphere of self-intersection $-\frac{1}{2}-\frac{3}{2}=-2$ and
$-\frac{1}{2}-\frac{1}{2}=-1$ respectively, which imply that $W\cup Z_1\cup_{\Phi_1} W_0=\s^2\times \s^2$ and
$W_{Mil}\cup Z_2\cup_{\Phi_1} W_0=\C\P^2\# \overline{\C\P^2}$. On the other hand, 
$W\cup Z_1\cup_{\Phi_2} W_0$ and $W_{Mil}\cup Z_2\cup_{\Phi_2} W_0$ contains an embedded sphere of
self-intersection $7-2=5$ and $7-1=6$ respectively, which imply that $W\cup Z_1\cup_{\Phi_2} W_0=\C\P^2\# \overline{\C\P^2}$ and $W_{Mil}\cup Z_2\cup_{\Phi_2} W_0=\s^2\times \s^2$. 

\end{proof}

Finally, we give a proof for Theorem 1.5. The following is the key lemma. 

\begin{lemma}
Let $S_1,S_2$ be a pair of embedded symplectic two-spheres in $(X,\omega)$, where $X=\s^2\times \s^2$ or 
$\C\P^2\#\overline{\C\P^2}$, such that $S_1,S_2$ have self-intersection $4$ and $0$ respectively and intersect at a single point with a tangency of order $2$. Fixing the homology classes of $S_1,S_2$, the union $S_1\cup S_2$ has a unique isotopy class in $(X,\omega)$ through such pairs of symplectic two-spheres.
\end{lemma}

\begin{proof}
For $i=0,1$, let $S_1^{(i)},S_2^{(i)}$ be two such pairs of embedded symplectic two-spheres in $(X,\omega)$. 
We will show that there exists a smooth family of such pairs $S_1^{(t)},S_2^{(t)}$, where $t\in [0,1]$, which connects 
$S_1^{(0)},S_2^{(0)}$ to $S_1^{(1)},S_2^{(1)}$. Here $S_1^{(i)},S_2^{(i)}$, for $i=0,1$, have the same homology classes, which are denoted by $A_1$ and $A_2$ respectively. Note that $A_1\cdot A_1=4$ and $A_2\cdot A_2=0$.  The argument uses standard Gromov's  theory of $J$-holomorphic curves, see e.g. \cite{Gr, HLS, IvaS, Bar, Shev, ST}; particularly for the structure of moduli spaces see \cite{IvaS} and for automatic transversality see \cite{HLS}. 

With the preceding understood, let $J$ be any $\omega$-compatible almost complex structure on $X$. Denote
by $\M(A_i,J)$, for $i=1,2$, the moduli space of $J$-holomorphic two-spheres of homology class $A_i$. By
the automatic transversality theorem in \cite{HLS}, if $c_1(TX)\cdot A_i>0$, then $\M(A_i,J)$ is a smooth manifold of dimension 
$$
\dim_\R \M(A_i,J)=2(c_1(TX)\cdot A_i-1)\geq 0.
$$
With this understood, it follows easily from the adjunction formula that 
$\M(A_i,J)$ is a smooth manifold of dimension $2(A_i\cdot A_i+1)$. 

We will actually need to consider a subspace of $\M(A_1,J)$ which consists of $J$-holomorphic two-spheres 
with certain constraints. More precisely, fixing a set of distinct points $\{x_i\}$ in $X$ and a set of complex lines 
$L_i\subset (T_{x_i} X,J)$, and positive integers $\{a_i\}$, the subspace of $\M(A_1,J)$ consisting of $J$-holomorphic two-spheres which pass through the points $\{x_i\}$ and are tangent to $L_i$ with a tangency of
order least $a_i$ when $a_i>1$ (when $a_i=1$ there is no tangency requirement) is a smooth manifold of dimension $2(A_i\cdot A_i+1-\sum_i a_i)$ if $A_i\cdot A_i+1-\sum_i a_i\geq 0$, with automatic transversality. 

With the preceding understood, we note first that $\M(A_2,J)$ is $2$-dimensional. In fact, $\M(A_2,J)$ is a two-sphere, whose elements give rise to a foliation of $X$ by $J$-holomorphic two-spheres carrying the homology class $A_2$. In particular, we observe the following consequence: for each point $x\in X$, there is a unique $J$-holomorphic two-sphere in the foliation which passes through $x$. We denote by $T_x^J$ the complex line in the complex tangent space $(T_x X,J)$ at $x$, which is the tangent space of the $J$-holomorphic two-sphere passing through $x$. With this understood, for any distinct triple points $x_1,x_2,x_3\in X$, we consider the subset 
$\M(A_1,J; x_1,x_2,x_3)$ of $\M(A_1,J)$ which consists of $J$-holomorphic two-spheres carrying the class $A_1$ such that the points $x_1,x_2,x_3$ are contained in the $J$-holomorphic two-sphere, which is tangent to the complex lines  $T_{x_1}^J$, $T_{x_2}^J$ at $x_1,x_2$ but is transversal to $T_{x_3}^J$ at $x_3$. It is easy to see that $\M(A_1,J; x_1,x_2,x_3)$ is a smooth manifold of dimension $A_1\cdot A_1+1-(2+2+1)=0$, which comes with automatic transversality. We observe that either $\M(A_1,J; x_1,x_2,x_3)$ is empty, or consists of a single element, because any two distinct elements in the space must have intersection number at least $5$, but $A_1\cdot A_1=4$. 

Now for $i=0,1$, let $J_i$ be an $\omega$-compatible almost complex structure such that $S_1^{(i)},S_2^{(i)}$ are
$J_i$-holomorphic. We claim that for each $i$, there exist distinct points $x_1^{(i)},x_2^{(i)},x_3^{(i)}\in X$
such that $S_1^{(i)}$ is an element of $\M(A_1,J_i; x_1^{(i)},x_2^{(i)},x_3^{(i)})$; in particular, it is non-empty. 
To see this, note that under the fibration $\pi_i: X\rightarrow \s^2$, whose fibers are the $J_i$-holomorphic 
two-spheres in $\M(A_2,J_i)$ which contains the two-sphere $S_2^{(i)}$ as one of the fiber, the two-sphere
 $S_1^{(i)}$ is sent to the base $\s^2$ as a double branched cover. The branch locus contains at least two distinct
 points, which means that $S_1^{(i)}$ is tangent to at least two distinct fibers of $\pi_i: X\rightarrow \s^2$. It is easy to see that $S_2^{(i)}$ is one of the two fibers. We simply let $x_1^{(i)},x_2^{(i)}$ be the intersection points of 
 $S_1^{(i)}$ with the two distinct fibers, with $x_1^{(i)}$ being the intersection of  $S_1^{(i)}$ with $S_2^{(i)}$. Finally, we pick a point $x_3^{(i)}\in S_1^{(i)}$ such that $\pi_i(x_3)\in \s^2$ does not lie in the branch locus. Then it is easy to see that $S_1^{(i)}$ is an element of $\M(A_1,J_i; x_1^{(i)},x_2^{(i)},x_3^{(i)})$. 
 
Next, we connect $J_0,J_1$ by a smooth path of $\omega$-compatible almost complex structures $J_t$, and connect the distinct triple points $(x_1^{(0)},x_2^{(0)},x_3^{(0)})$ and $(x_1^{(1)},x_2^{(1)},x_3^{(1)})$ by a smooth path $(x_1^{(t)},x_2^{(t)},x_3^{(t)})$, for $t\in [0,1]$, which satisfy the following conditions. First, we note that by Lemmas 2.3 and 2.4 in  \cite{C0}, there are no $J_i$-holomorphic spheres with self-intersection less than $-1$, for $i=0,1$, due to the presence of the $J_i$-holomorphic two-sphere $S_1^{(i)}$, as distinct pseudoholomorphic curves intersect positively. With this understood, we shall choose a generic path $J_t$, so that for any $t\in [0,1]$, there are no $J_t$-holomorphic spheres with self-intersection less than $-1$ because the virtual dimension of the corresponding moduli space is less than or equal to $-2$ (note that in the case of $X=\C\P^2\#\overline{\C\P^2}$, $J_t$-holomorphic $(-1)$-sphere does exist). Note that we are allowed to require that for each $t$, $J_t$ equals the standard complex structure in a Darboux neighborhood of $x_1^{(t)}$, which is assumed. Secondly, for each $t$, 
we require that $x_1^{(t)}, x_2^{(t)}, x_3^{(t)}$ are contained in three distinct $J_t$-holomorphic two-spheres carrying the class $A_2$, and finally, in the case of $X=\C\P^2\#\overline{\C\P^2}$, $x_3^{(t)}$ does not lie in the $J_t$-holomorphic $(-1)$-sphere, for any $t\in [0,1]$ (note that for $t=0$ and $1$, this is true because the sphere $S_1^{(i)}$, carrying the class $2H$,  is disjoint from the $J_i$-holomorphic $(-1)$-sphere for $i=0,1$.) 

With the preceding understood, we claim that the space $\M(A_1,J_t; x_1^{(t)},x_2^{(t)},x_3^{(t)})$ is non-empty for each $t\in [0,1]$, consisting of a single embedded $J_t$-holomorphic two-sphere $S_1^{(t)}$. To see this, note that since $\M(A_1,J_0; x_1^{(0)},x_2^{(0)},x_3^{(0)})$ is non-empty and has automatic transversality, 
$\M(A_1,J_t; x_1^{(t)},x_2^{(t)},x_3^{(t)})$ is non-empty for any $t$ lying in a sufficiently small open neighborhood of $0\in [0,1]$. On the other hand, we note that the $J_t$-holomorphic two-sphere $S_1^{(t)}$ will converge to an embedded $J_{t_0}$-holomorphic two-sphere $S_1^{(t_0)}$ if $t\mapsto t_0$. This is because, by Gromov compactness, if otherwise, then $S_1^{(t)}$ will converge to a $J_{t_0}$-holomorphic cusp-curve, which is easily seen must contain two distinct $J_{t_0}$-holomorphic two-spheres carrying the homology class $A_2$, each containing $x_1^{(t_0)}, x_2^{(t_0)}$ respectively. It is easy to see that this is not possible when $X=\s^2\times \s^2$, because in this case, $A_1=e_1+2e_2$ and $A_2=e_1$ where $e_1,e_2$ are the classes of the two $\s^2$-factors, and by the assumption on $J_t$, there are no $J_{t_0}$-holomorphic spheres with negative self-intersection. If $X=\C\P^2\#\overline{\C\P^2}$, then $A_1=2H$ and $A_2=H-E$ where $H,E$ is the standard basis of $H^2(X)$. It follows easily that the $J_{t_0}$-holomorphic cusp-curve must contain the $J_{t_0}$-holomorphic 
$(-1)$-sphere $C_{t_0}$. With this understood, note that the $J_{t_0}$-holomorphic cusp-curve contains all three points $x_1^{(t_0)},x_2^{(t_0)},x_3^{(t_0)}$. Since $x_1^{(t_0)}, x_2^{(t_0)}, x_3^{(t_0)}$ are contained in three distinct  $J_{t_0}$-holomorphic two-spheres carrying the class $A_2$ by our assumption, it follows easily that $x_3^{(t_0)}$ must lie in the $J_{t_0}$-holomorphic $(-1)$-sphere $C_{t_0}$, which contradicts our assumption on the path $x_3^{(t)}$. This proves the compactness of $\M(A_1,J_t; x_1^{(t)},x_2^{(t)},x_3^{(t)})$ under
$t\mapsto t_0$.
Since each $\M(A_1,J_t; x_1^{(t)},x_2^{(t)},x_3^{(t)})$ has automatic transversality when it is non-empty, it follows from the standard continuity argument that $\M(A_1,J_t; x_1^{(t)},x_2^{(t)},x_3^{(t)})$ is non-empty for each 
$t\in [0,1]$. Finally, let $S_2^{(t)}$ be the embedded $J_t$-holomorphic two-sphere carrying the class $A_2$ and passing through $x_1^{(t)}$. Then $S_1^{(t)},S_2^{(t)}$ form a smooth family of pairs of embedded symplectic 
two-spheres connecting the two given pairs $S_1^{(0)},S_2^{(0)}$ and $S_1^{(1)},S_2^{(1)}$. This finishes off the lemma. 

\end{proof}

{\bf Proof of Theorem 1.5:}

\vspace{2mm}

Let $W$ be a $\Q$-homology $4$-ball which is a symplectic filling of $(M_0,\xi_{tight})$. Then by the uniqueness of tight contact structures on $M_0$ (cf. \cite{GLS}), there is a contactomorphism between $(M_0,\xi_{inv})$ and
 $(M_0,\xi_{tight})$. We can form a closed symplectic $4$-manifold $W\cup V_0$ where $V_0$ is from Theorem 1.1. Note that $W\cup V_0$ is diffeomorphic to either $\s^2\times \s^2$ or $\C\P^2\#\overline{\C\P^2}$.
 Consequently, for one of the two contactomorphisms in Theorem 1.1, i.e., $\Phi_1$ or $\Phi_2$, $W\cup V_0$ is diffeomorphic to either $W_0\cup_{\Phi_1} V_0$ or $W_0\cup_{\Phi_2} V_0$, which for simplicity we write as
 $W_0\cup V_0$. Even more, there is a diffeomorphism between  $W\cup V_0$ 
 and  $W_0\cup V_0$ which preserves the symplectic structures on the two manifolds because the symplectic structures are completely determined by the areas of the pair of symplectic two-spheres $S_1,S_2$ in $V_0$. 
 
It follows easily that there exist two pairs of embedded symplectic two-spheres in $(X,\omega)$, $S_1^{(0)},S_2^{(0)}$ and $S_1^{(1)},S_2^{(1)}$, where $X=\s^2\times \s^2$ or $\C\P^2\#\overline{\C\P^2}$,
 such that each pair of spheres have self-intersection $4$ and $0$ respectively and intersect at a single point 
with a tangency of order $2$, and the two pairs have the same homology classes in $X$. Furthermore,
the complement of a regular neighborhood of one of the two pairs of spheres is diffeomorphic to $W$ while the complement of a regular neighborhood of other pair is diffeomorphic to $W_0$. By Lemma 7.14, there exists a smooth family of pairs of two-spheres $S_1^{(t)},S_2^{(t)}$, where $t\in [0,1]$, which connects 
$S_1^{(0)},S_2^{(0)}$ to $S_1^{(1)},S_2^{(1)}$. The smooth isotopy can be extended to an ambient isotopy of
$X$, which gives rise to a diffeomorphism between $W$ and $W_0$, This finishes the proof.

\section{$\Q$-homology ball symplectic fillings of $M((13,6), (3,2), (2,-1))$}
In this section, we shall give a proof of Theorem 1.10.  Recall that $M_1$ denote the small Seifert space 
$M((13,6), (3,2), (2,-1))$, which is equipped with the Milnor fillable contact structure $\xi_{Mil}$. It is known that $(M_1,\xi_{Mil})$ admits a symplectic filling by a $\Q$-homology ball, provided by the Milnor fiber of the rational homology disk smoothing (cf. \cite{BS}, Figure 1(f) with $q=1$). We are interested in an arbitrary $\Q$-homology ball symplectic filling of $(M_1,\xi_{Mil})$, not necessarily coming from a smoothing of the singularity. Theorem 1.10
reduces the classification of such a filling to an isotopy problem of certain pairs of symplectic rational curves in
$\C\P^2$.

We begin by considering the Hamiltonian $\s^1$-action on $(\C^2,\omega_0)$, 
$\lambda\cdot (z_1,z_2)=(\lambda^2 z_1,\lambda^{13} z_2)$, $\forall \lambda\in\s^1$, which has a Hamiltonian function $H=\frac{1}{2}(2|z_1|^2+13|z_2|^2)$. 
It defines a Seifert fibration on $H^{-1}(1)$, which is diffeomorphic to $\s^3$, with two 
singular fibers of multiplicity $13$ and $2$, given by $z_1=0$ by $z_2=0$ respectively, 

\begin{lemma}
There is a relative handlebody $Z$ of two symplectic $2$-handles built on $H^{-1}(1)$, which is symplectic with two concave contact boundaries $(H^{-1}(1),\xi_{std})$ and $(M_1, \xi_{Mil})$.
Moreover, $c_1(K_{Z})\cdot [\omega_{Z}]<0$. Here $\omega_Z$ denotes the symplectic structure on $Z$ and
$\xi_{std}$ is the standard tight contact structure on $H^{-1}(1)\cong \s^3$. 
\end{lemma}

\begin{proof}
We begin by describing the rational open book on $H^{-1}(1)$ used in the construction of $Z$.
It is of periodic monodromy of order $52$, with two binding components $B_1,B_2$, where
$B_1$ is a regular fiber defined by $z_2^2-z_1^{13}=0$ and $B_2$ is the singular fiber of multiplicity $2$, defined by $z_2=0$, both given with the fiber orientation. We note that as a consequence, the $\s^1$-invariant contact structure supported by the rational open book is transverse, hence it must be $\xi_{std}$. 

The rational open book is constructed using Theorem 3.4, for which we 
regard $H^{-1}(1)$ as the Seifert manifold $M((\alpha_1,\beta_1), (\alpha_2,\beta_2),
(\alpha_3,\beta_3))$ where $(\alpha_1,\beta_1)=(1,0)$, $(\alpha_2,\beta_2)=(2,-1)$, and 
$(\alpha_3,\beta_3))=(13,7)$. In applying Theorem 3.4, we let $n=52$, $c_1=51$, $c_2=25$,
and $b_1=b_2=1$. It is easy to check that $0<c_i<n$, $gcd(n,c_i)=1$ for $i=1,2$, and
$\sum_{i=1}^2 b_i=\sum_{i=1}^2 \frac{c_i}{n}+\frac{\beta_3}{\alpha_3}$. Moreover,
$\frac{\beta_i}{\alpha_i}+b_i-\frac{c_i}{n}=\frac{1}{52}>0, \mbox{ for } i=1,2$. Thus the assumptions 
in Theorem 3.4 are verified. Finally, using the formula in Remark 3.5, we find the multiplicities at 
$B_1,B_2$ to be $p_1=1$ and $p_2=2$.

To proceed further, let $(\bar{\alpha}_1,\bar{\beta}_1)=(2,1)$,  $(\bar{\alpha}_2,\bar{\beta}_2)=(3,2)$, 
then $\alpha_i\bar{\beta}_i+\bar{\alpha}_i\beta_i=1$ for $i=1,2$. Furthermore, for 
$F_i:=\frac{1}{\alpha_i}(\frac{n}{p_i}-\bar{\alpha}_i)$, we have $F_1=50>0$ and $F_2=\frac{23}{2}>0$.
Consequently, by Theorem 4.4 if we attach two symplectic $2$-handles $H_1,H_2$ to $H^{-1}(1)$ along $B_1,B_2$ with framings $F_1,F_2$ relative to the page framing, we obtain a relative handlebody $Z$ with
two concave contact boundaries, one of which is $(H^{-1}(1),\xi_{std})$, while for the other one
$(M^\prime,\xi^\prime)$, we have 
$$
M^\prime=M((2,1),(3,2), (13,-7))=M((13,6), (3,2),(2,-1))=M_1,
$$
and as $\bar{\alpha}_i>0$ for $i=1,2$, the $\s^1$-invariant contact structure $\xi^\prime$ is transverse
(see Theorem 4.1). Consequently,  $\xi^\prime=\xi_{Mil}$. 

Finally, to see $c_1(K_{Z})\cdot [\omega_{Z}]<0$, we use Lemma 5.2 to compute 
$\sum_{i=1}^2 c_1(K_Z)[p_i D_i]$, where $D_1,D_2$ denote the core disk of $H_1,H_2$ respectively. To this end, we note that the Euler characteristic of the pages 
$\chi(\Sigma)=-n(1-\frac{1}{\alpha_3})=-48$, $p_1(1+F_1)=51$,
and $p_2(1+F_2)=25$. Consequently, $\sum_{i=1}^2 c_1(K_Z)[p_i D_i]=48-51-25=-28<0$,
which implies that $c_1(K_{Z})\cdot [\omega_{Z}]<0$ by Lemma 5.1. 

\end{proof}

Next, we let $V$ be the union of $Z$ with $H^{-1}([0,1])$ along $H^{-1}(1)$, which is easily seen
a symplectic cap of $(M_1,\xi_{Mil})$. We note that $B_1$ bounds a singular holomorphic disk 
$z_2^2-z_1^{13}=0$ in $H^{-1}([0,1])$ and $B_2$ bounds an embedded holomorphic disk $z_2=0$
in $H^{-1}([0,1])$. With this understood, as we argued  in Section 7 for the construction of $V_0$,
one can choose an appropriate $\s^1$-invariant contact form on $H^{-1}(1)$, therefore a specific
symplectic structure on $Z$, such that we can glue $Z$ to $H^{-1}([0,1])$ along $H^{-1}(1)$ 
symplectically, so that the union of the core disks $D_1,D_2$ of the $2$-handles $H_1,H_2$ with 
some symplectic disks $D_1^\prime,D_2^\prime$ in $H^{-1}([0,1])$ form a pair of symplectic 
rational curves $\tilde{C}_1,\tilde{C}_2$ in $V$, where $\tilde{C}_1$ has a unique singularity at
$0\in H^{-1}([0,1])$ modeled by $z_2^2-z_1^{13}=0$, and $\tilde{C}_2$ is embedded 
which is given by $z_2=0$ near $0\in H^{-1}([0,1])$. Finally, $\tilde{C}_1,\tilde{C}_2$ intersect at 
a single point, with a local intersection number $13$. 

We proceed further by computing the self-intersection number of $\tilde{C}_1$ and $\tilde{C}_2$, which
requires us to understand the framings of the $2$-handles $H_1,H_2$ relative to the zero-framings of
the binding components $B_1,B_2$. To this end, we observe the following interesting application of Theorem 3.3.

\begin{lemma}
Let $M$ be a Seifert manifold as given in Theorem 3.3, where the condition (*) is satisfied. Furthermore, we assume $e(M)=-\frac{1}{n\alpha}$. Then the fiber $B$ of multiplicity $\alpha$ 
is null-homologous in $M$, and with respect to the zero-framing of $B$, the regular fibers of the
Seifert fibration on $M$ have slope $\frac{n}{\alpha}$. 
\end{lemma}

\begin{proof}
For simplicity, we continue to use the notations in Theorem 3.3. Since (*) is satisfied, there is a
rational open book with periodic monodromy of order $n$ on $M$, with a single binding component 
$B$ which is the fiber of multiplicity $\alpha$. Furthermore, with $e(M)=-\frac{1}{n\alpha}<0$,
it follows that the multiplicity $p=1$ at $B$, and $B$ has the fiber orientation. In particular, the page
$\Sigma$ of the rational open book is a Seifert surface of $B$, and the simple closed curve
$\gamma=\partial \Sigma$ is a longitude of $B$ defining the zero-framing. 

With the preceding understood, recall from the proof of Lemma 3.2 that the regular fibers of the
Seifert fibration is given by $H=k\gamma+n\tau$. On the other hand, we have identities 
$\alpha=k+np^\prime$ and $\mu=-p^\prime \gamma+\tau$, as $p=1$, where $\mu$ is the meridian
of $B$. It follows easily that $H=\alpha\gamma+n \mu$, which implies immediately that the slope
of $H$ relative to the zero-framing of $B$ equals $\frac{n}{\alpha}$. This finishes the proof. 

\end{proof}

Now back to computing the self-intersection number of $\tilde{C}_1,\tilde{C}_2$, we note that the
assumptions in Lemma 8.2 are satisfied by $H^{-1}(1)$ in regard to both $B_1,B_2$, so we can use 
it to determine the slope of the regular fibers of the Seifert fibration. To this end, note that for $B_1$,
we have $n=26$ and $\alpha=1$, so the slope is $26$. On the other hand, for $B_2$ we have
$n=13$ and $\alpha=2$, so the slope is $\frac{13}{2}$. Now using Lemma 4.6, we find the framing of
$H_1$ relative to the zero-framing of $B_1$ to be 
$26-\frac{\bar{\alpha}_1}{\alpha_1}=26-\frac{2}{1}=24$, and the framing of $H_2$ to be 
$\frac{13}{2}-\frac{\bar{\alpha}_2}{\alpha_2}=\frac{13}{2}-\frac{3}{2}=5$. Consequently, 
the self-intersection number of $\tilde{C}_1,\tilde{C}_2$ equals $24$ and $5$ respectively. 
Finally, we observe that the adjunction equation implies that
$$
c_1(K_V)\cdot \tilde{C}_1=-14, \;\; c_1(K_V)\cdot \tilde{C}_2=-7.
$$
It is interesting to note that 
$$
c_1(K_Z)\cdot [D_1]=c_1(K_V)\cdot \tilde{C}_1=-14, \;\; 
c_1(K_Z)\cdot [D_2]=c_1(K_V)\cdot \tilde{C}_2=-7,
$$
which are consistent with the identity $c_1(K_Z)\cdot [D_1]+c_1(K_Z)\cdot [2D_1]=-28$ we derived earlier in the proof of Lemma 8.1.

To proceed further, let $W$ be any  $\Q$-homology ball symplectic filling of $(M_1,\xi_{Mil})$. The union 
$X:=W\cup V$ is a symplectic Hirzebruch surface by Lemma 5.4, and is in fact diffeomorphic to
$\C\P^2\#\overline{\C\P^2}$ because it contains a sphere $\tilde{C}_2$ of odd self-intersection. 

We shall symplectically blow down $X$ to $\C\P^2$. To this end, let $H,E$ be a standard basis 
of $H^2(X)$ such that $c_1(K_X)=-3H+E$. It is easy to check that
$$
\tilde{C}_1=5H-E,\;\; \tilde{C}_2=3H-2E.
$$
With this understood, we fix a compatible almost complex structure $J$ such that both 
$\tilde{C}_1,\tilde{C}_2$ are $J$-holomorphic. Now we observe that there is an embedded 
$J$-holomorphic sphere $C$ whose class equals $-mH+(m+1)E$ for some $m\geq 0$ (cf. \cite{C0}, Lemma 2.3). With the presence of $\tilde{C}_1$, $C$ must be a $(-1)$-sphere, because 
$0\leq C\cdot \tilde{C}_1=-4m+1$, implying $m=0$. Moreover, $C$ intersects $\tilde{C}_1$ transversely at a single point away from its singularity. The intersection of $C$ with $\tilde{C}_2$ is either transversal at two points or at one point but with a tangency of order $2$. In the latter case, 
we shall slightly perturb $C$ into a symplectic $(-1)$-sphere, near $C\cap \tilde{C}_2$, such that the symplectic $(-1)$-sphere intersects $\tilde{C}_2$ transversely, positively, at two points. After this perturbation if necessary, we can symplectically blow down $X$ to $\C\P^2$ along the symplectic 
$(-1)$-sphere such that $\tilde{C}_1,\tilde{C}_2$ have a well-defined descendant in $\C\P^2$, to be denoted by $C_1,C_2$ (see \cite{C2}, Section 4, for a general discussion). Moreover, $C_1$ and $C_2$ can be made pseudo-holomorphic for some compatible almost complex structure which is integrable near the intersection of $C_1$ and $C_2$. It is clear that $C_1$ is a degree $5$ rational curve with a unique singularity whose link is a $(2,13)$-torus knot, and $C_2$ is a degree $3$ 
rational nodal curve (both in the pseudo-holomorphic sense). Finally, $C_1\cap C_2=\{x_1,x_2\}$, where $x_1$ is the singularity of $C_1$ and $x_2$ is the nodal point of $C_2$, with the local intersection number of $C_1,C_2$ at $x_1, x_2$ equaling $13, 2$ respectively. This finishes the proof of Theorem 1.10 in one direction.

Conversely, suppose $C_1,C_2$ is a pair of $J$-holomorphic rational curves in $\C\P^2$, where
$C_1$ is of degree 5, with a unique cuspidal singularily $x_1$ whose link is a (2,13)-torus knot, and $C_2$ is a degree $3$ nodal curve with nodal point $x_2$, such that $C_1\cap C_2=\{x_1,x_2\}$
with the local intersection number of $C_1,C_2$ at $x_1, x_2$ equaling $13, 2$ respectively.
Here $J$ is $\omega$-compatible with respect to some symplectic structure $\omega$ on
$\C\P^2$, and we assume $J$ is integrable in a neighborhood of $x_1,x_2$. 

To see the $\Q$-homology ball symplectic filling of $(M_1,\xi_{Mil})$ associated to the 
configuration $C_1\cup C_2\subset \C\P^2$, we shall perform a holomorphic blowing up 
of $\C\P^2$ at the point $x_2$ as $J$ is integrable near $x_2$, and let $\tilde{C}_1,\tilde{C}_2$ be the proper transforms of $C_1,C_2$ respectively. Then $\tilde{C}_1,\tilde{C}_2$ are $\tilde{J}$-holomorphic rational curves 
in a symplectic $\C\P^2\# \overline{\C\P^2}$ (see \cite{C5}, Lemma 3.1, for the technical details). Let $W$ be the complement of a regular neighborhood of $\tilde{C}_1\cup\tilde{C}_2$ in $\C\P^2\# \overline{\C\P^2}$. We will show that $W$ is a symplectic filling of $(M_1,\xi_{Mil})$ (which is clearly a $\Q$-homology ball), and this will finish the proof of Theorem 1.10 in the other direction.

A naive approach would be to construct a regular neighborhood of $\tilde{C}_1\cup\tilde{C}_2$ with a concave contact boundary $(M_1,\xi_{Mil})$ using the symplectic model $V=H^{-1}([0,1])\cup Z$, where $Z$ is a relative handlebody constructed in the fashion of Lemma 8.1. However, in view of Remark 6.5, there might be geometric constraints coming from the areas of the curves $\tilde{C}_1,\tilde{C}_2$. In the following lemma, we spell out
what this constraint is. 

\begin{lemma}
Suppose there is a rational open book with periodic monodromy with binding components $B_1,B_2$, which is constructed in the fashion of Lemma 8.1, such that the corresponding relative handlebody $Z$ is symplectic with
concave contact boundaries $(H^{-1}(1),\xi_{std})$ and $(M_1, \xi_{Mil})$, and $H^{-1}([0,1])\cup Z$ gives a symplectic model for a regular neighborhood of $\tilde{C}_1\cup\tilde{C}_2$. Then the areas of 
$\tilde{C}_1, \tilde{C}_2$ must obey the following constraint 
$$
\frac{24}{13}<\frac{Area(\tilde{C}_1)}{Area(\tilde{C}_2)}<\frac{13}{5}.
$$
\end{lemma}

\begin{proof}
Recall that $H^{-1}(1)$ is regarded as a Seifert manifold 
$M((\alpha_1,\beta_1), (\alpha_2,\beta_2), (\alpha_3,\beta_3))$, where $(\alpha_1,\beta_1)=(1,0)$, 
$(\alpha_2,\beta_2)=(2,-1)$, and $(\alpha_3,\beta_3))=(13,7)$, in the construction in Lemma 8.1. 
Suppose there is a rational open book constructed by applying Theorem 3.4, where we choose 
a positive multiple $n$ of $\alpha_3=13$, and $0<c_i<n$ such that
$gcd(n,c_i)=1$ for $i=1,2$, and integers $b_1,b_2$ such that $\sum_{i=1}^2 b_i=\sum_{i=1}^2 \frac{c_i}{n}+\frac{\beta_3}{\alpha_3}$, with $\frac{\beta_i}{\alpha_i}+b_i-\frac{c_i}{n}>0$ for $i=1,2$. With this understood, we note that $\frac{\beta_i}{\alpha_i}+b_i-\frac{c_i}{n}=\frac{p_i}{\alpha_i n}$ for $i=1,2$, where $p_1,p_2$ are
the multiplicities of the rational open book at the binding components $B_1,B_2$.

Assume $F_1,F_2>0$ are the framings relative to the page framing of the symplectic $2$-handles attached along
$B_1,B_2$. Note that $(\bar{\alpha}_1,\bar{\beta}_1)=(2,1)$,  $(\bar{\alpha}_2,\bar{\beta}_2)=(3,2)$, and the framings are given by the formula $F_i=\frac{1}{\alpha_i}(\frac{n}{p_i}-\bar{\alpha}_i)$ for $i=1,2$. With this understood, we set $r:=\frac{Area(\tilde{C}_1)}{Area(\tilde{C}_2)}$. Then in view of Remark 6.5, one has $r=\frac{p_1F_1}{p_2F_2}$. 

To proceed further, we note that $p_iF_i=n(\frac{1}{\alpha_i}-\bar{\alpha}_i\cdot \frac{p_i}{\alpha_i n})$ for $i=1,2$.
Setting $x=\frac{p_1}{\alpha_1 n}$ and $y=\frac{p_2}{\alpha_2 n}$, and noting that $\alpha_1=1,\alpha_2=2$ and
$\bar{\alpha}_1=2$ and $\bar{\alpha}_2=3$, one obtains $1-2x=r(\frac{1}{2}-3y)$. On the other hand, 
$$
x+y=\sum_{i=1}^2 (\frac{\beta_i}{\alpha_i}+b_i-\frac{c_i}{n})=\sum_{i=1}^3 \frac{\beta_i}{\alpha_i}=\frac{1}{26}.
$$
Solving for $y$, one obtains $y=\frac{13r-24}{26(2+3r)}$. With $0<y<\frac{1}{26}$, one arrives at $\frac{24}{13}<r<\frac{13}{5}$. 

\end{proof}

With the preceding understood, we note that the proper transforms $\tilde{C}_1,\tilde{C}_2$ of $C_1,C_2$ in 
$\C\P^2\#\overline{\C\P^2}$ have the homology classes 
$$
\tilde{C}_1=5H-E, \tilde{C}_2=3H-2E, 
$$																						
where the area of the exceptional divisor $E$ is sufficiently small, see Lemma 3.1 in \cite{C5}. It follows easily that the ratio of the areas, i.e.,  $\frac{Area(\tilde{C}_1)}{Area(\tilde{C}_2)}$, which lies in a small neighborhood of $\frac{5}{3}$, may not necessarily be in the range $\frac{24}{13}<\frac{Area(\tilde{C}_1)}{Area(\tilde{C}_2)}<\frac{13}{5}$.

To get around of this issue, we observe that the complement $W$ of a regular neighborhood of 
$\tilde{C}_1\cup \tilde{C}_2$ in $\C\P^2\#\overline{\C\P^2}$ is a symplectic filling of $(M_1,\xi_{LM})$, where
$\xi_{LM}$ is the Li-Mak contact structure of $M_1$ associated to the minimal normal crossing resolution of 
$\tilde{C}_1\cup \tilde{C}_2$. With this understood, the remaining part of Theorem 1.10
follows immediately from the following lemma. 

\begin{lemma}
There is a topologically trivial symplectic cobordism $W^\prime$ from $(M_1,\xi_{LM})$ to $(M_1,\xi_{Mil})$.
\end{lemma}

For a proof, we consider the symplectic manifold $V$, which is the union of the $4$-ball $H^{-1}([0,1])$ with
the relative handlebody $Z$ constructed in Lemma 8.1. We note that inside $V$, there is a pair of symplectic rational curves $\tilde{C}_1,\tilde{C}_2$ intersecting at the origin $0\in H^{-1}([0,1])$. Moreover, $\tilde{C}_1$
has a unique cuspidal singularity at $0\in H^{-1}([0,1])$, while $\tilde{C}_2$ is embedded, and near their intersection 
point $\tilde{C}_1,\tilde{C}_2$ are given by $z_2^2-z_1^{13}=0$ and $z_2=0$ in some local holomorphic coordinates, and the self-intersections of $\tilde{C}_1,\tilde{C}_2$ are $24$ and $5$ respectively. We should also note, in view of Remark 6.5, that we may assume without loss of generality that 
$Area(\tilde{C}_1)=50$ and $Area(\tilde{C}_2)=23$, as from the construction in Lemma 8.1, we have
$p_1=1$, $F_1=50$, and $p_2=2$, $F_2=\frac{23}{2}$. 

To proceed further, we shall first construct a regular neighborhood $U$ of $\tilde{C}_1\cup \tilde{C}_2$ in the interior of $V$, which has a concave contact boundary $(M_1,\xi_{LM})$. To this end, consider the minimal normal crossing 
resolution of $\tilde{C}_1\cup \tilde{C}_2$, which is a normal crossing divisor $D$ in a successive blowing up
$\tilde{V}$ of $V$. The plumbing graph $\Gamma$ for the divisor $D$ is shown in Figure 1. We label the vertices of
$\Gamma$ by $v_1,v_2,\cdots,v_{10}$ as shown there, where the components of $D$ represented by $v_1,v_2$
are the proper transforms of $\tilde{C}_1,\tilde{C}_2$, marked by $\hat{C}_1, \hat{C}_2$. As the labeling order of the components of $D$ is being fixed as in Figure 1, the intersection matrix of $\Gamma$ is a well-defined 
$10\times 10$ matrix, to be denoted by $Q_\Gamma$, which is clearly non-singular and not negative definite. 
																		
\begin{picture}(500,120)
\put(230,100){\circle*{5}}
\put(230,70){\circle*{5}}
\put(260,70){\circle*{5}}
\put(290,70){\circle*{5}}
\put(200,70){\circle*{5}}
\put(170,70){\circle*{5}}
\put(140,70){\circle*{5}}
\put(110,70){\circle*{5}}
\put(80,70){\circle*{5}}
\put(50,70){\circle*{5}}
\put(50,70){\line(1,0){240}}
\put(230,70){\line(0,1){30}}
\put(235,100){\makebox(0,0)[l]{$-2$}}
\put(50,65){\makebox(0,0)[t]{$-2$}}
\put(80,65){\makebox(0,0)[t]{$-2$}}
\put(110,65){\makebox(0,0)[t]{$-2$}}
\put(140,65){\makebox(0,0)[t]{$-2$}}
\put(170,65){\makebox(0,0)[t]{$-2$}}
\put(200,65){\makebox(0,0)[t]{$-3$}}
\put(230,65){\makebox(0,0)[t]{$-1$}}
\put(260,65){\makebox(0,0)[t]{$-2$}}
\put(290,65){\makebox(0,0)[t]{$-2$}}
\put(230,117){\makebox(0,0)[t]{$\hat{C}_1$}}
\put(290,87){\makebox(0,0)[t]{$\hat{C}_2$}}
\put(225,100){\makebox(0,0)[r]{$v_1$}}
\put(50,54){\makebox(0,0)[t]{$v_{10}$}}
\put(80,54){\makebox(0,0)[t]{$v_9$}}
\put(110,54){\makebox(0,0)[t]{$v_8$}}
\put(140,54){\makebox(0,0)[t]{$v_7$}}
\put(170,54){\makebox(0,0)[t]{$v_6$}}
\put(200,54){\makebox(0,0)[t]{$v_5$}}
\put(230,54){\makebox(0,0)[t]{$v_4$}}
\put(260,54){\makebox(0,0)[t]{$v_3$}}
\put(290,54){\makebox(0,0)[t]{$v_2$}}
\put(180,25){\makebox(0,0)[t]{Figure 1: The graph $\Gamma$ of the minimal normal crossing resolution of 
$\tilde{C}_1\cup \tilde{C}_2$}}
\end{picture}

With this understood, we observe the following lemma. 

\begin{lemma}
Fix $\lambda_1,\lambda_2>0$, and let $\vec{\lambda}=(\lambda_1,\lambda_2, 0,0,\cdots,0)^T\in \R^{10}$. Then
for any vector $\vec{a}\in \R^{10}$ which is sufficiently close to $\vec{\lambda}$, there is a vector 
$\vec{z}\in\R^{10}$ with positive entries such that $Q_\Gamma \vec{z}=\vec{a}$ if and only of $\frac{24}{13}<\frac{\lambda_1}{\lambda_2}<\frac{13}{5}$.
\end{lemma}

\begin{proof}
Since $Q_\Gamma$ is non-singular, it suffices to show that there is a vector $\vec{x}\in\R^{10}$ with positive entries such that $Q_\Gamma \vec{x}=\vec{\lambda}$ if and only of $\frac{24}{13}<\frac{\lambda_1}{\lambda_2}<\frac{13}{5}$.

Let $x_1,x_2,\cdots,x_{10}$ be the entries of $\vec{x}$. Then from the last six equations in 
$Q_\Gamma \vec{x}=\vec{\lambda}$, one obtains easily that
$(x_4,x_5,x_6,x_7,x_8,x_9,x_{10})=(13\delta, 6\delta,  5\delta,  4\delta,  3\delta,  2\delta,  \delta)$ for some
$\delta\in\R$. On the other hand, the first four equations in $Q_\Gamma \vec{x}=\vec{\lambda}$ give us to the following system of equations
$$
\left \{\begin{array}{lll}
-2x_1+ 13\delta &=& \lambda_1\\
-2x_2+x_3 & = & \lambda_2\\
x_2-2x_3+ 13\delta & = & 0\\
x_1+x_3-13\delta+6\delta & = & 0\\
\end{array}
\right. 
$$
where we substituted $x_4,x_5$ by $13\delta$, $6\delta$ respectively. It follows easily from these equations that
$$
\delta=\frac{1}{49}(3\lambda_1+2\lambda_2)>0, \mbox{ and } x_1=\frac{1}{2} (13\delta-\lambda_1),\;  x_2=\lambda_1-12\delta, \; x_3=\frac{1}{2}(\delta+\lambda_1)>0.
$$
With this understood, it is easy to check that $x_1>0$ if and only if $13\lambda_2>5\lambda_1$, and  $x_2>0$ 
 if and only if $24\lambda_2<13\lambda_1$. The lemma follows immediately.
 
\end{proof}
 
 Note that the range for the ratio $\frac{\lambda_1}{\lambda_2}$ in Lemma 8.5 is identical to the range for 
 $\frac{Area(\tilde{C}_1)}{Area(\tilde{C}_2)}$ in Lemma 8.3. It is possible that this is not a coincidence, but we do not have a good explanation of it at this moment. 

\vspace{2mm}

With Lemma 8.5 at hand, we shall perform a successive symplectic blowing up to $V$, where the underlying operations are holomorphic blowing ups, see Lemma 3.1 of \cite{C5}. This results a symplectic normal crossing divisor $D$ in a symplectic manifold $\tilde{V}$, with plumbing graph $\Gamma$ in Figure 1. 
Furthermore, according to Lemma 3.1 in \cite{C5}, one can arrange so that the components of $D$ which are 
either an exceptional divisor or a proper transform of an exceptional divisor have arbitrarily small areas, while the components $\hat{C}_1, \hat{C}_2$ have areas which are arbitrarily close to the areas of $Area(\tilde{C}_1)=50$ and $Area(\tilde{C}_2)=23$ respectively. With this understood, let $\vec{a}$ denote the vector whose entries are the areas of the components of the divisor $D$, and let $\lambda_1=50$ and $\lambda_2=23$ in Lemma 8.5, then there is a vector $\vec{z}$ of positive entries such that $Q_\Gamma \vec{z}=\vec{a}$, because 
$\frac{24}{13}<\frac{\lambda_1}{\lambda_2}<\frac{13}{5}$ is satisfied. Consequently, by Li-Mak \cite{LM}, the symplectic divisor $D$ has a regular neighborhood $\tilde{U}$ inside $\tilde{V}$, which has a concave contact boundary $(M_1,\xi_{LM})$. We then use the successive
symplectic blowing down procedure described in Section 4 of \cite{C2} to blow down $\tilde{U}$ to $U$, reducing 
$D$ to a pair of symplectic rational curves $C_1^\prime,C_2^\prime$ inside $U$. The boundary of $U$ is identical to
the boundary of $\tilde{U}$, so are the symplectic structures near the boundaries, so $U$ also has a concave contact boundary $(M_1,\xi_{LM})$. The curves $C_1^\prime,C_2^\prime$ are not necessarily the same as 
the original curves $\tilde{C}_1,\tilde{C}_2$, but they have the same self-intersections, the same local structures near their intersection point, and most importantly the same areas, i.e, 
$Area({C}_1^\prime)=50$ and $Area({C}_2^\prime)=23$, and $U$ is a regular neighborhood of 
$C_1^\prime\cup C_2^\prime$. Now one can build a smaller regular neighborhood of 
$C_1^\prime\cup C_2^\prime$ inside $U$, which can be identified with $V(\lambda_0,\epsilon_2)$, 
with $\lambda_0,\epsilon_2>0$ sufficiently small, where $V(\lambda_0,\epsilon_2)$ is the union of the relative handlebody $Z$ from Lemma 8.1 (but having very ``skinny" symplectic $2$-handles as $\epsilon_2>0$
is chosen to be small) with the $4$-ball $H^{-1}([0,\lambda_0])$ which is also small. Note that $V(\lambda_0,\epsilon_2)$ has a concave contact boundary $(M_1,\xi_{Mil})$, and $U\setminus V(\lambda_0,\epsilon_2)$ is a  topologically trivial symplectic cobordism from $(M_1,\xi_{LM})$ to $(M_1,\xi_{Mil})$. This finishes the proof of Lemma 8.4, and the proof of Theorem 1.10 is completed.

\section{Embeddings of other Stein surfaces and an infinite family of lens spaces}

In this section, we give a proof for Theorem 1.12 and Theorem 1.14. 

Recall that for $\delta\geq 4$, $M_\delta$ denote the lens space $L(\delta^2,-\delta-1)$.

\begin{lemma}
$M_\delta=M((\delta-3,-(\delta-2)), (4\delta-3,3\delta-2))$ as a Seifert fibered space. 
\end{lemma}

\begin{proof}
We recall the fact that the Seifert manifold $M((\alpha_1,\beta_1), (\alpha_2,\beta_2))$
is diffeomorphic to the lans space $L(a,b)$, where $a=\alpha_1\beta_2+\beta_1\alpha_2$ and $b=\alpha_1\beta_2^\prime+\beta_1\alpha_2^\prime$, for any $\alpha_2^\prime$, $\beta_2^\prime$ obeying $\alpha_2\beta_2^\prime-\beta_2\alpha_2^\prime=1$ (cf. \cite{N}). (Here if $a<0$, then $L(a,b)=L(-a,-b)$.) It follows easily that $M((\delta-3,-(\delta-2)), (4\delta-3,3\delta-2))=L(a,b)$, where 
$$
a=(\delta-3)(3\delta-2)-(\delta-2)(4\delta-3)=-\delta^2, \; b=(\delta-3)\cdot (-3)-(\delta-2)\cdot (-4)=\delta+1.
$$
(We choose $(\alpha_2^\prime,\beta_2^\prime)=(-4,-3)$ in the above.) The lemma follows immediately.

\end{proof}

\begin{lemma}
There is a relative handlebody $Z_\delta$ of two symplectic $2$-handles built on $M_\delta$,
which is symplectic with two concave contact boundaries $(M_\delta,\xi_{can})$ and
$(M_\delta^\prime,\xi_{Mil})$. Moreover, $c_1(K_{Z_\delta})\cdot [\omega_{Z_\delta}]<0$,
where $\omega_{Z_\delta}$ is the symplectic structure on $Z_\delta$.
\end{lemma}

\begin{proof}
We fix the Seifert fibration on $M_\delta$ from Lemma 9.1, and construct a compatible rational open book with periodic monodromy by applying Theorem 3.4. Adapting the notations therein, we regard $M_\delta$ as the Seifert fibered space $M((\alpha_1,\beta_1), (\alpha_2,\beta_2), 
(\alpha_3,\beta_3))$, where $(\alpha_1,\beta_1)=(1,0)$, $(\alpha_2,\beta_2)=(\delta-3, -(\delta-2))$, and $(\alpha_3,\beta_3)=(4\delta-3, 3\delta-2)$. With this understood, the index $i=1,2$, and index $j=3$.

Next, we proceed by letting $n=\alpha_3=4\delta-3$, $c_1=3\delta-2$, $c_2=2\delta-2$, and $b_1=b_2=1$. It is easy to check that $gcd(n,c_i)=1$ for $i=1,2$, and
$b_1+b_2=\frac{\beta_3}{\alpha_3}+\sum_{i=1}^2 \frac{c_i}{n}$ holds true. With this understood, we note that
$$
\frac{\beta_1}{\alpha_1}+b_1-\frac{c_1}{n}=\frac{\delta-1}{4\delta-3}>0,\;\;\; 
\frac{\beta_2}{\alpha_2}+b_2-\frac{c_2}{n}=-\frac{2\delta^2-4\delta+3}{(\delta-3)(4\delta-3)}<0.
$$
Thus by Theorem 3.4, there is a compatible rational open book $(B,\pi)$ on $M_\delta$ with
a periodic monodromy of order $n=4\delta-3$, such that the binding $B$ has two components 
$B_1$, $B_2$, where $B_1$ is a regular fiber of the Seifert fibration on $M_\delta$, and $B_2$ is the singular fiber of multiplicity $\delta-3$ but given with the opposite orientation. Finally, we note that the multiplicities 
$p_1$, $p_2$ of $(B,\pi)$ at $B_1,B_2$ are given by (cf. Remark 3.5)
$$
p_1=\delta-1,\;\; p_2=2\delta^2-4\delta+3.
$$

With the preceding understood, let $\xi$ be the $\s^1$-invariant contact structure on $M_\delta$
which is supported by the rational open book $(B,\pi)$. Then by Theorem 4.1, $\xi$ is non-transverse 
as $B_2$ has the opposite orientation, and moreover, the dividing set of $\xi$ consists of a circle surrounding the singular point of multiplicity $\delta-3$ in the base $2$-orbifold of the Seifert fibration on $M_\delta$. We note that since $e(M_\delta)>0$ in Lemma 9.1, it follows from Theorem 2.3(1) that
$\xi$ is a universally tight contact structure on the lens space $M_\delta$, i.e., $\xi=\xi_{can}$.

To proceed further, let $(\bar{\alpha}_1,\bar{\beta}_1)=(2,1)$ and $(\bar{\alpha}_2,\bar{\beta}_2)=(6\delta-17,6\delta-11)$. Since $B_2$ has opposite orientation, we shall change 
$(\alpha_2,\beta_2)$ to its negative $(-\alpha_2,-\beta_2)=(-\delta+3, \delta-2)$ in the application of
Theorem 4.4. With this understood, note that $\alpha_i\bar{\beta}_i+\bar{\alpha}_i\beta_i=1$ for $i=1,2$. If we attach symplectic $2$-handles $H_1,H_2$ along $B_1,B_2$, with framings 
$F_i=\frac{1}{\alpha_i}(\frac{n}{p_i}-\bar{\alpha}_i)$ (for $i=1,2$) relative to the page framing, 
where $F_1=2+\frac{1}{\delta-1}>0$, $F_2=6+\frac{2\delta-2}{2\delta^2-4\delta+3}>0$, we obtain a 
relative handlebody $Z_\delta$ with two concave contact boundaries, one of which is 
$(M_\delta,\xi_{can})$, while for the other one $(M^\prime,\xi^\prime)$, we use Theorem 4.4 to find that 
$$
M^\prime=M((2,1), (6\delta-17,6), (4\delta-3, \delta-1))=M_\delta^\prime, 
$$
and that $\xi^\prime$ is the Milnor fillable contact structure $\xi_{Mil}$ on $M_\delta^\prime$ as both 
$\bar{\alpha}_1, \bar{\alpha}_2>0$, so that $\xi^\prime$ must be transverse. Finally, a direct calculation with
Lemmas 5.1 and 5.2 shows that 
$c_1(K_{Z_\delta})\cdot [\omega_{Z_\delta}]=-\Omega (p_1+6p_2+1)<0$. 

\end{proof}

\begin{lemma}
$Z_\delta$ is simply connected. 
\end{lemma}

\begin{proof}
The key observation is that the binding component $B_2$ in the construction of $Z_\delta$, which is
the singular fiber of multiplicity $\delta-3$ in the Seifert fibered space 
$M_\delta=M((\delta-3, -(\delta-2)), (4\delta-3, 3\delta-2))$, is a generator of $\pi_1(M_\delta)$ which has order $\delta^2$. As $B_2$ is contractible in $Z_\delta$ because of the $2$-handle attached along it, we see immediately that $Z_\delta$ is simply connected.

To see that $B_2$ is a generator of $\pi_1(M_\delta)$, we begin with the observation that 
$r:=gcd(\delta-3,4\delta-3)$ only takes values $1,3$ or $9$, and $r>1$ if and only if 
$\delta=0\pmod{3}$, in which case, one also has $r=gcd(\delta-3,\delta^2)$. With this understood, we consider the case where $r=1$ first. In this case, the orbifold fundamental group of the base $2$-orbifold of $M_\delta=M((\delta-3, -(\delta-2)), (4\delta-3, 3\delta-2))$ is trivial, from which it follows 
that the fiber class must generate $\pi_1(M_\delta)$. For the case where $r>1$, the orbifold fundamental group of the base $2$-orbifold is nontrivial, which is cyclic of order $r$. It follows easily that the regular fiber $F$ has order $\delta^2/r$. With this understood, we note that $F=(\delta-3)B_2$, Writing $\delta-3=rs$ for some $s\in\Z$. Then $s$ is relatively prime with $\delta^2$ because of the fact that $r=gcd(\delta-3,\delta^2)$. Now let $s^\prime$ be given such that $ss^\prime=1\pmod{\delta^2}$. Then we have $rB_2=s^\prime F$ in $\pi_1(M_\delta)$. It follows easily that $ord(B_2)=r\cdot ord(F)=\delta^2$, which shows that $B_2$ is a generator of $\pi_1(M_\delta)$. 

\end{proof}

Let $W_\delta^\prime$ be any $\Q$-homology ball symplectic filling of $(M_\delta^\prime,\xi_{Mil})$.
We form the closed symplectic $4$-manifold $X_\delta:=W_\delta\cup Z_\delta \cup 
W_{\delta}^\prime$. Then since $c_1(K_{Z_\delta})\cdot [\omega_{Z_\delta}]<0$, $X_\delta$
is a symplectic Hirzebruch surface by Lemma 5.4. To finish off the proof of Theorem 1.12,
it remains to show that $X_\delta\cong \s^2\times \s^2$ if and only of $\delta\geq 5$ and is odd. 

\begin{lemma}
There is a class $\alpha\in H_2(X_\delta)$ such that $\alpha\cdot \alpha\equiv 2 \pmod{4}$
when  $\delta\geq 5$ and is odd, and $\alpha\cdot \alpha \equiv \pm 1 \pmod{4}$ when $\delta\geq 4$ and is even. 
\end{lemma}

\begin{proof}
Consider the rational open book $(B,\pi)$ on $M_\delta$ from the proof of Lemma 9.2.
Let $\Sigma$ be a page of $(B,\pi)$. Note that the multiplicity at the binding components 
$B_1,B_2$ is $p_1=\delta-1$ and $p_2=2\delta^2-4\delta+3$ respectively. It follows that the union 
$\Sigma\cup p_1D_1\cup p_2 D_2$ forms a closed, oriented singular chain 
in $X_\delta$. Here $D_1,D_2$ denote the core disk of the $2$-handles $H_1,H_2$
attached along $B_1,B_2$. We let $\alpha\in H_2(X_\delta)$ be the homology class of the singular 
chain. 

To compute the self-intersection of $\alpha$, we fix a framing at $B_1,B_2$, given by some choice
of longitude-meridian pair $(\lambda_1,\mu_1)$ and $(\lambda_2,\mu_2)$ respectively. 
Continue to use the notations from Lemma 3.2 and Lemma 4.6, we have 
$\gamma_i=p_i\lambda_i+q_i\mu_i$, $\tau_i=p_i^\prime\lambda_i+q_i^\prime \mu_i$ for
some $q_i,q_i^\prime$ obeying $p_iq_i^\prime-p_i^\prime q_i=1$, for $i=1,2$. Note that
$\gamma_1,\gamma_2$ are the two boundary components of the page $\Sigma$. 
As we have shown in Lemma 4.6, the regular fibers of the Seifert fibration have slopes $s_1,s_2$
with respect to $(\lambda_1,\mu_1)$ and $(\lambda_2,\mu_2)$, where
$$
s_1=\frac{q_1}{p_1}+\frac{n}{\alpha_1 p_1}, \;\; s_2=\frac{q_2}{p_2}+\frac{n}{\alpha_2 p_2}.
$$
Consequently, the framings of $H_1,H_2$ relative to $(\lambda_1,\mu_1)$ and $(\lambda_2,\mu_2)$
are given by
$$
f_1=\frac{q_1}{p_1}+\frac{n}{\alpha_1 p_1}-\frac{\bar{\alpha}_1}{\alpha_1},\;\;
f_2=\frac{q_2}{p_2}+\frac{n}{\alpha_2 p_2}-\frac{\bar{\alpha}_2}{\alpha_2}.
$$
With this understood, the self-intersection of the homology class $\alpha$ is given by
$$
\alpha\cdot \alpha= \sum_{i=1}^2 p_i\gamma_i\cdot (\lambda_i+f_i\mu_i)=\sum_{i=1}^2
p_i (p_if_i-q_i)= \sum_{i=1}^2 p_i (\frac{n-\bar{\alpha}_i p_i}{\alpha_i}). 
$$
With $p_1=\delta-1$, $p_2=2\delta^2-4\delta+3$, and $n=4\delta-3$, $\alpha_1=1$, 
$\bar{\alpha}_1=2$, $\alpha_2=-\delta+3$, $\bar{\alpha}_2=6\delta-17$, we have
\begin{eqnarray*}
\alpha\cdot\alpha & = &(\delta-1)(4\delta-3-2(\delta-1))+ (2\delta^2-4\delta+3)\frac{4\delta-3-(6\delta-17)(2\delta^2-4\delta+3)}{-\delta+3}\\
& = & (\delta-1)(2\delta-1)+(2\delta^2-4\delta+3)(12\delta^2-22\delta+16).\\
\end{eqnarray*}
It follows easily that when $\delta$ is even, $\alpha\cdot\alpha\equiv \pm 1\pmod{4}$. If $\delta$ is odd 
and $\delta\equiv 1\pmod{4}$, we have $\alpha\cdot\alpha\equiv 2\pmod{4}$. If $\delta$ is odd and 
$\delta=-1\pmod{4}$, we have $\alpha\cdot\alpha\equiv 0 \pmod{4}$. 

So we have to deal with the case where $\delta$ is odd and $\delta=-1\pmod{4}$ separately. We shall
consider a different rational open book on $M_\delta$, with the same binding components
$B_1,B_2$, but having a different page $\Sigma$ and different multiplicities $p_1,p_2$
at $B_1,B_2$. To this end, in the application of Theorem 3.4, we instead let $n=8\delta-6$,
$c_1=c_2=5\delta-4$. Since $\delta$ is odd, we have $gcd(n,c_i)=1$ for $i=1,2$. We continue to
have $b_1=b_2=1$, and $b_1+b_2=\frac{\beta_3}{\alpha_3}+\sum_{i=1}^2 \frac{c_i}{n}$ holds true.
Furthermore,
$$
\frac{\beta_1}{\alpha_1}+b_1-\frac{c_1}{n}=\frac{3\delta-2}{8\delta-6}>0,\;\;\; 
\frac{\beta_2}{\alpha_2}+b_2-\frac{c_2}{n}=-\frac{5\delta^2-11\delta+6}{(\delta-3)(8\delta-6)}<0.
$$
So $B_1$ will have the fiber orientation, but $B_2$ has the opposite orientation. As a consequence,
we change $(\alpha_2,\beta_2)$ to $(-\alpha_2,-\beta_2)=(-\delta+3,\delta-2)$. Finally, the multiplicities 
at $B_1,B_2$ are $p_1=3\delta-2$ and $p_2=5\delta^2-11\delta+6$ respectively.

Let $\alpha$ be the homology class of $\Sigma\cup p_1D_1\cup p_2 D_2$. By the same formula, we calculate the self-intersection of $\alpha$, with $\delta=-1\pmod{4}$ in the last step:
\begin{eqnarray*}
\alpha\cdot\alpha &= & \sum_{i=1}^2 p_i (\frac{n-\bar{\alpha}_i p_i}{\alpha_i})\\
& = & (3\delta-2)(8\delta-6-2(3\delta-2))+(5\delta^2-11\delta+6)
\frac{8\delta-6-(6\delta-17)(5\delta^2-11\delta+6)}{-\delta+3}\\
& = & (3\delta-2)(2\delta-2)+(5\delta^2-11\delta+6)(30\delta^2-61\delta+32)\\
& \equiv & 0+ (1-1+2)\cdot (2+1+0) \equiv 2 \pmod{4}\\
\end{eqnarray*}

The proof of the lemma is complete, and Theorem 1.12 follows immediately. 

\end{proof}

We end this section with a proof of Theorem 1.14. First, by Theorem 1.17, 
$W_C$ is a symplectic filling of $(M_C,\xi_{inv})$ where $\xi_{inv}$ is an $\s^1$-invariant, non-transverse contact structure on $M_C=M((3,-1), (3,1), (2,-1))$. In fact, there are additional informations
concerning $(M_C,\xi_{inv})$ which are contained in the proof of Theorem 1.17. If we adapt the notations from Theorem 4.4, we actually have $M_C=M((-3,4), (3,1), (2,1))$, where the Seifert invariant $(-3,4)$ indicates that the corresponding fiber is the binding component of the rational open book supporting $\xi_{inv}$, and the binding component has the opposite orientation of the fiber. 
See Example 6.7(1), with $(p,q)=(2,3)$. 

\begin{lemma}
There is a simply connected relative handlebody $Z$ consisting of one symplectic $2$-handle, which has two concave contact boundaries $(M_C,\xi_{inv})$ and $(M_{C^\prime}, \xi_{inv}^\prime)$, where 
$M_C=M((3,-1), (3,1), (2,-1))$, $M_{C^\prime}=M((22,7), (3,2), (2,-1))$, and $\xi_{inv}^\prime$ is an $\s^1$-invariant, transverse contact structure on $M_{C^\prime}$. Moreover, $c_1(K_{Z})\cdot [\omega_{Z}]<0$ for the
symplectic structure $\omega_Z$ on $Z$.
\end{lemma}

\begin{proof}
The rational open book $(B,\pi)$ on $M_C$ which supports $\xi_{inv}$ can be easily recovered 
by Theorem 3.3, where the order of the periodic monodromy equals $n=6$, and the multiplicity
$p=9$ at the binding component $B$, which has Seifert invariant $(\alpha,\beta)=(-3,4)$.
With this understood, observe that $(\bar{\alpha},\bar{\beta})=(22, 29)$ obeys $\alpha\bar{\beta}
+\beta\bar{\alpha}=1$, with the corresponding framing $F=\frac{1}{\alpha}(\frac{n}{p}-\bar{\alpha})
=\frac{1}{-3}(\frac{6}{9}-22)=\frac{64}{9}>0$. Thus by Theorem 4.4, if we attach a symplectic
$2$-handle $H$ along $B$ with a framing $F=\frac{64}{9}$ relative to the page framing from the
rational open book $(B,\pi)$, we get a relative handlebody $Z$, which is symplectic with two
concave contact boundaries, one of which is $(M_C,\xi_{inv})$, while the other one 
$(M^\prime,\xi^\prime)$ can be determined using the formula in Theorem 4.4:
$$
M^\prime=M((22,29), (3,-1), (2,-1))=M((22,7), (3,2), (2,-1))=M_{C^\prime}, 
$$
and $\xi^\prime=\xi_{inv}^\prime$ is transverse because $\bar{\alpha}=22>0$. See Example 6.7(5). 

To determine the sign of $c_1(K_{Z})\cdot [\omega_{Z}]$, we can use the formula in Corollary 5.3,
where $s=1,r=2$, $n=6$, $p=9$, $\alpha_i=-3$, $\bar{\alpha}_i=22$, 
$\alpha_j=2,3$, and $\Omega>0$: 
$$
c_1(K_{Z})\cdot [\omega_{Z}]=\Omega \cdot 
(6(1+2-2-\frac{1}{-3}-\frac{1}{2}-\frac{1}{3})- 9(1-\frac{22}{-3}))=\Omega\cdot (-72)<0.
$$
(If we use Lemma 5.2 to compute $c_1(K_Z)\cdot [pD]$, we get $-72$ which is consistent.)

It remains to show that $Z$ is simply connected. The fundamental group of 
$M_C=M((-3,4), (2,1), (3,1))$ has the following presentation (cf. \cite{N}):
$$
\pi_1(M_C)=\{q_1,q_2,q_3,h| q_1^3 h^{-4}=q_2^2h=q_3^3 h=q_1q_2q_3=1\}.
$$
The attaching of a $2$-handle to $M_C$ along the singular fiber of Seifert invariant $(-3,4)$ means that there is a pair of integers $(\alpha^\prime,\beta^\prime)$, with $3\beta^\prime+4\alpha^\prime=1$, such that $q_1^{\alpha^\prime}h^{\beta^\prime}=1$ in $\pi_1(Z)$. It follows easily that $h=1$ in 
$\pi_1(Z)$, which in turn implies that $q_1=1$, $q_2^2=q_3^3=1$ in $\pi_1(Z)$. With $q_2q_3=1$, it follows that $q_2=q_3=1$ in $\pi_1(Z)$ as well. This shows that each of the generators $q_1,q_2,q_3,h$ of 
$\pi_1(M_C)$ is trivial in $\pi_1(Z)$, proving that $\pi_1(Z)=1$. This finishes off the proof. 

\end{proof}

Theorem 1.14 follows immediately from Lemma 9.5: $X:=W_C\cup Z\cup W_{C^\prime}$ is a
symplectic $\Q$-homology $\C\P^2$, which must be diffeomorphic to $\C\P^2$ by Taubes \cite{T},
as $c_1(K_X)\cdot [\omega_X]=c_1(K_{Z})\cdot [\omega_{Z}]<0$ by Lemma 5.4. 


\vspace{2mm}

{\Small University of Massachusetts, Amherst.\\
{\it E-mail:} wch@umass.edu

\end{document}